\documentclass[10pt]{aart}

\usepackage[titletoc,toc,title]{appendix}

\usepackage{pstricks,pst-node}
\usepackage{fancyhdr}
\usepackage{graphicx}
\usepackage{fancybox}
\usepackage{latexsym}
\usepackage[english]{babel}
\usepackage{t1enc}
\usepackage{bbm}
\usepackage{mathrsfs}
\usepackage{ifthen}
\usepackage{url}
\usepackage{epsfig}
\usepackage{amsbsy}
\usepackage{extarrows}
\usepackage{amsfonts}
\usepackage{amsmath}
\usepackage{amssymb}
\usepackage{amsthm}
\usepackage{amsxtra}
\usepackage{bbm}
\usepackage{color}
\usepackage{units}

\usepackage{enumitem}

\usepackage[top=2.8cm, bottom=2.7cm, left=3cm, right=3cm]{geometry}

\usepackage{verbatim}
\usepackage{stmaryrd}
\usepackage{hyperref}

\numberwithin{equation}{section}
\numberwithin{figure}{section}

\newcommand{\wt}{\widetilde}
\newcommand{\wh}{\widehat}

\newcommand{\beqa}{\begin{eqnarray}}
\newcommand{\eeqa}{\end{eqnarray}}
\newcommand{\e}{\varepsilon}

\newcommand{\pt}{\partial}
\newcommand{\rd}{{\rm d}}
\newcommand{\bR}{{\mathbb R}}

\newcommand{\tr}{\mbox{Tr\,}}

\newcommand{\bx}{{\bf{x}}}
\newcommand{\by}{{\bf{y}}}

\newcommand{\bv}{{\bf{v}}}
\newcommand{\bw}{{\bf{w}}}
\newcommand{\bz}{{\bf {z}}}

\newcommand{\bX}{{\bf{X}}}
\newcommand{\bY}{{\bf{Y}}}

\newcommand{\fq}{{\frak q}}

\newcommand{\bla}{\mbox{\boldmath $\lambda$}}

\newcommand{\al}{\alpha}

\newcommand{\be}{\begin{equation}}
\newcommand{\ee}{\end{equation}}

\newcommand{\la}{\lambda}

\newcommand{\om}{{\omega}}

\newcommand{\cU}{{\mathscr U}}
\newcommand{\cB}{{\mathscr B}}
\newcommand{\cL}{{\mathscr L}}

\newcommand{\cG}{{\mathcal G}}

\newcommand{\cK}{{\mathcal K}}
\newcommand{\cT}{{\mathcal T}}

\newcommand{\cM}{{\mathcal M}}

\newcommand{\cH}{{\mathcal H}}

\newcommand{\re}{{\mathrm{Re} \, }}
\newcommand{\im}{{\mathrm{ Im }\, }}
\newcommand{\E}{{\mathbb E }}
\newcommand{\R}{{\mathbb R }}
\newcommand{\N}{{\mathbb N}}

\newcommand{\g}{\gamma}

\newcommand{\Z}{{\mathbb Z}}
\renewcommand{\P}{{\mathbb P}}

\newcommand{\C}{{\mathbb C}}

\newcommand{\euler}[1]{\mathrm{e}^{#1}}
\newcommand{\ii}{\mathrm{i}}
\newcommand{\dd}{\mathrm{d}}
\newcommand{\Hi}[1]{\left(\mathrm{T}{#1}\right)}
\newcommand{\supp}{\mathrm{supp}\,}
\newcommand{\ie}{\emph{i.e., }}
\newcommand{\eg}{\emph{e.g., }}
\newcommand{\cf}{\emph{c.f., }}
\newcommand{\Gauss}{\textsc{G}}
\newcommand{\aux}{\textsc{aux}}

\newcommand{\regu}{{\eta_*}}
\newcommand{\deq}{\mathrel{\mathop:}=}

\theoremstyle{plain} 
\newtheorem{theorem}{Theorem}[section]
\newtheorem*{theorem*}{Theorem}
\newtheorem{lemma}[theorem]{Lemma}
\newtheorem{corollary}[theorem]{Corollary}
\newtheorem*{corollary*}{Corollary}
\newtheorem{proposition}[theorem]{Proposition}
\newtheorem*{proposition*}{Proposition}
\newtheorem{definition}[theorem]{Definition}
\newtheorem*{definition*}{Definition}

\theoremstyle{definition} 

\newtheorem*{example*}{Example}
\newtheorem{remark}[theorem]{Remark}

\newtheorem*{remark*}{Remark}
\newtheorem*{remarks*}{Remarks}

\let\OLDthebibliography\thebibliography
\renewcommand\thebibliography[1]{
  \OLDthebibliography{#1}
  \setlength{\parskip}{0pt}
  \setlength{\itemsep}{0pt plus 0.3ex}
}

\newcommand{\msp}[1] {\mspace{#1 mu}}
\newcommand{\digammao}{\overset{\msp{5}\circ}{\digamma}\overset{(N)}{\phantom{,}}}
\newcommand{\digammac}{{\digamma^{(N)}}}
\newcommand{\ol}[1]{\overline{#1}}
\newcommand{\cy}{y_{\circ}}
\newcommand{\cz}{z_{\circ}}
\newcommand{\IO}{I_0}
\newcommand{\II}{I_\sigma}

\begin{document}

 \begin{minipage}{0.85\textwidth}
 \vspace{3.5cm}
 \end{minipage}
\begin{center}
\large\bf
Universality for Random Matrix Flows with Time-dependent Density
\end{center}

\renewcommand{\thefootnote}{\fnsymbol{footnote}}	
\vspace{1cm}
\begin{center}
 \begin{minipage}{0.5\textwidth}
\begin{center}
L\'aszl\'o Erd{\H o}s\footnotemark[1]  \\
\footnotesize {IST Austria}\\
{\it lerdos@ist.ac.at}
\end{center}
\end{minipage}
\begin{minipage}{0.45\textwidth}
 \begin{center}
Kevin Schnelli\footnotemark[2]\\
\footnotesize 
{IST Austria}\\
{\it kevin.schnelli@ist.ac.at}
\end{center}
\end{minipage}
\footnotetext[1]{Partially supported by ERC Advanced Grant RANMAT No.\ 338804.}
\footnotetext[2]{Supported by ERC Advanced Grant RANMAT No.\ 338804.}

\renewcommand{\thefootnote}{\fnsymbol{footnote}}	

\end{center}
\vspace{1cm}

\begin{center}
 \begin{minipage}{0.9\textwidth}
We show that the Dyson Brownian Motion exhibits local universality
after a very short time assuming that local rigidity and level repulsion hold.
These conditions are verified, hence bulk spectral universality is proven,
 for a large class of Wigner-like matrices, 
including deformed Wigner ensembles and ensembles with non-stochastic variance
matrices whose limiting densities differ from the Wigner semicircle law.

\end{minipage}
\end{center}
 
 \vspace{2mm}
 
 {\small

 
 \vspace{2mm}
 
 }

\thispagestyle{headings}

\section{Introduction and motivation}\label{sec:introduction}

In his groundbreaking paper~\cite{W}, Wigner conjectured that
the eigenvalue gap distribution of large random matrices is universal
and that it serves as a ubiquitous model for the local  spectral statistics
of many quantum systems. The Gaussian case was fully understood
in the subsequent works of Dyson, Gaudin and Mehta; see~\cite{M} 
for a summary. This simplest case can be generalized in two
directions. For invariant ensembles, the joint density function
of the eigenvalues can be explicitly expressed in terms of
a Vandermonde determinant; a formula that can also be interpreted
as the Gibbs measure of a gas of one-dimensional particles with
a logarithmic interaction. For specific values of the inverse temperature $\beta=1,2,4$,
the correlation functions may be expressed and analyzed using asymptotics
of orthogonal polynomials~\cite{FIK} and universality was proved under 
various conditions on the potential in~\cite{De1, DG1, PS:97, PS}, with many 
consecutive works following. 
This method, however,  is not applicable for other values of~$\beta$ even in the Gaussian
case, where the correlation functions were described in~\cite{VV}.
Universality for general $\beta$-ensembles was first established recently in~\cite{BEY, BEY2} for $\beta\ge 1$, 
with different proofs given later in~\cite{BFG, SM} that also hold for~$\beta>0$
albeit with more restrictions on the potential.

Among the non-invariant ensembles, the most prominent case is the~$N\times N$ symmetric or hermitian Wigner matrix
characterized by the independence of the entries (up to the constraint imposed by the symmetry class).
Beyond the Gaussian case there is no explicit formula for the eigenvalue distribution in general,
but in the hermitian case ($\beta=2$) and for distributions with a Gaussian component, 
the correlation functions can still be expressed using  an algebraic identity (Harish-Chandra-Itzykson-Zuber
integral). A rigorous analysis of this approach yielded universality for hermitian Wigner matrices
with a substantial Gaussian component,~\cite{J, BP}. The first proof of hermitian Wigner universality for 
an arbitrary smooth distribution was given in~\cite{EPRSY}, the smoothness condition was
later removed in~\cite{TV, ERSTVY}. Lacking the algebraic identity, the symmetric case ($\beta=1$) 
required a completely different approach based on the analysis of the Dyson Brownian motion.
The basic observation of Dyson~\cite{DyB} was that the eigenvalues of a  matrix ensemble, embedded
in a simple stochastic flow ({\it Dyson matrix flow}),
 evolve autonomously and satisfy a system of $N$ stochastic differential
equations, called the {\it Dyson Brownian Motion (DBM)}. The universal eigenvalue statistics emerge in the bulk spectrum
as a consequence of the invariant measure of the DBM. Local statistics require to understand only 
the local equilibration mechanism which occurs on a very small time scale that can be bridged 
by perturbative methods. The rigorous theory of this idea was initiated in~\cite{ESY3}
and developed in a series of papers~\cite{ESYY, EYYrigi} leading to the complete proof of
the Wigner-Dyson-Mehta universality conjecture for Wigner matrices in
all symmetry classes; see~\cite{EYBull} for a summary. More recently 
two stronger versions of the bulk universality have been proved. In contrast to the
previous results that required a  local averaging, the universality of each 
single gap was shown to be universal 
in~\cite{EYSG}, while the universality of correlation
functions at each fixed energy was obtained in~\cite{FE}. 
These papers heavily relied on a new tool from~\cite{EYSG}, the concept of H\"older regularity
theory for the parabolic equation with random coefficients given by the~DBM.

In all these  works  on the spectral universality for Wigner matrices, the global limiting
density was the semicircle law; in particular it did not change in time under the DBM.
The same Gaussian measure and its localized versions could be used
as equilibrium reference measures for all times. The main idea was to 
artificially speed up the global convergence by considering 
the {\it local relaxation flow}~\cite{ESY3} and then to prove that 
the  additional local relaxation terms do not substantially
modify the local statistics thanks to a-priori bounds on the location of the particles. These bounds are called  
{\it rigidity estimates} and they directly follow from short scale versions of the Wigner semicircle law 
that are called {\it local~laws}. 

The method of local relaxation flow has two main limitations that are related. 
First, it operates with  global measures, in particular, quite
precise rigidity information is needed for {\it all} eigenvalues.
This is clearly unnecessary (and in some cases hard to obtain); far away eigenvalues should not influence
local statistics too much. 
Second, if the initial matrix of the Dyson matrix flow does not
obey the semicircle law, then the density changes with time following the {\it semicircular flow}, 
related to the complex Burgers equation for its Stieltjes transform. The time dependence of the density
 was  originally not incorporated in the method of the local relaxation flow. 
 This second limitation was tackled  very recently in~\cite{LSSY}, where universality for {\it deformed} Wigner matrices with large diagonal elements was
proved. Using ideas from hydrodynamic limits~\cite{Y}, a {\it global} reference measure
was constructed as an invariant $\beta$-ensemble with a ``time'' parameter so that
its equilibrium density trails the semicircular flow. This equilibrium measure was then used 
as a basis to construct the local relaxation flow. Once the fast convergence to
the reference measure is established one can infer to the universality of the $\beta$-ensemble~\cite{BEY, BEY2},
or, alternatively, one can use the uniqueness of the local Gibbs measure established
in~\cite{EYSG} to conclude universality with a tiny Gaussian component. This result is then
easily complemented by a standard Green function comparison method to remove the Gaussian component entirely.

As a technical input for the analysis in~\cite{LSSY}, 
the global rigidity for the reference $\beta$-ensemble is required, which is not available
for the case when the equilibrium density is supported on several intervals.  In particular,
the result of~\cite{LSSY} is limited to the deformed Wigner ensembles with a single
interval support that excludes the case when the diagonal has a strongly bimodal distribution.

We remark that bulk universality for special classes of deformed Wigner matrices in the hermitian symmetry
class has also been
proven with different methods. The local sine kernel statistics
 for the sum of a GUE and a diagonal matrix with two eigenvalues $\pm a$
of equal multiplicity has been obtained with Riemann-Hilbert method~\cite{BK1,BK2,CW}.
In particular, the density in this model is supported on two disjoint intervals
if $a$ is sufficiently large. The GUE matrix
can be replaced with an arbitrary Hermitian matrix if the first four moments of its single entry distribution
matches those of the Gaussian~\cite{OV}. 
A much more general class of deformations of the GUE has been tackled in~\cite{S1}
relying on a version of  the Harish-Chandra-Itzykson-Zuber integral. Using Green function comparison techniques~\cite{EYY}
and the local laws from~\cite{LS,KY1,KY2}, one can replace the GUE with any hermitian Wigner matrix under the four moment
matching condition.

Random matrices whose limiting densities are supported on several intervals arise in 
other prominent contexts as well. We call symmetric or hermitian matrix ensembles, $H= (h_{ij})$, 
{\it Wigner-like}
if their entries are independent (up to the symmetry constraint). If, in addition, the
matrix elements are centered, $\E h_{ij}=0$,
and the sum of the variances $S_{ij} =\E |h_{ij}|^2$ in each row is constant, say one, \ie
\begin{align}\label{zerorowsum}
  \sum_{j=1}^N S_{ij} = \mbox{const} =1\,, \quad\qquad \forall i\,,
\end{align}
 then the limiting density is the semicircle law. If either condition is violated,
the limiting density is generally not the semicircle law and typically it may be supported on several intervals.
The case of~$H=W+A$, where~$W$ is a standard Wigner matrix
with i.i.d.\ centered entries and~$A$ is a deterministic matrix (representing the 
nonzero expectations $\E h_{ij}$), was considered in~\cite{KY1, KY2}, where local
laws and rigidity were established. 
If condition~\eqref{zerorowsum} is dropped, then an even richer class of possible
limiting densities arise. These were extensively analyzed in~\cite{AEK1, AEK2}, where
all possible density shapes are classified, local laws and rigidity are  proven.

In the current paper, we prove bulk universality  for all these models. As in the previous
papers using DBM, the key part is  to show
universality for matrices with a tiny Gaussian component.

Beyond these applications, our main result is formulated on a more conceptual level.
Dyson argued in~\cite{DyB} that the local equilibrium of the DBM is attained after 
a very short time {\it irrespective of the global density.} In fact, the global
density equilibrates on a time scale of order one, while the local equilibration time is of order $1/N$.
The local equilibration is solely due to the logarithmic interaction in the DBM, while
the evolution of the global density is given by the semicircular flow. In this paper we
fully decouple the effects of these two processes. In the main Theorem~\ref{main theorem} we 
 prove bulk local universality 
for the DBM assuming that it satisfies rigidity and level repulsion, but {\it only locally}.
On the global scale only a very weak version of rigidity is required, in particular
the condition is insensitive to outliers or to the behavior at the edges. 
These assumptions can then easily be verified from local laws in each model.

After completing this manuscript, we learned that
similar results were obtained independently in~\cite{LY}.

{\it Notational conventions:}
We use the symbol $O(\,\cdot\,)$ and $o(\,\cdot\,)$ for the standard big-O and little-o notation. The notations $O$, $o$, $\ll$, $\gg$, refer to the limit $N\to \infty$. Here $a\ll b$ means $|a|\le N^{-\xi}|b|$, for some small $\xi>0$. We use~$c$ and~$C$, $C'$ to denote positive constants that do not depend on~$N$. Sometimes we use subscripts or superscripts to distinguish $N$-independent constants, \eg $c_0,c_1,c'$ {\it etc}. Their value may change from line to line.  Similarly, we will use $\xi>0$ for a small, respectively $D>0$ for a large positive exponent, mainly appearing in various rigidity bounds.
Their precise values are immaterial; at the end of the proof it may be chosen sufficiently small, respectively sufficiently large, depending on 
all other exponents along the proof. Finally, we use double brackets to denote index sets, \ie for $n_1, n_2 \in \R$,
$$
\llbracket n_1, n_2 \rrbracket \deq [n_1, n_2] \cap \Z\,,\qquad\qquad  \N_N\deq\llbracket 1,N\rrbracket\,.
$$

{\it Acknowledgement:}
The authors thank O.\ Ajanki and T.\ Kr\"uger for many valuable discussions at the early stage of the project.

\section{Main results}

In this section, we give a detailed description of our model, including all assumptions, and state our main results. We start with introducing basic concepts such as the Stieltjes transform, the semicircular flow and the Dyson Brownian motion (DBM).

\subsection{Stieltjes transform}
Given a probability measure, $\nu$, on $\R$, define its Stieltjes transform,~$m_\nu$,~by

\begin{align}\label{le stieltjes transfrom}
 m_{\nu}(z)\deq\int_\R\frac{\dd\nu(v)}{v-z}\,,\qquad\qquad  z\in\C^+\deq\{ z\in\C\,, \im z>0\}\,.
\end{align}
Note that $m_\nu$ is an analytic function in upper half plane. In the following we usually write $z=E+\ii\eta$, $E\in\R$, $\eta>0$, and we refer to $E$ as an ``energy'' and to $z$ as the spectral parameter. For given $\eta> 0$, we let $P_\eta$ denote the Poisson kernel defined by
\begin{align}\label{le poisson kernel}
 P_\eta(E)\deq \frac{1}{\pi}\frac{\eta}{E^2+\eta^2}\,,\qquad\qquad E\in\R\,,
\end{align}
and we note that $ \int_{\R}P_{\eta}(E)\dd E=1$ and $P_{\eta_1+\eta_2}(E)=(P_{\eta_1}*P_{\eta_2})(E)$, for all $\eta,\eta_1,\eta_2>0$, $E\in\R$, where~$*$ denotes the convolution on $\R$. We further remark that
\begin{align}\label{le 1eleven}
\frac{1}{\pi} \im m_\nu(E+\ii\eta)= (P_{\eta}*\nu)(E)\,.
\end{align}
Assuming that $\nu$ admits a density, which we also denote by $\nu$, we can recover $\nu$ from~$m_\nu$ through the Stieltjes inversion formula
\begin{align}\label{le stieltjes inversion}
 \nu(E)=\frac{1}{\pi}\lim_{\eta\searrow 0}\im m_{\nu}(E+\ii\eta)=\lim_{\eta\searrow 0}(P_{\eta}*\nu)(E)\,,\qquad\qquad E\in\R\,.
\end{align}
The Hilbert transform, $\Hi{\nu}$, of $\nu$ is defined by as the principal value integral
 \begin{align}\label{le def of hilbert transform}
      \Hi{\nu}(E) \deq \int_\R \frac{\dd \nu(v)}{v-E}\,,\qquad \qquad E\in\R \,.
 \end{align}
  
\subsection{Semicircular flow}\label{le subsection semicicular flow} We next introduce the semicircular or classical flow. Let $\cM(\R)$ denote the set of probability measures on $\R$. Then the semicircular flow is the process $\R^+\times\cM(\R)\rightarrow \cM(\R)$, $(t,\varrho)\mapsto \mathcal{F}_t[\varrho]$ obtained via its Stieltjes transform as follows. For $t=0$, set $\mathcal{F}_0[\varrho]\deq \varrho$. For $t>0$, let $ m_t(z)$ satisfy
\begin{align}\label{free convolution first}
  m_t(z)=\int_\R\frac{\dd\varrho(y)}{\euler{-t/2} y-z-(1-\euler{-t}) m_t(z)}\,,\qquad\qquad \im  m_t(z)>0\,,\qquad z\in\C^+\,.
\end{align}
It is straightforward to check~\cite{Pastur} that~\eqref{free convolution first} has indeed a unique solution such that $\liminf_{\eta\searrow 0}\im  m_t(E+\ii \eta)<\infty$, for any $E\in\R$, $t>0$. In fact, for $t>0$, $ m_t$ has a continuous extension to $\C^+\cup\R$~\cite{B} that we also denote by~$ m_t$. Set then
\begin{align}\label{le varrhot in inversion}
\mathcal{F}_t[\varrho](E)\deq\frac{1}{\pi}\lim_{\eta\searrow 0}\im m_t(E+\ii\eta)\,,\qquad \qquad t>0\,,\qquad E\in\R\,,
\end{align}
so that $\mathcal{F}_t[\varrho]$ is defined through its density $\mathcal{F}_t[\varrho](E)$, $E\in\R$. In particular, for $t>0$, $\mathcal{F}_t[\varrho]$ is an absolutely continuous measure. (For simplicity we use the same symbol for absolutely continuous measures and their densities.)

Further, it is easy to check that $ m_t(z)$ converges pointwise to 
\begin{align}\label{le m0}
  m_0(z)=\int_\R\frac{\dd\varrho_0(y)}{y-z}\,,
\end{align}
for all $z\in\C^+$, as $t\searrow 0$. It follows that $\mathcal{F}_t[\varrho]$ converges weakly to $\varrho$ as $t\searrow 0$. Starting from~\eqref{free convolution first} and~\eqref{le varrhot in inversion}, one also checks that
\begin{align*}
 \mathcal{F}_{t+s}[\varrho]=\mathcal{F}_t\circ\mathcal{F}_s[\varrho]\equiv \mathcal{F}_t[\mathcal{F}_s[\varrho]]\,,\qquad\qquad t\ge s\,,\qquad \varrho\in\mathcal{M}(\R)\,.
\end{align*}
In fact, using the additive free convolution, the flow $t\mapsto \mathcal{F}_t$ can be endowed with a ($w^*$-continuous) semigroup structure~\cite{Voi86,Maa,NiSp}; see also~\cite{VDN,HP} for reviews. Yet, we will not pursue this point of view in the present paper. 

In the following, we often write $\varrho_t\deq \mathcal{F}_t[\varrho]$ with $\varrho_0=\varrho$ and we call $t\mapsto \varrho_t$ the {\it semicircular flow} started at $\varrho$. Recalling~\eqref{le stieltjes transfrom} it is clear that $ m_t$ is the Stieltjes transform of $\varrho_t$ and we simply write $ m_t\equiv m_{\varrho_t}$. We remark that the standard semicircle law,~$\varrho_{sc}$, is invariant under the semicircular flow, \ie $\mathcal{F}_t[\varrho_{sc}]=\varrho_{sc}$, for all $t\ge 0$, and that~$\varrho_t=\mathcal{F}_t[\varrho]$ converges weakly to~$\varrho_{sc}$, as $t\nearrow\infty$, for any~$\varrho\in\mathcal{M}(\R)$. This follows directly from~\eqref{free convolution first} and the fact that the Stieltjes transform, $m_{\varrho_{sc}}\equiv m_{sc}$, of~$\varrho_{sc}$ satisfies $m_{sc}(z)=-(m_{sc}(z)+z)^{-1}$, $z\in\C^+$.

For $N\in\N$ and fixed $t\ge 0$, let $\boldsymbol{\gamma}(t)\equiv (\gamma_k(t))$ denote the set of~$N$-quantiles with respect the density~$\varrho_t$, where~$\gamma_k(t)$ is the smallest number satisfying
\begin{align}\label{the gammas}
 \int_{-\infty}^{\gamma_k(t)}\varrho_t(x)\,\dd x=\frac{k}{N}\,,\qquad\qquad \varrho_t=\mathcal{F}_t[\varrho]\,,
\end{align}
for all $t\ge0$. It is straightforward to check that $\gamma_k(t)$ inside the ``bulk'', \ie where $\varrho_{t}$ is strictly positive, is a continuous function of $t$. This follows from the (weak) continuity of the flow  $t\mapsto\varrho_t$. Moreover, the points $\boldsymbol{\gamma}(t)$ in the bulk approximately 
satisfy a gradient flow of a classical  particle system
with a logarithmic two-body interaction potential between the particles (see Lemma~\ref{le lemma for new gamma} below). We refer to Appendix~\ref{section classical semicircle flow} for a more detailed discussion.

\subsection{Dyson Brownian motion}
Fix $N\in\N$ and let $\digammao\subset\R^N$ denote the set
\begin{align}\label{Weyl chamber}
\digammao\deq\{ \bla = (\la_1,\la_2 ,\ldots , \la_N )\subset \R^N \,:\, \la_1< \la_2<\ldots< \la_N\}\,,
\end{align}
and denote its closure by $\digammac$.

Dyson Brownian motion (DBM) is given by the following stochastic differential equation~(SDE)
\begin{align}\label{le original dbm}
 \dd \lambda_i(t)=\sqrt{\frac{2}{\beta N}} \,\rd B_i(t) - \frac{1}{N} \sum_{{ j\not=i}} \frac{1}{\lambda_j(t)-\lambda_i(t)} \,\rd t -\frac{\lambda_i(t)}{2}\,\rd t\,,\qquad i\in \N_N\,,\qquad \beta\ge 1\,,
\end{align}
with fixed initial condition $\bla(t=0)\in\digammac$, where $\beta\ge 1$ is a fixed parameter with the interpretation of inverse temperature, and where $(B_i)_{i=1}^N$ are a collection of independent standard Brownian motions in some probability space $(\Omega,\P)$. We denote by $\E$ the expectation with respect to $\P$.

It is well known, see Section~4.3.1 of~\cite{AGZ}, that~\eqref{le original dbm} with $\beta\ge 1$ has a unique strong solution, $\bla(t)$, for any initial condition $\bla(0)\in\digammac$. Further, for any $t>0$, we have $\bla(t)\in\digammao$ almost surely.

The equilibrium measure for the DBM is the Gaussian invariant ensemble explicitly given by
\begin{align}\label{le gaussian general measure}
 \mu_\Gauss(\bla)\,\dd\bla\equiv\mu^{(N)}_{\beta,\Gauss}(\bla)\,\dd\bla=\frac{1}{Z^{(N)}_{\beta,\Gauss}}\euler{-\beta N \cH_\Gauss}\,\rd\bla\,,\qquad \cH_\Gauss \deq\sum_{i=1}^N\frac{1}{4}\la_i^2-\frac{1}{N}\sum_{1\le i<j\le N}\log (\la_j-\la_i)\,,
\end{align}
where $\rd\bla\deq {\bf 1}(\bla\in\digamma^{(N)})\,\rd \la_1\rd \la_2\cdots\rd \la_N$ and where $Z^{(N)}_{\beta,\Gauss}$ is a normalization. For fixed $\beta$, we denote by $\E^{\Gauss}$ the expectation with respect the measure $\mu_\Gauss$ in~\eqref{le gaussian general measure}.

Consider next a sequence of vectors $\bla^{(N)}(0)=(\la^{(N)}_1(0),\ldots,\la_N^{(N)}(0))\in\digammac$, $N\in\N$. Let $\bla^{(N)}(t)=(\la^{(N)}_1(t),\ldots,\la_N^{(N)}(t))\in\digammac$ denote the sequence of vectors such that, for each $N\in\N$, $\bla^{(N)}(t)\in \digammac$ is the solution to~\eqref{le original dbm} with initial condition $\bla^{(N)}(0)$. For simplicity we abbreviate $\bla(t)\equiv\bla^{(N)}(t)$, respectively $\la_i(t)=\la_i^{(N)}(t)$, $i\in\N_N$, in the following. 

Assume that there is a probability measure, $\varrho_0^{\infty}$, on $\R$ such that
\begin{align*}
 \varrho^{(N)}_0\deq\frac{1}{N}\sum_{i=1}^N \delta_{\lambda_i(0)}\overset{w}{\longrightarrow} \varrho_0^{\infty}\,,
\end{align*}
as $N\to\infty$, \ie the empirical distribution of the initial data $\bla^{(N)}(0)$ converges weakly to $\varrho_0^{\infty}$. Then, under some mild technical assumptions on $\bla^{(N)}(0)$, Proposition~4.3.10 of~\cite{AGZ} states that 
\begin{align}\label{le GG limit}
 \varrho^{(N)}_t\deq\frac{1}{N}\sum_{i=1}^N\delta_{\lambda_i(t)}\overset{w}{\longrightarrow}\varrho_t^{\infty}\,,\qquad\qquad t\ge 0\,,
\end{align}
as $N\to\infty$, where $\varrho_t^{\infty} \deq\mathcal{F}_t[\varrho_0^{\infty}]$ denotes the semicircular flow started from $\varrho_0^{\infty}$, 
\cf Subsection~\ref{le subsection semicicular flow}.

\newcommand{\rl}{\ell}

\subsection{Main result}
In this subsection, we state our main result. We need one more definition:
A labeling $\rl$ is a random variable $\rl\,:\, \R \to \Z$, $x\mapsto \rl(x)$ such that $\rl(x+1)-\rl(x)=1$ and $\rl(x)=\rl(\lfloor x\rfloor)$.

\begin{theorem}\label{main theorem}
Let $\bla(t)$, $t\ge 0$, be the solution to the DBM in~\eqref{le original dbm} with deterministic initial condition~$\bla(0)$. Given any small positive $\epsilon>0$ and any small $\delta\in[0,1/20]$, with $\epsilon\ge 2\delta$, consider times $t_1,t_2\in\R^+$ with $N^{-1+\epsilon}\le t_2-t_1\le N^{-\epsilon}$. Let~$\varrho$ be a probability measure on $\R$. Denote by $\varrho_t\equiv \mathcal{F}_{t}[\varrho]$ the semicircular flow started from $\varrho$. Choose $E_*\in\R$ such that
$\varrho_{t_1}(E_*)>c$, for some small $c>0$. 

 Assume that $\bla(t)$ and $\varrho$ are such that the following conditions are satisfied.
\begin{itemize}[noitemsep,topsep=0pt,partopsep=0pt,parsep=0pt]
 \item[$(1)$] At time $t_1$, the density $\varrho_{t_1}\equiv \mathcal{F}_{t_1}[\varrho]$ is regular in the following sense. There is a constant $\Sigma>0$, independent of $N$, such that the Stieltjes transform $m_{\varrho_{t_1}}$ of $\varrho_{t_1}$, \ie
 \begin{align}
  m_{\varrho_{t_1}}(z)=\int_{\R}\frac{\varrho_{t_1}(y)\,\dd y}{y-z}\,,\qquad\qquad z\in\C^+\,,
 \end{align}
 extends to a continuous function on $\mathcal{D}_\Sigma\deq\{ z=E+\ii\eta \in \C\,:\,  E\in [E_*-\Sigma,E_*+\Sigma]\,, \eta \ge 0\}$, and satisfies
 \begin{align}\label{le bounds in assumption 1}
  |m_{\varrho_{t_1}}(z)|\le C\,,\qquad \qquad|\partial_z^n m_{\varrho_{t_1}}(z)|\le C (N^{\delta})^n\,, \qquad \quad n=1,2\,,
 \end{align}
uniformly on $\mathcal{D}_\Sigma$, for some constant $C$. Moreover, $\varrho_{t_1}$ has finite second moment and satisfies
\begin{align}\label{le positivity of the density}
 \varrho_{t_1}(E)\ge c\,,\qquad \qquad E\in [E_*-\Sigma,E_*+\Sigma]\,,
\end{align}
for some $c>0$.

 \item[$(2)$] The process $\bla(t)$ is \emph{rigid} and is related to $\varrho_t$
 in the sense that there is a small $\sigma\equiv\sigma(\Sigma)>0$, independent of $N$, such that the following holds. 
 \begin{itemize}[noitemsep,topsep=0pt,partopsep=0pt,parsep=0pt]
 \item[$(a)$] \emph{Strong rigidity inside $\II$:}  There is a time-independent labeling~$\rl$ such that $\gamma_{\rl(i)}(t_1)\in[E_*-\Sigma/4,E_*+\Sigma/4]$, for all $i\in \II=\llbracket L-\sigma N,L+\sigma N\rrbracket$, where $L\in N_\N$ is the largest integer such that $\gamma_{\rl(L)}(t_1)\le E_*$. Moreover, for any (small) $\xi>0$ and any (large) $D>0$ we have
 \begin{align}\label{assumption:rigidity strong}
\P\left(\left|\lambda_{i}(t)-\gamma_{\rl(i)}(t) \right| \le \frac{N^{\xi}}{N}\,,\;\forall t\in[t_1,t_2]\right)\ge 1- N^{-D}\,,\qquad  \forall i\in \II\,,
 \end{align}
 for large enough $N\ge N_0(\xi,D)$, where~$(\gamma_i(t))$ are $N$-quantiles with respect to the measure~$\varrho_t$; see~\eqref{the gammas}.
 \item[$(b)$] \emph{Weak rigidity outside $\II$:}  For any $\xi\in(0,\delta)$ and any (large) $D>0$, we have
 \begin{align}\label{assumption:rigidity weak}
  \P\bigg(\bigg|\frac{1}{N}\sum_{k\,:\, |L-k|\ge \sigma N}\frac{1}{\lambda_k(t)-{E_*}  }- \int\limits_{\R\backslash\boldsymbol{I}(t)}\frac{\varrho_t(x)\dd x}{x-{E_*} }\bigg|\le\frac{N^{\xi}}{N^{\delta}}\,, \forall t\in[t_1,t_2]\bigg)\ge 1-N^{-D}\,,
 \end{align}
for large enough $N\ge N_0(\xi,D)$, where $\boldsymbol{I}(t)\deq[\gamma_{\rl(L-\sigma N)}(t),\gamma_{\rl(L+\sigma N)}(t)]$.

 \end{itemize}

 \item[$(3)$] \emph{Level repulsion inside $\II$:} For any  $i\in \II$ and  $t\in[t_1,t_2]$,
 \begin{align}\label{lever}
 \P\left(|\lambda_i(t)-\lambda_{i\pm 1}(t)|\le \frac{u}{N}\,,  \; |\lambda_i(t)-\gamma_{\ell(i)}(t)|\le \frac{N^{\xi}}{N} 
 \right)\le N^{\delta} u^{\beta+1}\,,\qquad\qquad u>0\,,
 \end{align}
for large enough $N$.

 \item[$(4)$] \emph{ H\"older continuity of DBM:}  For any (small) $\xi>0$ and any (large) $D>0$, we have
 \begin{align}
  \P\Big(|\lambda_k(t)-\lambda_k(s)|\le N^\xi\sqrt{t-s}\,,\; \forall t,s\in[t_1,t_2]\,,t\ge s\Big)\ge 1- N^{-D}\,,\qquad  \forall k\in \N_N\backslash \II\,,
 \end{align}
for large enough $N\ge N_0(\xi,D)$.	
\end{itemize}

Then, there are small constants $\frak{f}, \chi,\alpha>0$, such that the following holds. Fix $n\in\N$ and let $\mathcal{O}\,:\,\R^n\longrightarrow\R$ be smooth and compactly supported. Fix any $T\in[t_1+N^{2\delta}/N,t_2]$. Set $\varrho_*\deq \varrho_{T}(\gamma_{L}(T))$. Then
\begin{align}\label{equation main result}
\E\left[\mathcal{O}\left(\Big((N\varrho_*)(\lambda_{i_0}(T)-\lambda_{i_0+j}(T))\right)_{j=1}^n\Big) \right]=\E^\Gauss\left[\mathcal{O}\left(\Big((N\varrho_\#)(\lambda_{i_0'}-\lambda_{i_0'+j})\right)_{j=1}^n\Big) \right]+ O(N^{-\frak{f}})\,,
\end{align}
for $N$ sufficiently large, for any $i_0,i_0'\in\N_N$ satisfying $|i_0-L|\le N^\chi$, $|i_0'-L'|\le N^{\chi}$ with any $L'\in [\alpha N,(1-\alpha)N]$, and where $\varrho_\#\deq \varrho_{sc}(\gamma_{L',sc})$ denotes the density of the semicircle law $\varrho_{sc}$ at the location of the $L'$-th $N$-quantile of $\varrho_{sc}$.

\end{theorem}

\begin{remark} The formula~\eqref{equation main result} expresses the {\it single gap universality}, \ie
that the joint distribution of $n$ consecutive gaps coincides with that of a Gaussian invariant  ensemble for any fixed $n$.
Single gap universality clearly implies the weaker {\it averaged gap universality}, where
\eqref{equation main result} is averaged over~$N^b$ consecutive~$i_0$'s, for some $0<b<1$.
 It is well known that averaged  gap universality implies the {\it averaged energy universality}, \ie the universality of
 the local correlation functions around an energy~$E$, averaged over~$E$ near the reference energy $E_*$; see \eg Section 7 of~\cite{ESYY}. 
\end{remark}

\begin{remark} The measure $\varrho$ in Theorem~\ref{main theorem} is assumed to be deterministic, but it may depend on $N$ in contrast to the measure $\varrho_0^{(\infty)}$ of~\eqref{le GG limit} which is indeed the limiting object as $N\to\infty$. Consequently, the semicircular flow $\varrho_t=\mathcal{F}_t[\varrho]$ will also be $N$-dependent. Typically one expects that $\varrho_t$ converges weakly to $\varrho_t^{\infty}$, yet the speed of convergence may be very slow and hence not be compatible with Assumption~$(2)$ of Theorem~\ref{main theorem}. In Subsection~\ref{appendix hoelder}, we discuss Assumption~$(1)$ in more detail.

Notice that the initial condition~$\bla(0)$ of
the DBM and the initial data~$\varrho$ of the semicircular flow do not have to be  related; this will allow us
for an additional freedom in the applications. We only require that~$\bla(t)$ is close to the quantiles of~$\varrho_t$ in a short time interval $t\in [t_1, t_2]$ and only locally near the reference energy~$E_*$. 
We also allow for a possible relabeling~$\ell$ that can be used to accommodate outliers in applications.
At first reading the reader may ignore~$\ell$ and consider~$\ell (i)=i$ for simplicity.
 \end{remark}

\subsection{Random matrix flow and universality} 
In this subsection, we briefly explain how Theorem~\ref{main theorem} can be used to prove bulk universality for many random matrix ensembles
$H$. We will follow the three-step strategy initiated in a series of works~\cite{ESY4, ESYY, EYY}; see~\cite{EYBull} for a concise summary. 

{\it Step 1} is to prove a local law, which includes
 rigidity for the eigenvalues and bounds on the resolvent matrix elements $G(z) = (H-z)^{-1}$ down almost to
 the scale of the eigenvalue spacing, \ie  for  $\eta = \im z \gg N^{-1}$.
 This step is typically  model dependent, mainly because the limiting  density of the eigenvalues varies from 
 model to model. The key tool is the self-consistent equation for the Stieltjes transform of the
 density (and its vector version for the individual matrix elements $G_{ii}$); its solvability and
 stability properties need to be investigated for each model.

 {\it Step 2} is to prove universality for matrices with a small Gaussian component
 that can conveniently be generated by running a matrix valued Ornstein-Uhlenbeck process. 
Theorem~\ref{main theorem} is used in this step and it replaces the previous argument that relied
on a global equilibrium measure and its version with  relaxation.
As advertised in the introduction, Theorem~\ref{main theorem}
 requires rigidity information only locally, in particular it also applies to
 models where the limiting density  is supported on several intervals.
 Step~2 is model independent once the input conditions of Theorem~\ref{main theorem} are verified.
 
 Finally, {\it Step 3} is a perturbation argument which is also very general. Using the Green function comparison
 strategy~\cite{EYY} and the moment matching (introduced first in~\cite{TV} in the context of random matrices),
 one can remove the tiny Gaussian component. The main input here is the a priori bound
 on the resolvent matrix elements obtained in Step~1.

More concretely, 
consider a random $N\times N$ hermitian or symmetric matrix $H_t=H^*_t$ with matrix elements $(h_{ij})$.
Suppose the matrix elements are time-dependent and they satisfy the Ornstein-Uhlenbeck~(OU) process 
\begin{align}
  \rd h_{ij}&= \frac{\rd B_{ij}}{\sqrt{N}} - \frac{1}{2} h_{ij}\,\rd t\,, \qquad\qquad
i,j\in \N_N\,,\quad i\le j \,,\label{zij1}
\end{align}
where $(B_{ij}\,:\, i<j)$ are independent complex Brownian motions with variance $t$ and $(B_{ii})$ are independent real Brownian with variance $t$ for $\beta=2$; while for $\beta=1$, $(B_{ij}\,:\,i<j)$ are independent real Brownian motions with variance $t$ and $B_{ii}$ are real Brownian motion with variance $2t$. It is easy to check that the solution to~\eqref{zij1}, $H_t$, with initial condition $ H_0$, satisfies the distributional equality
\begin{align}\label{OUbis}
  H_t \sim \euler{-t/2} H_0 + (1-\euler{-t})^{1/2} U\,,
\end{align}
where $U$ is Gaussian, \ie belongs to the GUE ($\beta=2$), respectively to the GOE ($\beta=1$),
and  $U$ is independent of $ H_0$. 

The eigenvalues of $H_t$, here denoted by $\bla(t)$, satisfy~\cite{DyB} the SDE~\eqref{le original dbm}, with $\beta=1$ or $\beta=2$, where the initial condition $\bla(0)$ is given by the eigenvalues of the initial matrix $H_0$.
We will run the OU process until time $t_2=o(1)$.
Let $\varrho$ denote the limiting density of $H_0$. 
We fix an energy $E_*$ in the bulk spectrum of $H_0$, \ie $\varrho(E_*)\ge c>0$; it is easy to see that
$E_*$ stays in the bulk of $H_t$ as well for any $t\le t_2$. 
The assumptions of Theorem~\ref{main theorem} can then, 
via the identification~\eqref{OUbis}, be checked from the matrix flow $H_t$ in the time slice $t\in [t_1, t_2]$. 
The typical choice is $t_2= N^{-\epsilon}$ and $t_1 = t_2 - N^{-1+2\delta}$, with some small positive 
exponents $\epsilon\ll \delta$.

Assumption~$(2)$ can be checked from a local law for the random matrix $H_t$.
We need such information not only for the original matrix $H_0$, but along the
whole OU flow.  Typically, however, when the local law is proven for some matrix $H_0$,
it also holds for $H_t$, \ie for $H_0$  with a Gaussian convolution.
 Notice that
the strong form of rigidity, an almost optimal bound on $\lambda_i(t) -\gamma_i(t)$
expressed in~\eqref{assumption:rigidity strong}, is
needed only for eigenvalues near $E_*$ in the bulk. Much weaker information  is needed for
far away eigenvalues; the condition~\eqref{assumption:rigidity weak} involves
controlling the density only on the macroscopic scale. In terms of the Stieltjes
transform, $m^{(N)}_t(z) \deq \frac{1}{N} \tr ( H_t-z)^{-1}$, $z\in \C^+$, of the empirical density, 
Assumption~$(2)$ follows if the bounds 
\begin{align}
   &|m^{(N)}_t(z) - m_t(z)| \le \frac{N^\xi}{N\eta}\,, \qquad \mbox{for} \;\;   z= E+i\eta\,, \quad
   \eta \in [N^{-1+\xi},c\Sigma]\,, \quad |E-E_*|\le \Sigma\,,\label{mm1}\\
  &|m^{(N)}_t(E_*+\ii\eta) - m_t(E_*+\ii\eta)| \le \frac{N^\xi}{N^\delta}\,, \qquad \mbox{for} \;\; 
   \eta \in { [c\Sigma,1]}\,,\label{mm2}
\end{align}
hold with high probability, for any $t\in [t_1, t_2]$. Indeed, \eqref{mm2} directly implies 
\eqref{assumption:rigidity weak}. By a simple application
of the Helffer-Sj\"ostrand formula
(\eg following the proof of  Lemma 8.1 in \cite{EKYY1}), we see from \eqref{mm1} that
\begin{align}\label{nee}
    \Big| \# \{ \lambda_j(t) \in [E_1,E_2]\} - \# \{ \gamma_j(t) \in [E_1,E_2]\}\Big| \le CN^\xi\,, 
    \qquad \forall \; E_1, E_2 \in \left[ E-\frac{1}{2}\Sigma, E+\frac{1}{2}\Sigma\right]\,.
\end{align}
 In particular, rigidity between the $\bla(t)$ and $\boldsymbol{\gamma}(t)$ sequences holds
on scale $N^{-1+\xi}$  within $[E-\frac{1}{2}\Sigma, E+\frac{1}{2}\Sigma]$. This implies
$|\lambda_i(t) -\gamma_{\ell(i)}(t)|\le N^{-1+\xi}$, for any $i\in \II$, up to an overall shift in the labeling that is encoded in the labeling function $\ell(i)$.
We only need to show that the labeling $\ell(i)$ is time-independent, \ie that along the whole time 
interval $t\in [t_1, t_2]$ it is 
{\it the same} element of the $\boldsymbol{\gamma}(t)$ sequence that stays close to
a given element of $\bla(t)$ within the rigidity precision $N^{-1+\xi}$. 
 We call this property the {\it persistent trailing} of DBM by the flow of the quantiles.
Given \eqref{nee}, it is sufficient to check this for one element of the sequence; \eg that
if $|\lambda_L(t_1) - \gamma_{\ell(L)}(t_1)|\le N^{-1+\xi}$ with some shifted index $\ell(L)$,
then $|\lambda_L(t) - \gamma_{\ell(L)}(t)|\le N^{-1+\xi}$, for all $t\in [t_1, t_2]$.
Notice that persistent trailing  is  a nontrivial feature of the DBM
 since the  length of the time interval $t_2-t_1 = N^{-1+2\delta}$ is
much bigger than the rigidity scale $N^{-1+\xi}$. Nevertheless, in Proposition~\ref{prop:persist} 
in Appendix~\ref{appendix:persist} we show that there is an event $\Xi_0$ in the probability space
of the Brownian motions with $\P (\Xi_0)\ge 1-N^{-\xi/2}$ such that 
$\gamma_{\ell(L)}(t)$ persistently trails $\lambda_L(t)$. It is easy to see that the universality in  Theorem~\ref{main theorem} also holds if Assumptions $(2)$-$(3)$ are valid only on the event~$\Xi_0$.

Level repulsion estimates of the form of Assumption~$(3)$ for random matrix ensembles can
be obtained using the method of~\cite{ESY4}. This approach requires two inputs:
strong local rigidity as in~\eqref{assumption:rigidity strong} and smoothness of the 
distribution of the matrix elements of~$H$. The former is already verified by Assumption~$(2)$,
the latter needs a slight extension of~\cite{ESY4} to ``almost-smooth'' distributions,
where smoothing may be provided by the OU process.
Indeed, in Appendix~B of~\cite{FE} it was shown that~$H_t$ satisfies  level repulsion
in the form~\eqref{lever}, if~$t = N^{-c\delta}$ with some small constant $c>0$
(another merit of the proof in~\cite{FE} is that it also presents the necessary modifications
to cover  symmetric matrices as well, while~\cite{ESY4} was written for hermitian matrices only).
So we will choose $\epsilon = \frac{c}{2}\delta$ in the definition $t_2= N^{-\epsilon}$
to guarantee that~\eqref{lever} holds for any $t\in [t_1, t_2]$.
Notice that the only reason to run the DBM up to a relatively large time $t_2= N^{-\epsilon}$ is
to guarantee that the  smoothing effect is substantial to yield level repulsion. If the
distribution of $H_0$ were smooth initially, so level repulsion
in its original form~\cite{ESY4} applied, we could have chosen $t_1=0$, $t_2= N^{-1+2\delta}$
with some small $\delta>0$.

Finally, Assumption~$(4)$ can easily checked as follows. For any two $N\times N$ matrices $A=A^*$, $B=B^*$, we have $ \mathrm{dist} \{\mathrm{Sp}(A),\mathrm{Sp}(B)\}\le \|A-B\|_{\infty}$, where  $\mathrm{Sp}(A)$, $\mathrm{Sp}(B)$ denote the spectra of $A,B$ and where $\|\,\cdot\|_{\infty}$ denotes the operator norm. Also recall that the operator norm of $U$ is bounded by a constant with overwhelming probability; see, \eg Exercise 2.1.30 of~\cite{AGZ}. Thus, choosing $A=H_t$, $B=H_0$, we see that Assumption~$(4)$ is satisfied provided that $\|H_0\|_\infty\le N^{\xi/2}$ with overwhelming probability. 
This bound can be easily proven for all matrix models we have in mind.

Having checked the assumptions, the conclusion of Theorem~\ref{main theorem} is that gap
universality holds for any matrix with a substantial Gaussian component of size $t_2\sim N^{-\epsilon}$.
The rest is a standard moment matching and Green function comparison argument
that we sketch for completeness.

Given an initial Wigner-like matrix $ \wh H$ for which we eventually wish to
prove universality, we choose $t_2= N^{-\epsilon}$ with 
a sufficiently small $\epsilon>0$. By moment matching (see, \eg Lemma~6.5 of~\cite{EYY}),
we construct another matrix $H_0$ 
such that the solution $H_{t_2}$ at time $t_2$
of the matrix Ornstein-Uhlenbeck process~\eqref{zij1} with initial condition~$H_0$ is close to~$\wh H$ 
in the four moment sense. Choosing~$T=t_2$ in Theorem~\ref{main theorem}, we obtain gap universality
for $H_T$ which also implies universality of local correlation functions at $E$
with a  small averaging in the energy parameter $E$ around $E_*$.
The local eigenvalue statistics of $\wh H$ and $H_T$ 
coincide by the Green function comparison theorem introduced in~\cite{EYY}.
More precisely, the method of~\cite{EYY} gives coincidence in the sense of correlation functions
while Theorem~1.10 of ~\cite{KY} 
 extends the Green function comparison
method to individual eigenvalues, hence to gaps as well. This completes our sketch on how to apply
Theorem~\ref{main theorem} for random matrix models.

\subsection{Strategy of the proof of Theorem~\ref{main theorem}}
The first part of the proof is to understand the dynamics on 
a macroscopic scale, \ie to control the semicircular flow
and the induced dynamics on the time-dependent quantiles $\gamma_i(t)$.
This analysis is of interest itself and it is deferred to the Appendix~\ref{section classical semicircle flow}
since it requires quite different tools than the main part of the proof. The key information (collected in Section~\ref{sec:reg})  is that the quantiles in the bulk move
coherently with a local mean velocity that varies in
time on the macroscopic scale. Since we concentrate on
the vicinity of a fixed energy~$E_*$ and on a small time window, by a simple
linear shift we can achieve that the mean velocity 
is negligible near~$E_*$.

The second step is to localize the problem:
we choose an integer $K\gg 1$ such that
\begin{align}\label{le choice of K}
N^{\delta}K^{10}\le N\,, \qquad \qquad K\le N^{\delta}\,.
\end{align}
We consider the conditional measure on $\cK=2K+1$ consecutive {\it internal 
points} $\bx = (\lambda_{L-K}, \ldots ,\lambda_{L+K})$, labeled by $I\deq \llbracket L-K, L+K\rrbracket $, conditioned
on the remaining $N-\cK$ {\it external points} $\by = ( \lambda_i \; : \; |i-L|>K)$. 
The index $L$ is chosen so that $\gamma_{\ell(L)}$ is close to $E_*$,
where $\gamma_k$ is the $k$-th $N$-quantile of the density 
at $t_1$.
In the equilibrium setup this corresponds to the local Gibbs measure
with boundary conditions given by~$\by$ (this idea was first
introduced in~\cite{BEY} in the $\beta$-ensemble context).
In our non-equilibrium setup, we work in the path space
and condition on  the whole trajectory $\bY = \{ \by(t)\; : \; t\in [T_1, t_2]\}$,
starting at some time $T_1\ge t_1$ chosen later.
The configuration interval $\boldsymbol{J}(t)= [y_{L-K-1}(t), y_{L+K+1}(t)]$ for the conditional measure 
is time-dependent, but by rigidity it is quite close to the corresponding
interval $[\gamma_{\ell(L-K-1)}(t), \gamma_{\ell(L+K+1)}(t)]$ given by the quantiles
that remains practically constant owing to the removal  of the mean velocity.
Still, $\boldsymbol{J}(t)$ may wiggle on the rigidity  scale $N^\xi/N$ which is
much bigger than our target scale, $1/N$,  the size of the gap, so that we cannot tolerate
this imprecision.
Furthermore, similarly to the basic idea of the local relaxation flow~\cite{ESY4, ESYY} we want to 
achieve universality by showing that the measure converges to a
(local) reference equilibrium measure. The local Gibbs measures
with boundary condition $\by(t)$ change too quickly to serve
as useful reference measures.

Therefore, in the third step, we define a {\it time-independent} local measure,~$\om_{T_1}$
with exterior points~$\wt\gamma_k$, $k\in I^c$. These exterior points coincide with~$y_k(T_1)$
for~$k$ far away from~$L$ while they are given by a typical configuration~$\bz$ of an auxiliary quadratic $\beta$-ensemble for $k$ near the boundary of $I$ (with a smooth interpolation in between). The auxiliary ensemble is chosen in such a way that the local density around~$E_*$ matches.
Using the rigidity bounds for both~$\by$ and~$\bz$, we establish that~$\om_{T_1}$
satisfies the logarithmic Sobolev inequality (LSI) and the corresponding
dynamics approaches to equilibrium on a time scale of order~$K/N$.
Furthermore, we show that the measure~$\om_{T_1}$ is rigid by using
a general criterion for rigidity of local measures given in  Theorem~4.2 of~\cite{EYSG}
together with the careful choice of the auxiliary ensemble. Moreover, we notice
that~$\om_{T_1}$ satisfies a level repulsion bound due to Theorem~4.3 of~\cite{EYSG}.
Finally, Theorem~4.1 of~\cite{EYSG} implies that the gap statistics of~$\om_{T_1}$
are universal.

The fourth step is to consider $\wt x_i(t)$, $i\in I$, $t\ge T_1$, 
the solution of the local DBM with exterior points $\wt\gamma_k$, $k\in I^c$,
and with initial condition $\wt x_i (T_1)=x_i (T_1)$. Writing
the distribution of $\wt\bx(t)$ as $g_t\, \om_{T_1}$, we derive
fast convergence to equilibrium, \ie for times $t\ge T_1'\deq T_1 + K(K/N)$
the measure $g_t\, \om_{T_1}$ is exponentially close to equilibrium 
in the relative entropy sense. This information can be used to
transfer rigidity and level repulsion from~$\om_{T_1}$ to~$g_t\, \om_{T_1}$,
furthermore it shows that the gap statistics of~$\wt \bx(t)$ are the same
as those of~$\om_{T_1}$, hence are universal.

The next idea, in the fifth step, is to {\it couple} the evolution of $\wt \bx$ to $\bx$ by using
the same Brownian motions in the DBMs. This basic coupling idea first appeared in~\cite{FE} in this
context ; its main advantage is that taking the difference
of the original DBM and the DBM for $\wt \bx$, we see that the 
difference vector $\bv \deq \bx - \wt \bx$ satisfies 
a system of ordinary differential equations (ODEs); the stochastic differentials drop out.
Roughly speaking, these ODEs have the form (see~\eqref{heat})
\begin{align}\label{heat3}
  \frac{\rd v_i}{\rd t} =  -  (\cB \bv)_i + F_i\,, \qquad   (\cB \bv)_i\deq \sum_{j\in I} B_{ij} (v_i-v_j) + W_i v_i\,,
\end{align}
with time-dependent coefficients $B_{ij}$, $W_i$ and a ``forcing term'' $F_i$
that all depend on the paths $\bx(t), \wt\bx(t)$. These coefficients are crudely
given by
$$
    B_{ij} = \frac{1}{N(x_i-x_j)(\wt x_i- \wt x_j)}\,, \quad\quad W_i = \frac{1}{N}\sum_{k\not\in I}
     \frac{1}{(x_i-y_k)(\wt x_i- \wt\gamma_k)}\,, \quad\quad F_i = \frac{1}{N}\sum_{k\not\in I}
     \frac{y_k-\wt\gamma_k}{(x_i-y_k)(\wt x_i- \wt\gamma_k)}\,.
$$
The equation~\eqref{heat3} is very similar to the basic equation studied in~\cite{EYSG}
but the forcing term is new. The key result of~\cite{EYSG} is a H\"older regularity theory
for~\eqref{heat3} without forcing, under suitable conditions on the coefficients.
We extend this statement to include the forcing term; here we rely on 
the finite speed of propagation, proved also in~\cite{EYSG}.
H\"older regularity in this context yields that, after some time of order $K^c/N$, $c\sim 1/100$, the discrete derivative
$v_{i+1} - v_{i}$ is much smaller than its naive size~$1/N$. Since $v_{i+1} - v_{i}= x_{i+1} - x_{i} - (\wt x_{i+1}-\wt x_{i})$,
we see that the gaps of $\bx$ and $\wt \bx$ coincide to leading order. Since the gaps of $\wt\bx$ were shown to be universal
in the previous step, we obtain that the gaps of $\bx(T)$, $T\deq T_1' + K^cN^{-1}$,  are universal.

There are several technical complications behind this scheme, most importantly 
we need to regularize the local singularity in the kernel~$B_{ij}$ when $x_i\approx x_{i\pm 1}$.
In fact, two different regularizations are used; the regularization of the dynamics
 in Section~\ref{subsection:step 1}
is borrowed from Section~3.1 of~\cite{FE}, while the regularization  of the equilibrium
measure $\om_{T_1}$ explained at the end of Section~\ref{subsection:three}  
is similar to the one in Section~9.3 of~\cite{EYSG} but with a different choice of regularization scale.

\section{Concepts}
In this section we recall essential concepts that will be used in the proof of Theorem~\ref{main theorem}.
\subsection{Definition of general $\beta$-ensembles }\label{sec:beta ensemble}
We first recall the notion of  $\beta$-ensembles or log-gases.  Let $N\in\N$ and recall the definition of the set $\digammac\subset\R^N$ in~\eqref{Weyl chamber}. Consider the probability distribution on $\digammac$ given by
\begin{align}\label{eqn:measure}
\mu_{\beta, V}^{(N)} (\rd\bla)\equiv \mu_V (\rd \bla) \deq\frac{1}{Z_{V}}\euler{-\beta N \cH}\rd\bla\,,\qquad\cH\equiv\cH(\bla) \deq\sum_{i=1}^N\frac{1}{2}V(\la_i)-\frac{1}{N}\sum_{1\le i<j\le N}\log (\la_j-\la_i)
\end{align}
 where $\beta>0$, $\rd\bla\deq {\bf 1}(\bla\in\digamma^{(N)})\,\rd \la_1\rd \la_2\cdots\rd \la_N$,
and $Z_{V}\equiv Z_{\beta,V}^{(N)}$ is a normalization. Here $V$ is a $N$-independent potential, \ie a real-valued, sufficiently regular function on $\R$ to be specified in each case. In the following, we often omit the parameters $N$ and~$\beta$ from the notation. We use~$\mathbb{P}^{\mu_V}$ and $\E^{\mu_V}$ to denote the probability and the expectation with respect to~$\mu_V$.
We view~$\mu_V$ as a Gibbs measure of~$N$ particles on~$\R$ with a logarithmic interaction, where the parameter
$\beta > 0$ may be interpreted as the inverse temperature. We refer to the variables~$(\la_i)$ as particles or points and we call the system a $\beta$-log-gas or a $\beta$-ensemble. We assume that the potential~$V$ is a~$C^4$
function on $\R$ such that its second derivative is bounded below, \ie we have
\begin{align}\label{assumption 1 for beta universality}
\inf_{x\in\R} V'' (x)\ge -2C_V\,,
\end{align}
for some constant $C_V\ge 0$, and we further assume that
\begin{align}\label{assumption 2 for beta universality}
V(x) > (2 + c) \log(1 + |x|)\,, \qquad\qquad x\in\R\,,
\end{align}
for some $c > 0$, for large enough $|x|$. It is also well known, see, \eg~\cite{BPS}, that under these conditions
the measure is normalizable, $Z_{V} < \infty$. Further, the averaged density of the empirical spectral measure,~$\varrho_V^{(N)}$,
defined as
\begin{align}\label{reference section 7 1}
\varrho_{V}^{(N)}\deq \E^{\mu_{V}}\frac{1}{N}\sum_{i=1}^N\delta_{\la_i}\,,
\end{align}
converges weakly in the limit $N\to\infty$ to a continuous function, $\varrho_V$, the equilibrium density, of compact support. It is
well known that $\varrho_V$ satisfies 
\begin{align}\label{regular potential U}
V'(x)=-2\int_\R\frac{\varrho_V(y)\rd y}{y-x}\,,\qquad\qquad x\in\supp\varrho_V\,.
\end{align}
In fact, equality in~\eqref{regular potential U} holds if and only if $x\in\supp\varrho_V$. 

Viewing the points $\bla=(\la_i)$ as points or particles on~$\R$, we define the quantile  of the $k$-th particle,~$\gamma_k$, under the $\beta$-ensemble~$\mu_V$ by
\begin{align}\label{general beta ensemble classical location}
 \int_{-\infty}^{\gamma_k}\varrho_V(x)\rd x=\frac{k}{N}\,.
\end{align}
 For a detailed discussion of general $\beta$-ensemble and the proof of the properties mentioned above we refer, \eg to~\cite{AGZ,BEY}.

Assume for the moment that the minimizer  $\varrho_V$ is supported on a single interval $[a,b]$, and that $V$ is ``regular''
in the sense of~\cite{KML}, \ie the equilibrium density of $V$ is positive on $(a,b)$ and vanishes
like a square root at each of the endpoints of $[a, b]$.  From~\cite{BEY,BEY2} we then have the following rigidity result.
\begin{proposition}\label{rigidity for time dependent beta one}
Let $V\in\C^4(\R)$ be a ``regular'' potential and assume that $\varrho_V$ is supported on a single interval. Then, for any $\xi>0$ there are constants $c_0,c_1>0$, such that
\begin{align}\label{rigidity for b ensemble one}
\P\left(|\lambda_k-\gamma_k|\ge N^{\xi} N^{-\frac{2}{3}}{\check k}^{-\frac{1}{3}}\right)\le \euler{-c_0N^{c_1}}\,,
\quad\qquad  1\le k\le N\,,
\end{align}
where $\check k\deq\min\{k,N-k+1\}$, for $N$ sufficiently large.
\end{proposition}
Proposition~\ref{rigidity for time dependent beta one} will only be used as an auxiliary result (see Subsection~\ref{subsubsection:aux beta} below), since, for most potentials of interests in the present paper, the equilibrium density~$\varrho_V$ is not supported on a single interval. The extension of Proposition~\ref{rigidity for time dependent beta one} to that settings has not been established.

Finally, for the Gaussian case, $V(x)=x^2/2$, we write $\mu_\Gauss$ instead of $\mu_V$, since $\mu_\Gauss$ is the equilibrium measure for the DBM. More precisely, the Gaussian distribution on $\digamma^{(N)}$ is given by
\begin{align}\label{gaussian measures}
\mu_{\Gauss}(\rd\bla)=\frac{1}{Z_{\Gauss}}\euler{-\beta N \cH_\Gauss}\rd\bla\,,\qquad\quad\cH_\Gauss(\bla) \deq\sum_{i=1}^N\frac{1}{4}\la_i^2-\frac{1}{N}\sum_{1\le i<j\le N}\log (\la_j-\la_i)\,,
\end{align}
where $Z_{\Gauss}\equiv Z_{\beta,\Gauss}^{(N)}$ is the normalization.

\subsection{Dyson Brownian motion}\label{sec:DBM}
Consider the DBM, $\bla(t)$, $t\ge 0$, on $\digammac$ of~\eqref{le original dbm} with initial condition $\bla(0)\in\digammac$. Denote by $f_0\,\mu_\Gauss$,  the distribution of $\bla(0)$\footnote{Strictly speaking, the distribution of $\bla(0)$ may not allow a density $f_0$ with respect to $\mu_\Gauss$, but for $t>0$, $\bla(t)$ admits such a density. Our proofs are not affected by this technicality.} and let $f_t\,\mu_\Gauss$ denote the distribution of $\bla(t)$. 
Then $f_t\equiv f_{t,N}$ satisfies the forward equation 
\begin{align}\label{dy1}
\partial_{t} f_t =  \cL f_t\,,
\end{align}
where 
\begin{align}\label{L}
\cL=\cL_N\deq   \sum_{i=1}^N \frac{1}{\beta N}\partial_{i}^{2}  +\sum_{i=1}^N
\Bigg(- \frac{1}{2} \lambda_{i} -  \frac{1}{N}\sum_{j\ne i}
\frac{1}{\lambda_j - \lambda_i}\Bigg) \partial_{i}\,, \qquad 
\partial_i=\frac{\partial}{\partial\lambda_i}\,,
\end{align}
or in short $\cL=\frac{1}{\beta N}\Delta-(\nabla \cH_\Gauss)\cdot\nabla$, with $\cH_\Gauss$ as in~\eqref{gaussian measures}.

\subsection{Relative entropy, Bakry-\'{E}mery criterion and the logarithmic Sobolev inequality}\label{subsec:rel entro}
A cornerstone in our proof is the analysis of the relaxation of the dynamics~\eqref{dy1}. Such an approach  was first introduced  in Section 5.1 of~\cite{ESY3}.
The presentation here follows~\cite{ESYY}.

Let $\mu$ be a probability measure on $\digamma^{(N)}$ be given by a general 
Hamiltonian $\cH$:
\begin{align}
  \rd\mu(\bx) = \frac{1}{Z}\euler{-\beta N \cH(\bx)}\,\rd \bx\,, \qquad 
\label{convBE}
\end{align}
and let $L$ be the generator, symmetric with respect to the measure $\rd\mu$,
 defined by the Dirichlet form
\begin{align}\label{D}
   D(f)= D_\mu(f) = -\int f L f \rd \mu\deq \frac{1}{\beta  N}\sum_j 
  \int(\pt_j f)^2 \rd \mu\,, \qquad \pt_j = \pt_{x_j}.
\end{align}
The relative entropy of two absolutely continuous probability measures on $\R^N$ is given by 
$$
   S (\wt\nu| \nu )\deq 
\int   \frac {\rd \wt\nu}  {\rd \nu} \log \left ( \frac {\rd \wt\nu}  {\rd \nu} \right ) \rd\nu\,.
$$
If $\rd\wt\nu = f\rd\nu$, then we use the notation $S_{\nu}(f)\deq S(f\nu|\nu)$.
The entropy can be used to control the total variation norm via the well-known inequalities
\begin{align}\label{entropyneq}
   \int |f-1|\rd \nu \le \sqrt{ 2 S_{\nu}(f)}\,,\qquad\qquad \mathbb{P}^{f\nu}(A)\le \mathbb{P}^{\nu}(A)+\sqrt{2S_{\nu}(f)}\,,
\end{align}
for any $\nu$-measurable event $A$.

Let $f_t$ be the solution to  the evolution equation $\partial_t f_t = L f_t$, $t>0$, with a given initial condition~$f_0$. Assuming that the Hamiltonian~$\cH$ satisfies
\begin{align}\label{convexham}
 \nabla^2\cH(\bla)=\mbox{Hess} \, \cH(\bla) \ge \vartheta \,,
 \quad\qquad \bla\in\digamma^{(N)}\,,
\end{align}
the Bakry-\'{E}mery criterion~\cite{BakEme} yields the logarithmic Sobolev inequality (LSI)
\begin{align}\label{first LSI}
 S_\mu(f) \le\frac{2}{\vartheta}  D_\mu (\sqrt{f})\,, \qquad\qquad f = f_0\in \mathrm{L}^{\infty}(\R^N,\rd\bla)\,,
\end{align}
and the exponential relaxation of the entropy and Dirichlet form 
\begin{align*}
    S_\mu(f_t )  \le \euler{-2\vartheta t } S_\mu(f_0  )\,,
 \qquad\qquad D_\mu (\sqrt{f_t}) \le \euler{-2\vartheta t } D_\mu (\sqrt{f_0}) \,,\qquad\quad t>0\,.
\end{align*}

We assume from now on
that~$\cH$ is given by~$\cH_\Gauss$ (see~\eqref{gaussian measures}),~$L$ is given by~$\cL$ (see~\eqref{L}) and that the
equilibrium measure is the Gaussian one, $\mu=\mu_\Gauss$.
We then have  the convexity inequality
\begin{align}\label{convex}
\Big\langle \bv , \nabla^2  \cH_\Gauss(\bx)\bv\Big\rangle
\ge   \frac{1}{2} \,  \|\bv\|^2 + \frac{1}{N}
 \sum_{i<j} \frac{(v_i - v_j)^2}{(x_i-x_j)^2} \ge  \frac{1 }{2} \,  \|\bv\|^2,
 \qquad\qquad \bv\in\bR^N\,.
\end{align}
This  guarantees that $\mu_\Gauss$ satisfies the LSI with $\vartheta=1/2$ and the relaxation time is of order one.

\subsection{Localized measures}\label{subsect:localized measures}
Following~\cite{EYSG}, we choose $K\in\llbracket N^{\varpi}, N^{1/10}\rrbracket$, for some small $0<\varpi<1/10$ and pick $L\in\llbracket K,N-K\rrbracket$. 
We denote by $ I=\llbracket L-K,L+K\rrbracket$ a set of $\cK\deq 2K+1$ consecutive indices around $L$. Recall the definition of the set $\digamma^{(N)}\subset\R^N$ in~\eqref{Weyl chamber}. For $\bla\in\digamma^{(N)}$, we rename the points~as
\begin{align}\label{the bflambda}
 \bla=(\lambda_1,\lambda_2,\ldots,\lambda_N)=(y_1,\ldots, y_{L-K-1},x_{L-K},\ldots, x_{L+K},y_{L+K+1},\ldots ,y_{N})\,,
\end{align}
and we call $\bla$ a configuration (of $N$ particles or points on the real line). Note that on the right side of~\eqref{the bflambda} the points retain their original indices and are in increasing order, 
\begin{align}\label{bfx notation}
 \bx=(x_{L-K},\ldots, x_{K+L})\in\digamma^{(\cK)}\,,\qquad\by=(y_1,\ldots, y_{L-K-1},y_{L+K+1},\ldots, y_N)\in \digamma^{(N-\cK)}\,.
\end{align}
We refer to $\bx$ as the {\it internal points or particles} and to $\by$ as the {\it external points or particles}. In the following, we often fix the external points and consider the conditional measures on the internal points: Let~$\nu$ be a measure with density on $\digamma^{(N)}$. Then we denote by $\nu^{\by}$ the measure obtained by conditioning on~$\by$, \ie for~$\bla$ of the form~\eqref{the bflambda},
\begin{align*}
\nu^{\by}(\rd\bx)\equiv \nu^{\by}(\bx)\rd\bx\deq\frac{\nu(\bla)\rd\bx}{\int\nu(\bla)\rd \bx}=\frac{\nu(\bx,\by)\rd \bx}{\int \nu(\bx,\by)\rd\bx}\,,
\end{align*}
where, with slight abuse of notation, $\nu(\bx,\by)$ stands for $\nu(\bla)$. We refer to the fixed external points~$\by$ as boundary conditions of the measure $\nu^{\by}$. For fixed $\by\in\digamma^{(N-\cK)}$, all $(x_i)$ lie in the configuration interval
\begin{align}\label{localization interval}
 \boldsymbol{J}_\by\deq [y_{L-K-1},y_{L+K+1}]\,.
\end{align}
Thus $\nu^\by$ is supported on $(\boldsymbol{J}_\by)^{\mathcal{K}}\cap \digamma^{(\mathcal{K})}$, but with  a slight abuse of terminology we often say that~$\nu^{\by}$ is supported on $\boldsymbol{J}_\by$. In case $\nu=f\mu$, we define the conditioned density by~$f^{\by}\mu^{\by}=(f\,\mu)^{\by}$.

For a potential $V$, we consider the $\beta$-ensemble $\mu_V$ of~\eqref{eqn:measure}. For $K,L$ and $\by$ fixed, we can write~$\mu_V^{\by}$~as
\begin{align}\label{le definition of a localized gibbs measure}
 \mu_V^{\by}(\rd\bx)=\frac{1}{Z_V^{\by}}\euler{-\beta N\cH^{\by}(\bx)}\rd\bx\,,\qquad\qquad \cH^{\by}(\bx)=\sum_{i\in I}\frac{1}{2} V^{\by}(x_i)-\frac{1}{N}\sum_{\substack{i,j\in  I\\ i<j }}\log |x_j-x_i|\,,
\end{align}
 $x_i\in\boldsymbol{J}_\by$, with $Z_V^{\by}\equiv Z_{\beta,V}^{\by}$ a normalization and with the {\it external potential}
\begin{align}
V^{\by}(x)&= V(x)-\frac{2}{N}\sum_{i\not\in I}\log |x-y_i|\,.\label{equation external potential}
\end{align}

\section{Localizing the measures}\label{subsec:localizing the measures}

\subsection{Localization at time $T_1$}
Let $K\in\N_N$ satisfy~\eqref{le choice of K}, with $K\ge \lfloor N^{\varpi}\rfloor$, $0<\varpi<1/10$, \ie  
\begin{align*}
 K^{10} N^{\delta}\le N\,,\qquad\qquad K\ge N^{\varpi}\,.
\end{align*}
Recall the constant $\sigma>0$ from the assumptions of Theorem~\ref{main theorem}. Let $\chi\in (0,\varpi)$ be a small constant, to be chosen later on. Note that $N^\chi\ll K$. Then introduce the intervals of integers
\begin{align}\label{le definition of Is}
 I\deq\llbracket L-K, L+K\rrbracket\,,\; \IO\deq\llbracket L-K^5-N^{\chi}K^4, L+K^5+N^{\chi}K^4\rrbracket\,,\;\II\deq\llbracket L-\sigma N,L+\sigma N\rrbracket\,,
\end{align}
and we denote by $I^{c}$, $\IO^c$, $\II^c$ the complements of $I$, $\IO$, $\II$ in $\N_N$. Note that $I\subset \IO\subset \II$. For a configuration $\bla\in\digamma^{(N)}$, we introduce $\bx$ and $\by$ as in~\eqref{bfx notation}. 

Fix a small $\xi>0$. Let
\begin{align}\label{le Cgs1}
 \cG_s^1\deq &\left\{ \by\in \digamma^{(N-\cK)}\: \; : |y_k-\gamma_{\rl(k)}(s)|\le N^\xi/N\;, \;\forall k\in \II \right\}\,,
\end{align}
respectively,
\begin{multline}\label{le Cgs2}
\cG_s^{2}\deq \bigg\{ \by\in \digamma^{(N-\cK)}\: \; :
 \left|\frac{1}{N}\sum_{k\in \II^c}\frac{1}{y_k-x}- \int_{y\in[y_{L-\sigma N},y_{L+\sigma N}]^c}\frac{\varrho_s(y)\dd y}{y-x}\right|\le \frac{N^{\xi}}{N^{\delta}}\,,\\ \forall x\in[y_{L-K^5},y_{L+K^5}]\,
 \bigg\}\,,
\end{multline}
where $\cK=2K+1$. Note that for each $s\ge 0$, we choose the labeling in $\cG_s^1$ to be the one of Assumption~$(2)$ of Theorem~\ref{main theorem}. We then set
\begin{align}\label{le cGs}
 \cG_s\deq \cG_s^1\cap \cG_s^2\,.
\end{align}

For  any $\bY\deq\{\by(s)\in\digamma^{(N-\cK)}\, : \, s\in [t_1,t_2]\}$
 trajectory, we define the conditional measure  $\P^{\bY}$
on the~$\bX\deq\{\bx(s)\in\digamma^{(N-\cK)}\,:\, s\in [t_1,t_2]\}$ trajectories in the usual way. We use $\P^\bY$ to denote
the conditional measure on the whole $\bX$ trajectories, while
for any fixed time $s$, we use  $\P^{\by(s)}$ for the conditional
measure (on the $\bx(s)$ configurations, for any fixed $s$). We set
\begin{multline}\label{le event calG}
  \cG\deq \Big\{ \bY=\{ \by(s)\;: \; s\in [t_1, t_2]\} \; : \\ \; \by(s)\in \cG_s, \; \forall s\; : \; \P^{\bY}( |x_i(s)-\gamma_{\rl(i)}(s)|
\le N^\xi/N\,,  \forall i\in I ) \ge 1-N^{-D}\,,\\ \textrm{ and }\sup_{t,s\in [t_1,t_2]}\max_{k\in I^c}| y_k(t)-y_k(s)|\le N^{\xi}\sqrt{t-s} \Big\} \,,
\end{multline}
for small $\xi>0$ and large $D>0$.
\begin{lemma}\label{le lemma probability estimates}
Let $\bla(t)$, $t\ge 0$ be the DBM on $\digammac$ of~\eqref{le original dbm} with fixed initial condition $\bla(0)$. Let $x_i(t)\equiv\lambda_i(t)$, $i\in I$, respectively $y_k(t)\equiv\lambda_k(t)$, $k\not\in I$, for all $t\ge 0$.  Under the assumptions of Theorem~\ref{main theorem}, for any (small) $\xi>0$ and any (large) $D>0$, we have
\begin{align}\label{proba cGs}
   \P \left(\by(s)\in  \cG_s,\, s\in [t_1,t_2]\right) \ge 1- N^{-cD}\,,
\end{align}
and
\begin{align}\label{proba cG}
   \P (\cG) \ge 1-N^{-cD}\,,
\end{align}
$c>0$, for large enough $N\ge N_0(\xi,D)$, where $\P$ is with respect the Brownian motions $(B_i)$ in~\eqref{le original dbm}.
\end{lemma}
\begin{proof}
 Both estimates~\eqref{proba cGs} and~\eqref{proba cG} follow directly from the assumptions $(1)$-$(4)$ in Theorem~\ref{main theorem}.
\end{proof}

Throughout the rest of this section
we will consider the trajectory $\bY\in\cG$ as fixed. Nevertheless, all estimates will be uniform on $\cG$. In particular, all constants only depend on the constants in~\eqref{le event calG}, 
the constants $\delta,\epsilon$ and $\sigma$ of Theorem~\ref{main theorem} as well as the parameter $\xi>0$. 

\subsection{Regularity of the semicircular flow and removal of mean drift}\label{sec:reg}
Consider the DBM, $\bla(t)$ of~\eqref{le original dbm} with initial condition $\bla(0)$ and the semicircular flow $\varrho_t=\mathcal{F}_t[\varrho]$. We first study some regularity properties of $\varrho_t$ for $t\in [t_1,t_2]$. The following result is proven in Subsection~\ref{le subsubsection proof of classical flow result} of the Appendix~\ref{section classical semicircle flow}.
\begin{lemma}\label{le density lemma}
  Under the assumptions of Theorem~\ref{main theorem}, the semicircular flow $\varrho_t$ satisfies
  \begin{align}\label{assumption on regularity of varrho}
  C^{-1}\le \varrho_t(E)\le C\,,\qquad\qquad |\partial_E\varrho_t(E)|\le CN^\delta\,,
  \end{align}
for all $E\in[E_*-\Sigma/2,E_*+\Sigma/2]$ and all $t\in[t_1,t_2]$.
\end{lemma}

Let $L\in\N_N$ be as in Theorem~\ref{main theorem}. In particular, we have $\varrho_{t_1}(\gamma_L(t_1))\ge c>0$. Then, we have 
\begin{align}\label{sloth}
 |x_i(t)-\gamma_{\rl(i)}(t)|\le\frac{N^{\xi}}{N}\,,\qquad \qquad i\in I\,,\qquad \by(t)\in\cG_t\,,
\end{align}
with high probability for $N$ sufficiently large, uniformly in $t\in[t_1,t_2]$, for some labeling $\ell$ that will be fixed throughout the paper. Recall from~\eqref{the gammas} that the quantiles $\boldsymbol\gamma$ are determined by
\begin{align}\label{nochmals classical locations}
 \int_{-\infty}^{\gamma_{k}(t)}\,\varrho_t(x)\,\dd x=\frac{k}{N}\,.
\end{align}
The evolution of $\boldsymbol{\gamma}$ is studied in the Appendix~\ref{section classical semicircle flow} where the following result is proved.
\begin{lemma}\label{le lemma for new gamma}
 Under the assumptions of Theorem~\ref{main theorem}, the quantiles $(\gamma_k)$ defined through~\eqref{nochmals classical locations}, satisfy
\begin{align}\label{le eom dotgamma}
 \dot\gamma_{\rl(i)}(t)=-\int_\R\frac{\varrho_t(y)\dd y}{y-\gamma_{\rl(i)}(t)}-\frac{\gamma_{\rl(i)}(t)}{2}\,,\qquad\qquad \ell(i)\in\II\,,
\end{align}
for all $t>0$, where $\dot\gamma_{\rl(i)}(t)\equiv\frac{\dd}{\dd t}\gamma_{\rl(i)}(t)$.  In particular, we have $|\dot\gamma_{\rl(L)}|\le C$. Further, we have 
\begin{align}\label{le eom dotgamma 2}
 \dot\gamma_{\rl(i)}(t)  = -\frac{1}{N}\sum_{k=1}^N \frac{1}{\gamma_k(t) -\gamma_{\rl(i)}(t)}-\frac{\gamma_{\rl(i)}(t)}{2} + O\Big(\frac{N^\delta}{N}
  \Big)\,,\qquad\qquad \ell(i)\in\II\,,
\end{align}
uniformly in $t\in[T_1,t_2]$. Moreover, we have the estimates
\begin{align}
|\dot\gamma_{\rl(i)}(t)-\dot\gamma_{\rl(L)}(t)|\le CN^{-1+\delta} |i-L|\,,\qquad
\qquad |\dot\gamma_{\rl(i)}(t)-\dot\gamma_{\rl(i)}(T_1)|\le  CN^{\delta}(t-T_1)\,,\label{le estimate on dotgamma}
\end{align}
uniformly in $t\in[T_1,t_2]$ and $\ell(i)\in\II$. 
\end{lemma}
Equation~\eqref{le eom dotgamma} shows that the points~$( \gamma_{\rl(i)}(t))$  approximately satisfy 
a gradient flow evolution of particles with quadratic confinement and
interacting via the
mean field potential $ \frac{1}{N}\log |x-y|$.

Lemma~\ref{le lemma for new gamma} is proved in Subsection~\ref{le subsubsection proof of classical flow result} of Appendix~\ref{section classical semicircle flow}. Let us briefly mention how the constant~$\Sigma$ in Assumption~$(1)$ and the constant $\sigma$ in Assumption~$(2)$ can be related. For given $\Sigma>0$, we can choose $\sigma>0$ such that $\gamma_{\rl(i)}(t_1)$, for any $i\in\II
=\llbracket L-\sigma N,L+\sigma N\rrbracket $, all lie inside $[E_*-\Sigma/4,E_*+\Sigma/4]$. Then we know from Lemma~\ref{le lemma for new gamma} that $\gamma_{\rl(i)}(t)\in [E_*-\Sigma/2,E_*+\Sigma/2]$ for all $t\in[t_1,t_2]$. By Lemma~\ref{le density lemma}, we have control over $\varrho_t$ on $[E_*-\Sigma/2,E_*+\Sigma/2]$.

For simplicity of notation, we henceforth drop the labeling $\ell$ and simply write, with some abuse of notation, $\gamma_{i}(t)\equiv \gamma_{\rl(i)} (t)$. From~\eqref{le estimate on dotgamma} and~\eqref{sloth}, we conclude that
\begin{align*}
 |\lambda_L(t)-\lambda_L(T_1)|\le \upsilon_L (t-T_1)+o(t-T_1)+O\left(\frac{N^\xi}{N}\right)\,, \qquad\qquad t\in[ T_1,t_2]\,,
\end{align*}
 with high probability, where we have set $\upsilon_L\deq \dot\gamma_L(T_1)$.
We denote by $\overline\bla(t)$, $t\ge T_1$ the process obtained from $\bla(t)$ by setting 
\begin{align}\label{the shift}
 \overline\bla(t)\deq\bla(t)-\upsilon_L (t-T_1)\,,\qquad t\ge T_1\,.
\end{align}
Thus $\overline\bla(t)$ satisfies the SDE
\begin{align*}
 \dd\overline\bla(t)=\sqrt{\frac{2}{\beta N}} \,\rd B_i(t)-\upsilon_L\,\dd t - \frac{1}{N} \sum_{{ j\not=i}} \frac{1}{\overline\lambda_j(t)-\overline\lambda_i(t)} \,\rd t -\frac{\overline\lambda_i(t)+\upsilon_L (t-T_1)}{2}\,\rd t\,,\qquad i\in \N_N\,,\quad \beta\ge 1\,,
\end{align*}
for $t\ge T_1$. In the following we write $\overline{x}_i\equiv \overline{\lambda}_i$, $i\in I$, respectively $\overline{y}_k\equiv \overline{\lambda}_k$, $k\not\in I$, so that
$ \overline{\bY}=\left\{\overline{\by}(t)\in\digamma^{(N-\cK)}\,:\, t\in[T_1,t_2]\right\}$. Having shifted the original process $(\bla(t))$ as in~\eqref{the shift}, we also shift the distribution $\varrho_t$ and the quantiles $\boldsymbol{\gamma}$, for $t\ge T_1$, accordingly:
\begin{align*}
 \overline{\varrho}_t(x)\deq \varrho_t(x+\upsilon_L(t-T_1))\,,\qquad \overline\gamma_i(t)\deq \gamma_i(t)-\upsilon_L(t-T_1)\,,\quad\qquad t\in[T_1,t_2]\,, 
\end{align*}
$x\in\R$, $i\in\N_N$. In a similar way, we introduce the events $\overline\cG_s^1$, $\overline\cG_s^2$, $\overline\cG_s$ and $\overline\cG$ by replacing the quantities without bars with bars in~\eqref{le Cgs1},~\eqref{le Cgs2},~\eqref{le cGs} and~\eqref{le event calG}.

\subsection{The reference points~$\widetilde\gamma_j$}\label{subsection the reference points}
Once $\bY\in\cG$, thus also $\overline\bY\in\overline\cG$, is fixed, we introduce time-independent ``reference points'', $\wt{\boldsymbol{\gamma}}\equiv (\wt \gamma_k)$, as follows: For $k\in\N_N$, let
\begin{align*}
 \iota_k\deq \begin{cases}\frac{k-(L-K^5-N^{\chi}K^4)}{N^{\chi}K^4}&\text{if } \quad L-K^5-N^{\chi}K^4 \le k\le L-K^5\,,\\ \frac{L+K^5+N^\chi K^4-k}{N^{\chi}K^4}&\text{if }\quad L+K^5\le k\le L+K^5+N^{\chi}K^4\,, \\ 1 &\text{if }\quad L-K^5\le k\le L+K^5\,,\\ 0&\text{else}\,,\end{cases}
\end{align*}
with $\chi>0$ as in~\eqref{le definition of Is}, \ie $\iota_k$ is a linearly mollified cutoff of the
 indicator function ${\bf 1}(k\in \II)$. Set
\begin{align}\label{def:wtg}
 \wt\gamma_k\deq \iota_k\, z_k+(1-\iota_k)\,y_k(T_1)\,,\qquad\qquad k\in I^c\,,
\end{align}
where the external points $\bz\equiv (z_k)\in\digamma^{(N-\cK)} $ in~\eqref{def:wtg} will be chosen in Subsection~\ref{subsubsection:aux beta} below. Note that $\widetilde\gamma_k=z_k$, for$ |L-k|\le K^5$;$\widetilde\gamma_{k}=y_k(T_1)=\overline{y}_k(T_1)$, for $|L-k|\ge K^5+N^{\chi}K^4$. Thus the sequence~$\widetilde\gamma$ smoothly interpolates between the external points~$\by(T_1)$ from the DBM and~$\bz$. The external points~$\bz$ are constructed from an appropriate $\beta$-ensemble whose equilibrium density has a single interval support. This will guarantee rigidity; in particular, 
 \begin{align}\label{rigidity deterministic z}
 |z_{k+1}-z_{k}|\le C \frac{N^\xi}{N^{2/3}\check k^{1/3}}\,,
 \end{align} 
 with $\check k=\min\{k,N-k +1\}$, for all $k\in\llbracket 1,N-1\rrbracket$; see~\eqref{le rigidity for zk and gammaauxk} below.

Anticipating the precise choice of $\bz$, we mention that they are chosen such that
\begin{align}
z_{L-K-1}=y_{L-K-1}(T_1) \,,\qquad z_{L+K+1}=y_{L+K+1}(T_1)\,.
\end{align}
In fact, this choice will assure
that the configuration interval of the localized measures, both with~$\by(T_1)$ and with~$\wt{\boldsymbol{\gamma}}$ as external
points, will have the same (and time-independent) support 
\begin{align}\label{le boldsymbolJbz}
\boldsymbol{J}_{\boldsymbol{y}(T_1)}=\boldsymbol{J}_\bz = [z_-, z_+]\,,
\end{align}
where $z_- \equiv z_{L-K-1}$, $z_+ \equiv z_{L+K+1}$. We next estimate the size of the interval $\boldsymbol{J}_\bz$.
\begin{lemma}\label{le lemma le estimate on lenght of interval bz}
 Let $\boldsymbol{J}_\bz$ be as in~\eqref{le boldsymbolJbz} and assume that $K$ satisfies~\eqref{le choice of K}. Then, we have
 \begin{align}\label{le estimate on lenght of interval bz}
  |\boldsymbol{J}_\bz|=\frac{\mathcal{K}}{N\varrho_{T_1}(\cz)}+O\left(\frac{N^\xi}{N} \right)\,,
 \end{align}
on $\cG_{T_1}^{1}$, where $\cz\deq (z_{L+K+1}-z_{L-K-1})/2 $. 
\end{lemma}
\begin{proof}
 We mainly follow the proof of Lemma~4.5 in~\cite{EYSG}. First, we write, by~\eqref{le boldsymbolJbz},
 \begin{align*}
  |\boldsymbol{J}_\bz|=z_+-z_-=y_{L+K+1}(T_1)-y_{L-K-1}(T_1)=\gamma_{L+K+1}(T_1)-\gamma_{L-K-1}(T_1)+O(N^{-1+\xi})\,,
 \end{align*}
where we used that $\by(T_1)\in\cG_{T_1}^{1}$. Next, we note that by Assumption~$(1)$ of Theorem~\ref{main theorem} we have
$ \varrho_{T_1}(x)=\varrho_{T_1}(\overline z)+O\left(N^\delta |x-\overline z|\right)$.
Thus from~\eqref{the gammas},
\begin{align*}
 \mathcal{K}&=N\int_{\gamma_{L-K-1}(T_1)}^{\gamma_{L+K+1}(T_1)}\varrho_{T_1}(x)\,\dd x
 =N\varrho_{T_1}(y)|\boldsymbol{J}_\bz|+O(N^{1+\delta}|\boldsymbol{J}_\bz|^2)+O(N^{\xi})\,,
\end{align*}
where we used that $\varrho_{T_1}\sim 1$. Since ${K}N^{\delta}\ll N$, by~\eqref{le choice of K}, we get~\eqref{le estimate on lenght of interval bz}.
\end{proof}

\subsection{Localizing the DBM and the reference measure $\omega_{T_1}$}
Having $\overline\bY\in\overline\cG$ fixed, we consider the DBM on the $\overline\bx$-variables given by the stochastic differential equation (SDE)
\begin{multline}\label{DBMx}
  \rd \overline{x}_i(t) = \sqrt{\frac{2}{\beta N}} \,\rd B_i(t)-\upsilon_L\,\dd t - \frac{1}{N} \sum_{\substack{j\in I \\ j\not=i}} \frac{1}{\overline{x}_j(t)-\overline{x}_i(t)}\, \rd t\\ -\frac{1}{N} \sum_{k\not\in I} \frac{1}{\overline{y}_k(t)-\overline{x}_i(t)}\, \rd t -\frac{\overline{x}_i(t)+\upsilon_L(t-T_1)}{2}\,\rd t\,,
\end{multline}
$i\in I$, $t\ge T_1$, with $(B_i)_{i\in I}$ a collection of independent standard Brownian motions. We let $\P^{\ol\bY}$ denote the associated path space measure.

For $t\ge T_1$, we define an approximate coupled dynamics, $\wt\bx(t)$ by letting 
\begin{multline}\label{DBMwtx}
  \rd \wt x_i(t) = \sqrt{\frac{2}{\beta N}} \,\rd B_i(t) -\upsilon_L\,\dd t-\frac{1}{N}
\sum_{\substack{j\in I \\ j\not=i}} \frac{1}{\wt x_j(t)-\wt x_i(t)}\,\rd t  -\frac{1}{N}\sum_{k\not\in I} \frac{1}{\wt \gamma_k-\wt x_i(t)}\,\rd t   -\frac{\widetilde x_i(t)}{2}\,\rd t\,, 
\end{multline}
$i\in I$, with initial condition~$\wt{\bx}(T_1)=\ol\bx(T_1)$. The corresponding path space measure is denoted by~$\wt \P^{\ol\bY}$. Going from~\eqref{DBMx} to~\eqref{DBMwtx} we replaced the time-dependent external points $\by(t)$ by the time-independent reference points~$\wt{\boldsymbol{\gamma}}$ and we neglected the drift term~$\upsilon_L(t-T_1)\dd t/2$. Note that the Brownian motions in~\eqref{DBMx} and~\eqref{DBMwtx} are the same.

We define the local ``reference'' measure
\begin{align}\label{def:om}
   \om_{T_1}({\bx})\,\dd\bx\deq \frac{1}{Z_{T_1}}\euler{-\beta N \cH_{T_1}(\bx)}\,\dd\bx\,, \qquad \cH_{T_1}(\bx)\deq \sum_{i\in I} V^{\wt{\boldsymbol{\gamma}}}(x_i) - \frac{1}{N}\sum_{\substack{i,j\in I \\ i<j}} \log (x_j-x_i)\,,
\end{align}
where the external potential is given by
\begin{align}\label{le external potential for omega}
  V^{\wt{\boldsymbol{\gamma}}}(x) \deq
   { \frac{1}{2}}V(x)-\frac{1}{N}\sum_{k\not\in I} \log |x-\wt\gamma_k| \,, \quad\qquad V(x) =\frac{x^2}{2}+2\upsilon_L x\,.
\end{align}
The subscript $T_1$ in $\om_{T_1}$  indicates that 
the external points $\wt{\boldsymbol{\gamma}}$ in the construction of this measure
were obtained in~\eqref{def:wtg} by matching the external points $\by(T_1)$ of the original DBM at time $T_1$.
Note that this measure as well as the measure  $\widetilde\P^{\ol \bY}$ are supported on the fixed configuration interval 
\begin{align*}
\boldsymbol{J}_{\by(T_1)}=[y_{L-K-1}(T_1),y_{L+K+1}(T_1)]\,.
\end{align*}
The measure $\om_{T_1}$ is the equilibrium measure of 
 the SDE~\eqref{DBMwtx}.

We write the distribution of  $\wt\bx(t)$  as  $g_t\,\om_{T_1}$ (for $t\ge T_1$). Since they are supported on the same configuration interval,
 the measures $g_t\,\omega_{T_1}$ (for $t> T_1$) and $\omega_{T_1}$ are both absolutely continuous with respect to the Lebesgue measure, hence also to each other.

\subsubsection{Entropy bound}\label{sec:ent}
In this section, we compare the measures $\om_{T_1}$ and $g_t\,\om_{T_1}$ for $t\ge T_1$. We show that the process $(\wt\bx(t))$ equilibrates on a time scale $\sim K/N$,
\ie the local statistics of $g_t\,\om_{T_1}$ and $\om_{T_1}$ are very close beyond times $t\ge T_1'\deq  T_1 + K(K/N)$, with $t\le t_2$.

Since $\om_{T_1}$ is supported on an interval of size $O(K/N)$ (see Lemma~\ref{le lemma le estimate on lenght of interval bz}), the Hessian of its Hamiltonian~$\cH_{T_1}$ from~\eqref{def:om}  satisfies
\begin{align}\label{lower bound on hessian 1}
    \cH_{T_1}''(\bx) &\ge \min_{i\in I}\frac{1}{N} \sum_{k\not\in I} \frac{1}{( x_i- \wt\gamma_k)^2}\ge  \min_{i\in I}\frac{1}{N} \sum_{k\,:\, K<|k-L|\le K^5} \frac{1}{(x_i- z_k)^2}\ge \frac{cN}{K}\,,
\end{align}
for all $\bx\in(\boldsymbol{J}_{\by(T_1)})^{\cK}\cap\digamma^{(\cK)}$, where we used~\eqref{rigidity deterministic z}. Thus, recalling the discussion in Section~\ref{subsec:rel entro},  $\om_{T_1}$ satisfies the logarithmic Sobolev inequality
\begin{align}\label{LSI}
    S_{\om_{T_1}}(f) \le \frac{CK}{N} D_{\om_{T_1} }(\sqrt{f})\,;
\end{align}
\cf~\eqref{first LSI}. Further, for $t\ge T_1+\frac{1}{2}K(K/N)$ the process $(\wt\bx(t))$ has become absolutely continuous with respect to Lebesgue measure, and  one can easily prove that
\begin{align}\label{inent}
    S(g_t\,\om_{T_1}| \om_{T_1}) \le N^C\,,\qquad \qquad t\ge T_1+\frac{1}{2}K(K/N)\,;
\end{align}
for some large $C$; see, \eg Lemma~4.7 in~\cite{EKYY2}. Therefore, running the Bakry-\'{E}mery argument of Subsection~\ref{subsec:rel entro} from time $T_1+\frac{1}{2}K(K/N)$ to time $T_1'= T_1+ K(K/N)$ and using the initial entropy estimate~\eqref{inent}, we immediately get the following result.

\begin{lemma}\label{lm:SD} For any $t\ge T_1'=T_1+K(K/N)$, we have
\begin{align}\label{SD}
  D_{\om_{T_1}}(\sqrt{g_t}) +  S_{\om_{T_1}} (g_t) \le \euler{- c  K}\,, \qquad\qquad t\in[T_1',t_2]\,,
\end{align}
for some $c>0$. In particular, the statistics of $\wt \bx(t)$ for any $t\in[T_1',t_2]$ are the same as the
 statistics of the local equilibrium measure $\om_{T_1}$ as follows from
 \begin{align}
  \left|\int O\,(g_t-1)\dd\om_{T_1}\right|\le \|O\|_\infty \sqrt{2S_{\om_{T_1}}(g_t)}\le C\euler{-cK/2}\,,
 \end{align}
for any bounded observable $O$.
\end{lemma}

\subsubsection{Construction of an auxiliary $\beta$-ensemble}\label{subsubsection:aux beta} We now turn to the choice of the reference points $\bz$ introduced first at the beginning of Subsection~\ref{subsection the reference points}. We construct a global $\beta$-ensemble, $\mu_{\aux}$, with potential $V_\aux$ and equilibrium density $\varrho_{\aux}$ such that it has a single interval support and such that the density matches with $\varrho_{T_1}$ at~$\gamma_L(T_1)$. The main properties of $\mu_\aux$ are summarized in the next lemma.
\begin{lemma}\label{lemma: construction of the auxilary measure}
 There exists a $\beta$-ensemble $\mu_\aux\equiv \mu_{\aux}^{(N)}$, with quadratic potential $V_\aux$ and equilibrium density $\varrho_\aux$, and a set of external configurations $\bz\in\digamma^{(N-\cK)}$, with $z_{L-K-1}=y_{L-K-1}(T_1)$, $z_{L+K+1}=y_{L+K+1}(T_1)$, such that the following holds for $N$ sufficiently large.
 \begin{itemize}[noitemsep,topsep=0pt,partopsep=0pt,parsep=0pt]
  \item[$(1)$] The limiting equilibrium density $\varrho_\aux$ of $\mu_\aux$ is a shifted semicircle law with finite variance satisfying, for any $\xi>0$,
\begin{align}\label{le almost matching densities}
 \varrho_\aux(y)=\varrho_{T_1}(y)+O(N^\xi/K)+O(N^{\delta}|y-z_{L-K-1}|)\,,\qquad y\in\R\,.
\end{align}
\item[$(2)$] The external points $\bz$ satisfy, for any $\xi>0$,
  \begin{align}\label{le rigidity for zk and gammaauxk}
   |z_k-\gamma_{\aux,k}|\le C\frac{N^\xi}{N^{2/3}\check k}\,,\qquad\qquad k\in I^c\,,
  \end{align}
where $(\gamma_{\aux,k})$ are the quantiles of the equilibrium density $\varrho_\aux$, \ie $\int_{-\infty}^{\gamma_{\aux,k}}\dd\varrho_\aux =k/N$, and $\check k=\min\{k,N-k+1\}$. In particular, since $V_\aux$ is ``regular'', the rigidity estimate~\eqref{rigidity deterministic z} holds.

  \item[$(3)$] The localized measure $\mu_\aux^\bz$ satisfies, for any $\xi>0$,
  \begin{align}
   \P^{\mu_\aux^\bz}(|x_i-\alpha_i|\ge N^{\xi}/N\,,\forall i\in I)\le CN^{\xi}\frac{K}{N}\,,
  \end{align}
where $(\alpha_i)$ are $\cK$ equidistant points in $\boldsymbol{J}_\bz=[z_{L-K-1},z_{L+K+1}]=\boldsymbol{J}_{\by(T_1)}$, \ie
\begin{align}
 \alpha_i=\cz+\frac{i-L}{2K+1}|\boldsymbol{J}_{\bz} |\,,\qquad\qquad \cz=\frac{1}{2}(z_{L-K-1}+z_{L+K+1})\,.
\end{align}

 \end{itemize}

\end{lemma}
\begin{proof}
The proof is split into three steps. {\it Step 1:} We introduce the quadratic potential~$V_{\Gauss(\varsigma)}(x)\deq x^2/2\varsigma^2$, with some $\varsigma>0$, and consider the~$\beta$-ensemble,~$\mu_{\Gauss(\varsigma)}$, with Hamiltonian
\begin{align*}
 \cH_{\Gauss(\varsigma)}\deq\frac{1}{2}\sum_{i=1}^N V_{\Gauss(\varsigma)}(\lambda_i)-\frac{1}{N}\sum_{\substack{i,j=1 \\i<j}}^N\log |\lambda_j-\lambda_i|\,.
\end{align*}
It is easy to check that the limiting equilibrium density, $\varrho_{\Gauss(\varsigma)}$, of $\mu_{\Gauss(\varsigma)}$ satisfies $\varrho_{\Gauss(\varsigma)}(x)=\varrho_{sc}(x/\varsigma)/\varsigma$, with $\varrho_{sc}(x)=\frac{2}{\pi}\sqrt{(4-x^2)_+}$ the standard semicircle law. Similarly, the quantiles, $(\gamma_{\Gauss(\varsigma),i})$, of $\varrho_{\Gauss(\varsigma)}$ satisfy $\gamma_{\Gauss(\varsigma),i}=\varsigma\gamma_{sc,i}$, where $\gamma_{sc,i}$ denote the quantiles with respect the standard semicircle law, \ie $\int_{-\infty}^{\gamma_{sc,i}}\dd\varrho_{sc}=i/N$. Thus $\varrho_{\Gauss(\varsigma)}(\gamma_{\Gauss(\varsigma),i })=\varrho_{sc}(\gamma_{sc,i})/\varsigma$. In particular, we can fix $\varsigma$ such that $\varrho_{T_1}(\gamma_{L}(T_1))=\varrho_{\Gauss(\varsigma)}(\gamma_{\Gauss(\varsigma),L })$, \ie we set
\begin{align*}
 \varsigma\deq\frac{\varrho_{sc}(\gamma_{sc,L})}{\varrho_{T_1}(\gamma_{L}(T_1))}\,.
\end{align*}
We next choose boundary conditions $\tilde\by$ with the following properties: $(1)$ For any $\xi>0$,
\begin{align}\label{rigz aux pre}
 |\tilde{y}_k-\gamma_{sc,k}|\ge N^{\xi}/N\,,\qquad\qquad \forall k\in I^c \,,
\end{align}
for $N\ge N_0(\xi)$, (\ie $\tilde\by$ are rigid in the sense of sense of~\eqref{rigidity for b ensemble one} with $V=V_{\Gauss(\varsigma)}$); $(2)$ for any $\xi>0$, there are  $c_0',c_1'>0$ such that
\begin{align}\label{rigz aux}
    \P^{\mu_{\Gauss(\varsigma)}^{\tilde\by}} ( |x_i-\tilde\al_i|\ge N^\xi/N\,,\forall i\in I) \le \euler{-c_0'N^{c_1'}}\,,\qquad 
\end{align}
where $\tilde\al_i$ are the $\mathcal{K}$ equidistant points in the configuration interval $\boldsymbol{J}_{\tilde\by}=[\tilde y_{L-K-1},\tilde y_{L+K+1}]$. The precise choice of $\tilde\by$ is unimportant for our argument, as long as $\tilde\by$ satisfy~\eqref{rigz aux pre} and~\eqref{rigz aux}. That we can choose a $\tilde\by$ such that~\eqref{rigz aux pre} and~\eqref{rigz aux} are satisfied follows from Proposition~\ref{rigidity for time dependent beta one} and an application of Markov's inequality.

{\it Step 2:}
The length of the configuration intervals $\boldsymbol{J}_{\by(T_1)}$ and $\boldsymbol{J}_{\tilde\by}$ may differ slightly. Using the scale invariance of the Gaussian measure, we now adjust~$\varsigma$ and~$\tilde\by$ to guarantee that the lengths of the configuration intervals agree: Following the proof of Lemma~\ref{le lemma le estimate on lenght of interval bz} or the proof of Lemma~4.5 in~\cite{EYSG}, we get from the rigidity estimates for $\mu_{\Gauss(\varsigma)}$ that
\begin{align*}
 |\boldsymbol{J}_{\tilde\by}|&=|\tilde y_{L-K-1}-\tilde y_{L+K+1}|=\frac{\cK}{N\varrho_{\Gauss(\varsigma)}(\gamma_{\Gauss(\varsigma),L})}+O(N^{-1}N^{\xi})\,,
 \end{align*}
 and from Lemma~\ref{le lemma le estimate on lenght of interval bz} that
 \begin{align*}
 |\boldsymbol{J}_{\by(T_1)}|&=|y_{L-K-1}(T_1)-y_{L+K+1}(T_1)|=\frac{\cK}{N\varrho_{T_1}(\gamma_{L})}+O(N^{-1}N^{\xi})\,.
\end{align*}
Using that $\varrho_{T_1}(\gamma_L)=\varrho_{\Gauss(\varsigma)}(\gamma_{\Gauss(\varsigma),L})$, by our choice of $\varsigma$, we hence conclude that
\begin{align}\label{le s}
s\deq\frac{|\boldsymbol{J}_{\tilde\by}|}{|\boldsymbol{J}_{\by(T_1)}|}=1+O(N^\xi K^{-1})\,.
\end{align}
Setting $\tilde\bz\deq \tilde\by/s$ we have $|\boldsymbol{J}_{\tilde\bz}|=|\boldsymbol{J}_{\by(T_1)}|$ and
\begin{align}\label{le comparision of gaussian ps}
 \P^{\mu_{\Gauss(\varsigma)}^{\tilde\bz}} ( |x_i-\tilde\al_i/s|\ge N^\xi/N\,,\forall i\in I)= \P^{\mu_{\Gauss(\varsigma')}^{\tilde\by}} ( |x_i-\tilde\al_i|\ge s N^\xi/N\,,\forall i\in I)\,,
\end{align}
where we have set $\varsigma'\deq s\varsigma$. Using the rigidity of $\wt{\by}$ we get, similarly to~\eqref{lower bound on hessian 1}, that
 $\nabla^2_{\bx}\cH^{\wt\by}(\bx)\ge \frac{cN}{K}$,
for all $\bx\in(\boldsymbol{J}_{\widetilde{\by}})^\cK\cap \digamma^{\cK}$. Thus the logarithmic Sobolev inequality 
\begin{align*}
 S(\mu_{\Gauss(\varsigma')}^{\tilde\by}|\mu_{\Gauss(\varsigma)}^{\tilde\by})\le\frac{CK}{N}D(\mu_{\Gauss(\varsigma')}^{\tilde\by}|\mu_{\Gauss(\varsigma)}^{\tilde\by})\,,
\end{align*}
with the local Dirichlet form
\begin{align*}
 D(\mu_{\Gauss(\varsigma')}^{\tilde\by}|\mu_{\Gauss(\varsigma)}^{\tilde\by})\deq\frac{1}{\beta N}\sum_{i\in I}\int \left(\partial_i \left(\frac{\rd \mu_{\Gauss(\varsigma')}^{\tilde\by}}{\rd \mu_{\Gauss(\varsigma)}^{\tilde\by}}\right)(\bx) \right)^2\,\rd \mu_{\Gauss(\varsigma)}^{\tilde\by}(\bx)
\end{align*}
holds. A straightforward calculation together with~\eqref{le s} then shows that
\begin{align*}
 S(\mu_{\Gauss(\varsigma')}^{\tilde\by}|\mu_{\Gauss(\varsigma)}^{\tilde\by})&\le \frac{CK}{N^2}\sum_{i\in I}\int\left|\partial_i\euler{-\beta N\sum_{j\in I}(V_{\Gauss(\varsigma')}(x_j)-V_{\Gauss(\varsigma)}(x_j)}\right|^2\,\rd\mu_{\Gauss(\varsigma)}^{\tilde\by}\le C\frac{N^{2\xi}K^{2}}{N^2}\,.
\end{align*}
Thus, using first~\eqref{le comparision of gaussian ps} and then the entropy inequality~\eqref{entropyneq}, we get
\begin{align}\label{rigz pre}
 \P^{\mu_{\Gauss(\varsigma)}^{\tilde\bz}}\left( |x_i-\tilde\al_i/s|\ge N^\xi/N\,, \forall i\in I\right)&=\P^{\mu_{\Gauss(\varsigma')}^{\tilde\by}} \left( |x_i-\tilde\al_i|\ge s N^\xi/N\,,\forall i\in I\right)\nonumber\\
&\le    \P^{\mu_{\Gauss(\varsigma)}^{\tilde\by}} \left( |x_i-\tilde\al_i|\ge sN^\xi/N\,, \forall i\in I\right)+\sqrt{2S(\mu_{\Gauss(\varsigma')}^{\tilde\by}|\mu_{\Gauss(\varsigma)}^{\tilde\by}) }\nonumber\\
 &\le C\euler{-N^{\xi}}+C\frac{N^{\xi}K}{N}\,,
\end{align}
where we used~\eqref{rigz aux} (with an additional factor $s$) to get the last line.

{\it Step 3:} Finally, we achieve that $ \boldsymbol{J}_{\tilde\bz}=\boldsymbol{J}_{\by(T_1)}$  by a simple shift in the energy: we replace $V_{\Gauss(\varsigma')}(x)$ by $V_{\Gauss(\varsigma')}	(x-b)$, $b\deq y_{L-K-1}(T_1)-\tilde y_{L-K-1}$, $x\in\R$. We now choose $\mu_\aux$ as the Gaussian measure defined by the potential $V_{\Gauss(\varsigma')}(\cdot-b)$ and we set  $z_i\deq \tilde z_i-b$, for $i\in\llbracket1, N\rrbracket$. With these choices,~\eqref{rigz pre} asserts that
\begin{align}\label{rigz}
 \P^{\mu_{\aux}^{\bz}}\left(|x_i-\alpha_i|\ge N^{\xi}/N\,, \forall i\in I\right)\le C \frac{N^{\xi}K}{N}\,,
\end{align}
where $\alpha_i$ are the $\mathcal{K}$ equidistant points in the interval $\boldsymbol{J}_\bz=\boldsymbol{J}_{\by(T_1)}=[y_{L-K-1}(T_1),y_{L+K+1}(T_1)]$. 

In sum, we have established the following. We consider the $\beta$-ensemble $\mu_\aux$ with quadratic potential, whose equilibrium density $\varrho_\aux$ is a semicircle law with radius $\sqrt{2}\varsigma'$ which is centered at $b$. Taylor expanding the densities $\varrho_{T_1}$ and $\varrho_\aux$ around $y_{L-K-1}(T_1)$ and recalling~\eqref{le s} as well as~\eqref{assumption on regularity of varrho}, we obtain~\eqref{le almost matching densities}. This proves statement $(1)$ of Lemma~\ref{lemma: construction of the auxilary measure}. The points $\boldsymbol{z}=(z_i)$ are rigid as follows from~\eqref{rigz aux pre} and the choices $z_i=\widetilde{y}_i/s$, $z_i=\widetilde z_i-b$, $i\in\N_N$. This immediately implies statement $(2)$ of Lemma~\ref{lemma: construction of the auxilary measure}. Finally, the rigidity statement $(3)$ of Lemma~\ref{lemma: construction of the auxilary measure} for the localized measure $\mu_\aux^{\by}$ was obtained in~\eqref{rigz}. This concludes the proof of Lemma~\ref{lemma: 
construction of the auxilary measure}.
\end{proof}
 
 We conclude this subsection with a straightforward technical result that will be used in the next section. Recall the definition of the interval of integers $I,\IO$ and $\II$ in~\eqref{le definition of Is}.
 \begin{corollary}\label{corollary almost quantiles}
  Let $\bz$ be as in Lemma~\ref{lemma: construction of the auxilary measure} and let $\widetilde{\boldsymbol{\gamma}}$ be defined as in~\eqref{def:wtg}. Then, for any $\xi>0$,
  \begin{align}\label{le almost quantiles}
   |\wt\gamma_k-\gamma_{k}(T_1)|\le C\frac{N^\xi}{N}+C\cdot{\bf 1}(k\in \IO) \left(\frac{N^{\xi}|\wt\gamma_k-\wt\gamma_L|}{K}+{N^{\delta}|\wt\gamma_k-\wt\gamma_L|^2}\right)\,,\qquad k\in \II\backslash I\,,
  \end{align}
for $N$ sufficiently large. Moreover, we have, for any $\xi>0$,
\begin{align}\label{le almost quantiles zwei}
 \widetilde\gamma_k-\widetilde\gamma_{k-1}\le N^{\xi}/N\,,\qquad\quad k\in \II\backslash I\,,
\end{align}
for $N$ sufficiently large.
 \end{corollary}
\begin{proof}
 Recall that $\wt\gamma_k=y_k(T_1)$ for $k\not\in \IO$. Since $\by(T_1)\in\cG_{T_1}^1$ we immediately get
 \begin{align}\label{le almost quantiles 1}
   |\wt\gamma_k-\gamma_{k}(T_1)|\le C\frac{N^\xi}{N}\,,
 \end{align}
 for $k\in \II\backslash \IO$. Next, assume first that $K+1\le L-k\le K^5$. Then we have $\wt\gamma_k=z_k$, and we can write
 \begin{align*}
 \int_{\wt\gamma_{k}}^{y_{K-L-1}(T_1)}\varrho_\aux(y)\dd y=\frac{|L-K-1-k|}{N}+O\left(\frac{N^{\xi}}{N} \right)\,,
 \end{align*}
where we used $y_{K-L-1}(T_1)=z_{K-L-1}$, the rigidity estimate in~\eqref{le rigidity for zk and gammaauxk} and the fact that~$(\gamma_{\aux,k})$ are the quantiles of~$\varrho_\aux$. Using~\eqref{le almost matching densities}, we hence can write
\begin{align*}
 \int_{\wt\gamma_{k}}^{y_{K-L-1}(T_1)}\varrho_{T_1}(y)\,\dd y&=\frac{|L-K-1-k|}{N}+O\left(\frac{N^{\xi}}{N} \right)\nonumber\\
 &\qquad+O\left(\frac{N^{\xi}|\wt\gamma_k-y_{L-K-1}(T_1)|}{K} \right)+O\left(N^\delta|\wt\gamma_k-y_{L-K-1}(T_1)|^2 \right)\,.
\end{align*}
On the other hand, since $\by(T_1)\in\cG_{T_1}^1$, \ie $|y_{L-K-1}(T_1)- \gamma_{L-K-1}(T_1)|\le CN^{\xi}/N$, and using that~$\gamma_{k}(T_1)$ are the quantiles with respect to $\varrho_{T_1}$, we have
\begin{align*}
 \int_{\gamma_{k}(T_1)}^{y_{K-L-1}(T_1)}\varrho_{T_1}(y)\,\dd y=\frac{|L-K-1-k|}{N}+O\left(\frac{N^{\xi}}{N} \right)\,.
\end{align*}
Comparing these last two equations and using the lower bound on the density~$\varrho_{T_1}$, we conclude that
\begin{align}\label{le almost quantiles 2}
  |z_k-\gamma_{k}(T_1)|\le C\frac{N^{\xi}}{N}+C\frac{N^{\xi}|\wt\gamma_k-\wt\gamma_L|}{K}+C{N^{\delta}|\wt\gamma_k-\wt\gamma_L|^2}\,,
\end{align}
for $k$ such that  $K+1\le L-k\le K^5$. Here, we also used that $\wt\gamma_L -y_{L-K-1}(T_1)\le CK/N$. The same argument applies to the case $K+1\le k-L\le K^5$.

It remains to consider the transition regime $K^5\le|L-k|\le K^5+N^{\chi}K^4$. Using the definition of~$\wt{\boldsymbol{\gamma}}$ in~\eqref{def:wtg}, we can estimate
\begin{align*}
 |\widetilde\gamma_k-\gamma_{k}(T_1) |&\le \iota_k|z_k-\gamma_{k}(T_1)| +(1-\iota_k)|y_k(T_1)-\gamma_{k}(T_1)|\nonumber\\ &\le C\frac{N^{\xi}}{N}+C\frac{N^{\xi}|\wt\gamma_k-\wt\gamma_L|}{K}+C{N^\delta|\wt\gamma_k-\wt\gamma_L|^2}\,,
\end{align*}
for such~$k$, where we used~\eqref{le almost quantiles 1},~\eqref{le almost quantiles 2} and the rigidity of $\by(T_1)\in \cG^{1}_{T_1}$. This proves~\eqref{le almost quantiles}.

The estimate~\eqref{le almost quantiles zwei} follows directly from the rigidity of $\by(T_1)$, $\bz$ and~\eqref{le almost quantiles}.
\end{proof}

\subsection{Three measures and their properties}\label{subsection:three}
Having $\ol\bY\in\ol\cG$ fixed and having constructed the external points $\bz$, we have, up to this point, introduced three distinct measures on the internal particles: 

\begin{itemize}[noitemsep,topsep=0pt,partopsep=0pt,parsep=0pt]
\item[$(1)$] $\om_{T_1}$ is given by an explicit formula in~\eqref{def:om}. It is  a
 local  $\beta$-ensemble on  $\boldsymbol{J}_\bz$ which we refer to as the ``reference'' measure.
\item[$(2)$]  $g_t\,\om_{T_1}$ is the distribution of $\wt\bx(t)$ from the dynamics~\eqref{DBMwtx} on  $\boldsymbol{J}_\bz$.
\item[$(3)$]   $\P^{\ol\by(t)}$ is the measure of the $\ol\bx(t)$ dynamics~\eqref{DBMx}
at time $t$, it is also the
conditional measure $\P^{\ol\bY}$ of the original measure $\P$, conditioned
on the $\ol\bY$-trajectory at time $t\ge t_1$. This measure is also on $\cK$ particles, but
now the configuration
interval is time-dependent $\boldsymbol{J}_{\ol\by(t)}\deq [\ol y_{L-K-1}(t), \ol y_ {L+K+1}(t)]$.

\end{itemize}
In the remainder of this subsection, we establish rigidity for the measures $\om_{T_1}$ and~$\g_t\,\om_{T_1} $:

\begin{definition}\label{definition of regularity and rigidity} We say that the measure $\nu$ (on $\cK$-point configurations labeled with $I$, $|I|=\mathcal{K}$, in a fixed interval $\boldsymbol{J}$) is
rigid with exponent~$\xi$ if
\begin{align}
   \nu ( |\bx_i-\al_i|> N^{\xi}/N\,, \; \forall i\in I)\le C\euler{-cN^{\xi}}\,,
\end{align}
where $\al_i$ are the $\mathcal{K}$ equidistant points in $\boldsymbol{J}$ and where $c>0$. The path-space measure $Q$ for times $[T_1, t_2]$ on the same
configuration interval $\boldsymbol{J}$  is
rigid with exponent~$\xi$ if
\begin{align}
   Q ( \sup_{s\in[T_1, t_2]} |\bx_i(s)-\al_i|>N^{\xi}/N\,, \; \forall i\in I)\le C \euler{-cN^{\xi}}\,.
\end{align}
\end{definition}

Note that if for all $t$ the fixed-time marginals $Q_t$ of a space time
measure $Q$ satisfy rigidity, then~$Q$ satisfies rigidity
(since the trajectories typically have some mild continuity; see~Section~9.3 of~\cite{EYSG}). 

We will establish the following main technical input. Recall that $T_1'=T_1+K(K/N)$.

\begin{proposition} \label{lm:rigreg}
Let $\xi>0$ be sufficiently small and let $K$ satisfy~\eqref{le choice of K}. Then, for any $\ol\bY\in \ol\cG$ and any $t\in [T_1, t_2]$ the following holds.
\begin{itemize}[noitemsep,topsep=0pt,partopsep=0pt,parsep=0pt]
\item[$(1)$]  $\om_{T_1}$ (\ie the local ``reference'' measure) is rigid with exponent~$\xi$ and satisfies
\begin{align}\label{regi}
  \max_{i\in I}\E^{\om_{T_1}} \frac{1}{[N|x_i-x_{i\pm1}|]^p}\le C_p\,,
\end{align}
(with $x_{i\pm 1}(t)=y_{L\pm{(K+1)}}(T_1)$ if $i=L\pm K$), for any $p<2$.
\item[$(2)$]  $g_t\, \om_{T_1}$  (\ie the time marginals of the process $\widetilde{\boldsymbol{X}}=\{\widetilde{\bx}(s)\,:\,s\in[T_1',t_2]\}$) is rigid with exponent~$\xi$,
moreover, the whole process $\{\wt \bx(s) \; : \; s\in [T_1', t_2]\}$ with measure $\wt\P^{\ol\bY}$ is rigid with
 exponents $\xi$. Furthermore,
\begin{align}\label{regii}
  \max_{i\in I} \E^{g_t\,\om_{T_1}} 
  \frac{1}{[N|\wt x_i(t)-\wt x_{i\pm 1}(t)|]^p} \le
  C_{p}\,,
\end{align}
(with $\tilde x_{i\pm 1}(t)=y_{L\pm{(K+1)}}(T_1)$ if $i=L\pm L$), for any $p<2$ and $t\ge T_1'$.

\end{itemize} 
\end{proposition}

To simplify the exposition, we split the proof of Proposition~\ref{lm:rigreg} according to its statements. 

\subsubsection{Proof of statement $(1)$ of Proposition~\ref{lm:rigreg}}
We start with the rigidity of the reference measure~$\omega_{T_1}$. For notational simplicity, we write in the following
\begin{align*}
 \gamma_{k}\equiv\gamma_{\rl(k)}{(T_1)}\,,\qquad\qquad k\in\II\,,
\end{align*}
where the labeling $\rl$ is chosen according to~\eqref{le Cgs1}.

\begin{proof}[Proof of rigidity of $\omega_{T_1}$]
We first recall the following general result of Theorem~4.2 (see also the remark after Lemma~4.5) of~\cite{EYSG}. For any local equilibrium measure~$\mu^\by$ on $\cK$ points with potential~$V^\by$  on 
an interval~$\boldsymbol{J}$ of size~$|\boldsymbol{J}|\sim K/N$ rigidity (with exponent $\xi>0$) in the sense of Definition~\ref{definition of regularity and rigidity} holds if the following two conditions are satisfied:
\begin{align}\label{Vder}
    (V^\by)'(x) = \varrho(\cy)\log \frac{d_+(x)}{d_-(x)} + O\Big( \frac{N^{c\xi}}{N d(x)}\Big)\,,
\end{align}
and
\begin{align}\label{Ex}
  \Big|\E^{\mu^\by} x_i -\al_i \Big|\le N^{c\xi}/N\,,\qquad\qquad \forall i\in I\,,
\end{align}
where $\cy$ is the midpoint of the interval $\boldsymbol{J}$, $d(x)$ is the distance of $x$ to the boundary of $\boldsymbol{J}$, 
\begin{align}\label{le dpm}
d_-(x)\deq d(x)+\varrho(\cy) N^{\xi}/N\,,\qquad d_+(x)\deq\max\{|x-y_{L-K-1}|,|x-y_{L+K+1} |\}+\varrho(\cy)N^{\xi}/N\,,
\end{align}
 and $\al_i$ are the~$\mathcal{K}$ equidistant points in~$\boldsymbol{J}$.

We now apply this result with the choices $\by=\wt{\boldsymbol{\gamma}}$ and $\boldsymbol{J}=\boldsymbol{J}_\bz=\boldsymbol{J}_{\by(T_1)}$ to the reference measure~$\omega_{T_1}$. The condition~\eqref{Vder} will  follow from the 
global condition~\eqref{le once more the vL} below and from the fact that the reference points $\wt{\boldsymbol{\gamma}}$ are rigid in the sense of Corollary~\ref{corollary almost quantiles}. The details are as follows.

To check condition~\eqref{Vder}, we introduce the supplemental potential~$\wt V^{\wt{\boldsymbol\gamma}}$ by setting
\begin{align}\label{le tilde V}
 \wt V^{\wt{\boldsymbol\gamma}}(x)\deq\frac{1}{2}x^2+2\upsilon_L x-\frac{1}{N}\sum_{k\,:\, |k-L|\ge K+N^{\xi}}\log |x-\widetilde\gamma_k|\,.
 \end{align}
We then have
\begin{align}\label{le bound difference V wtV}
\left|{( V^{\wt{\boldsymbol\gamma}})}'(x)- {(\wt V^{\wt{\boldsymbol\gamma}})}'(x)\right|\le\frac{1}{N}\sum_{k\,:\,K<|k-L|<K+N^{\xi}}\frac{1}{|\wt \gamma_k-x|}\le\frac{N^{\xi}}{Nd(x)}\,,\qquad\qquad x\in\boldsymbol{J}_\bz\,.
\end{align}

To control~$\wt V^{\wt{\boldsymbol\gamma}}$, we can follow Appendix~A of~\cite{EYSG}. First, recall from Lemma~\ref{le lemma for new gamma} that $\upsilon_L$ satisfies

\begin{align}\label{le once more the vL}
\upsilon_L=\dot\gamma_L=-\int_\R\frac{\varrho_{T_1}(y)\dd y}{y-\gamma_{L}}-\frac{\gamma_{L}}{2}\,,\qquad\qquad \gamma_L\equiv\gamma_{L}(T_1)\,.
\end{align}
Thus,
\begin{align}\label{le almost equibrium relation}
 \left|\upsilon_L+\frac{x}{2}+\int_{\R}\frac{\varrho_{T_1}(y)\rd y}{y-x}\right|&=\left|\upsilon_L+\frac{\gamma_L}{2}+\int_{\R}\frac{\varrho_{T_1}(y)\rd y}{y-\gamma_L}\right| +O(|\re m_{T_1}(\gamma_L)- \re m_{T_1}(x)|)+O(|x-\gamma_L|)\nonumber\\
 &\le C\frac{N^\delta K}{N}+C\frac{K}{N}\,,\qquad\qquad x\in \boldsymbol{J}_\bz\,,
\end{align}
where we used~\eqref{le once more the vL}. To bound the second and third term on the right we used Assumptions $(1)$ of Theorem~\ref{main theorem} and the estimate on~$|\boldsymbol{J}_\bz|$ in~\eqref{le estimate on lenght of interval bz}. We may thus split
\begin{align*}
 (\wt V^{\wt{\boldsymbol\gamma}})'=\Omega_1+\Omega_2+\Omega_3+O\left(N^\delta\frac{K}{N}\right)\,,
\end{align*}
with
\begin{align*}
 \Omega_1(x)&\deq-\int_{\gamma_{L-K-N^{\xi}}}^{\gamma_{L+K+N^{\xi}}}\frac{\varrho_{T_1}(y)\rd y}{y-x}\,,\\
 \Omega_2(x)&\deq-\int_{\wt\gamma_1}^{\gamma_{L-K-N^{\xi}}}\frac{\varrho_{T_1}(y)\rd y}{y-x}+\frac{1}{N}\sum_{k=1}^{L-K-N^{\xi}}\frac{1}{\wt\gamma_k-x}\,,\\
 \Omega_3(x)&\deq-\int_{\gamma_{L+K+N^{\xi}}}^{\wt\gamma_{N}}\frac{\varrho_{T_1}(y)\rd y}{y-x}+\frac{1}{N}\sum_{k=L+K+N^{\xi}}^{N}\frac{1}{\wt\gamma_k-x}\,.
\end{align*}

To estimate $\Omega_1$ we use that $\varrho_{T_1}(y)=\varrho_{T_1}(x)+O(N^{\delta}|y-x|)$ (\cf~\eqref{assumption on regularity of varrho} and~\eqref{le estimate on lenght of interval bz}) to get
\begin{align}\label{le bound on omega 1}
 \Omega_1(x)&=-\varrho_{T_1}(\cy)\,\log\frac{d_+(x)}{d_-(x)}+O(N^\delta K/N)+O\left(\frac{N^{2\xi}}{{N^2d(x)}}\right)\,,\qquad\qquad x\in\boldsymbol{J}_\bz\,.
\end{align}
To obtain the third line, we used 
\begin{align*}
\gamma_{L+K+N^{\xi}}-x&=(\gamma_{L+K+N^{\xi}}-\gamma_{L+K+1})+(\gamma_{L+K+1}-x)=d_+(x)+O(N^{\xi}N^{-1})\,,
\end{align*}
respectively $x-\gamma_{L+K+N^{\xi}}=d_-(x)+O(N^{\xi}N^{-1})$, where we used the definition of $\wt{\boldsymbol{\gamma}}$ in~\eqref{def:wtg} and the definition of $d_\pm$ in~\eqref{le dpm}.

We next estimate $\Omega_2$ ($\Omega_3$ is estimated in the very same way): We split
\begin{align}\label{le splitting in estimating omega2}
 \frac{1}{N}\sum_{k=1}^{L-K-N^{\xi}}\frac{1}{\wt\gamma_k-x}=\frac{1}{N}\sum_{k=1}^{M-1}\frac{1}{y_k(T_1)-x}+\frac{1}{N}\sum_{k=M}^{L-K-N^{\xi}}\frac{1}{\wt\gamma_k-x}\,,
\end{align}
with $M=L-\sigma N$, such that we can estimate on one hand 
\begin{align}\label{le on the way to omega2 from assumption}
 \frac{1}{N}\sum_{k=1}^{M-1}\frac{1}{y_k(T_1)-x}=\int_{\wt\gamma_1}^{\gamma_{M-1}}\frac{\varrho_{T_1}(y)\dd y}{y-x}+O\left(\frac{N^{\xi}}{N^{\delta}}\right)\,,
\end{align}
since $\by(T_1)\in\cG_{T_1}$; \cf~\eqref{le Cgs2}. On the other hand, we estimate 
\begin{align}\label{le omega2 estimate the remainder}
 \frac{1}{N}\sum_{k=M}^{L-K-N^\xi}\frac{1}{\wt\gamma_k-x}&=\sum_{k=M}^{L-K-N^\xi}\int_{\gamma_{k-1}}^{\gamma_{k}}\frac{\varrho_{T_1}(y)\dd y}{\wt\gamma_k-x}\nonumber\\
&=\int_{\gamma_{M-1}}^{\gamma_{L-K-N^{\xi}}}\frac{\varrho_{T_1}(y)\dd y}{y-x}+O\left(\sum_{k=M}^{L-K-N^\xi}\int_{\gamma_{k-1}}^{\gamma_{k}}\frac{|\wt\gamma_k-y|\varrho_{T_1}(y)\dd y}{(y-x)^2}\right)\,.
\end{align}
Using Corollary~\ref{corollary almost quantiles} and recalling the definition of $I$ from~\eqref{le definition of Is}, we can bound the remainder in the above equation as
\begin{align}\label{le omega2 estimate the remainder with log}
 \sum_{k=M}^{L-K-N^\xi}\int_{\gamma_{k-1}}^{\gamma_{k}}\frac{|\wt\gamma_k-y|\varrho_{T_1}(y)\dd y}{(y-x)^2}&\le C\frac{N^\xi}{N}\int_{\gamma_{M-1}}^{\gamma_{L-K-N^{\xi}}}\frac{\dd y}{(y-x)^2}\nonumber\\ &\qquad+C\int_{\gamma_{L-K^5-N^{\chi}K^4-1}}^{\gamma_{L-K-N^{\xi}}}\frac{(N^{\xi}/K)(y-x)+N^\delta(y-x)^2}{(y-x)^2}\dd y\nonumber\\
 &\le C\frac{N^{\xi}}{Nd(x)} +C\frac{N^{\xi}}{K}\log\left(\frac{x-\gamma_{L-K^5-N^\chi K^4-1}}{x-\gamma_{L-K-N^\xi}}\right)+CN^\delta\frac{K^5}{N}\,.
\end{align}
Thus, using that $d(x)\le |\boldsymbol{J}_{\by(T_1)}|\le CK/N$, \ie $K\ge cNd(x)$, and that $K$ satisfies~\eqref{le choice of K}, we have
\begin{align*}
\sum_{k=M}^{L-K-N^\xi}\int_{\gamma_{k-1}}^{\gamma_{k}}\frac{|\wt\gamma_k-y|\varrho_{T_1}(y)\dd y}{(y-x)^2} \le C\frac{N^{2\xi}}{Nd(x)}\,,
\end{align*}
where we bounded the logarithmic term on the right side of~\eqref{le omega2 estimate the remainder with log} by $N^{\xi}$. Hence, combining this last estimate with~\eqref{le on the way to omega2 from assumption}, we find
\begin{align}\label{le bound on omega 2}
 |\Omega_2(x)|&\le C\frac{N^{2\xi}}{Nd(x)}+C\frac{N^{\xi}}{N^{\delta}}\le C\frac{N^{2\xi}}{Nd(x)}\,,
\end{align}
where we used that $K\ge cNd(x)$ and that $K$ satisfies~\eqref{le choice of K}. The same bounds holds for $\Omega_3$.

Combining~\eqref{le bound on omega 2},~\eqref{le bound on omega 1} and~\eqref{le bound difference V wtV}, we get~\eqref{Vder} for $ ( V^{\widetilde{\boldsymbol{\gamma}}})' $ (with $c=2)$.

It remains to check~\eqref{Ex} with the external points $ \by=\wt{\boldsymbol\gamma}$, \ie $|\E^{\omega_{T_1}}x_i-\alpha_i|\le N^{c\xi}/N$, $i\in I$.
First, we note that from~\eqref{rigz} we have $|\E^{\mu_{\aux}^{\bz}} x_i -\al_i|\le C N^\xi/N$. Then, using the logarithmic Sobolev inequality~\eqref{LSI}, we bound the relative entropy
\begin{multline}\label{le entropy estimate with mu aux}
   S(\mu_{\aux}^{\bz}  |\om_{T_1})  \le C \frac{K}{N}  \frac{1}{N}\sum_{i\in I} 
\E^{\om_{T_1}}\big| \partial_i \euler{-\beta N\sum_i [ V^{\bz}(x_i)- V^{\widetilde{\boldsymbol{\gamma}}}(x_i)]}\big|^2
\\
   \leq CK \E^{\mu_\aux}
\sum_{i\in I} \left| \frac{1}{2}V_{\aux}' (x_i) - \frac{1}{2}{x_i}-\upsilon_L -   \frac{1}{N} \sum_{k \; : \; |k-L|\ge K^{5}}
 \left[ \frac{1}{\widetilde\gamma_k-x_i}-\frac{1}{z_k-x_i}\right]
 \right|^2 \,.
\end{multline}
Note that  when $k$ is close to the interval~$I$ in the summation above, \ie when $|k-L|\le K^5$, then 
the corresponding terms exactly cancel by the choice of $\boldsymbol{\wt\gamma}$ in~\eqref{def:wtg}.

To bound the right side of~\eqref{le entropy estimate with mu aux}, we first recall that we have from~\eqref{regular potential U} the equilibrium relation
\begin{align}\label{le equilibrium relation for aux}
    V'_\aux(x) = -2\int_\R \frac{\varrho_\aux(y)}{y-x}\,\rd y \,,\qquad\qquad x\in\supp \varrho_\aux\,,
\end{align}
for the auxiliary $\beta$-ensemble $\mu_\aux$. We denote by $(\gamma_{\aux,i})_{i=1}^N$ the quantiles of the measure $\varrho_\aux$ and let $\gamma_{\aux,0}=a_\aux$, $\gamma_{\aux,N}=b_\aux$, where~$a_\aux,b_\aux$ are the endpoints of the support of $\varrho_{\aux}$.

We then bound the summation in~\eqref{le entropy estimate with mu aux} for all $k\le L-K^5$ as follows (the case $k\ge L+K^5$ is treated in the very same way),
\begin{align}\label{le estimate rhoaux exterior}
 \bigg|\sum_{k=1}^{L-K^5}\int_{\gamma_{\aux, k-1}}^{\gamma_{\aux, k}}\left(\frac{\varrho_{\aux}(y) \rd y}{z_k-x_i}-\frac{\varrho_{\aux}(y) \rd y}{y-x_i}\right)\bigg|&\le C\sum_{k=1}^{L-K^5}\int_{\gamma_{\aux, k-1}}^{\gamma_{\aux, k}}\frac{|y-z_k|\varrho_\aux(y)\dd y}{(y-x_i)^2}\nonumber\\
 &\le C\sum_{k=1}^{L-K^5}\frac{N^{\xi}}{N^{2/3} k^{1/3}}\int_{\gamma_{\aux, k-1}}^{\gamma_{\aux, k}}\frac{\varrho_\aux(y)\dd y}{(y-x_i)^2}\le C\frac{N^{\xi}}{K^5}\,,
\end{align}
for all $i\in I$, where we used the rigidity of $\bz$ (see~\eqref{le rigidity for zk and gammaauxk}) and that $\varrho_{\aux}$ vanishes like a square root at the endpoints $a_\aux,b_\aux$ of its support (recall that $\varrho_\aux$ is a rescaled and re-centered semicircle).

On the other hand, reasoning exactly as in~\eqref{le bound on omega 1}, we find that 
\begin{align}\label{le estimate rhoaux interior}
 \int_{\gamma_{\aux,L-K^5}}^{\gamma_{\aux,L+K^5}}\, \frac{\varrho_\aux(y)}{y-x_i}\,\rd y&= \int_{\gamma_{\aux,L-K^5}}^{\gamma_{\aux,L+K^5}}\, \frac{\varrho_{\aux}(x_i)+O(|y-x_i|)}{y-x_i}\,\rd y=O(K^{-4})+O(K^5/N)\,,
\end{align}
for $i\in I$. We therefore get, combining~\eqref{le equilibrium relation for aux},~\eqref{le estimate rhoaux exterior} and~\eqref{le estimate rhoaux interior}, 
\begin{align}\label{le rigidity estimate entropy 1}
    \left|\frac{1}{2} V_{\aux}' (x_i)+\frac{1}{N}\sum_{k \; : \; |L-k|\ge K^5} \frac{1}{z_k-x_i	} \right|\le \frac{C}{K^4}\,,\qquad\qquad i\in I\,.
\end{align}

Second, using the definition $\wt{\boldsymbol{\gamma}}$ in~\eqref{def:wtg}, we obtain similarly to~\eqref{le splitting in estimating omega2} and~\eqref{le on the way to omega2 from assumption}, 
\begin{align}\label{le nashorn}
 &\bigg|\frac{1}{N}\sum_{k=1}^{L-K^5}
 \frac{1}{\wt\gamma_k-x_i}-\int_{\wt\gamma_1}^{\gamma_{L-K^5}}\frac{ \varrho_{T_1}(y) \,\rd y}{y-x_i}\bigg|\le \bigg|\frac{1}{N}\sum_{k=M}^{L-K^5}
 \frac{1}{\wt\gamma_k-x_i}-\int_{\gamma_{M}}^{\gamma_{L-K^5}}\frac{ \varrho_{T_1}(y) \,\rd y}{y-x_i} \bigg|+C\frac{N^{\xi}}{N^{\delta}}\,,
\end{align}
with $M=L-\sigma N $. The first term on the right side of~\eqref{le nashorn} can be controlled, similarly to~\eqref{le omega2 estimate the remainder} and~\eqref{le omega2 estimate the remainder with log}, as 
\begin{align*}
 \bigg|\frac{1}{N}\sum_{k=M}^{L-K^5}
 \frac{1}{\wt\gamma_k-x_i}-\int_{\gamma_{M}}^{\gamma_{L-K^5}}\frac{ \varrho_{T_1}(y) \,\rd y}{y-x_i} \bigg|&\le C\frac{N^\xi}{K^5}+CN^\delta\frac{K^5}{N}+C\frac{N^{\xi+\chi}}{K^2}\,,
 \end{align*}
where we used $|\wt\gamma_{L-K^5}-x_i|\sim K^5/N$ and the assumption on $K$ in~\eqref{le choice of K}.
A similar estimate holds for the summations over $\llbracket L+K^5,N\rrbracket$. Further, repeating the arguments of~\eqref{le bound on omega 1}, we get
\begin{align*}
 \int_{\gamma_{L-K^5}}^{\gamma_{L+K^5}}\, \frac{\varrho_{T_1}(y)}{y-x_i}\,\rd y&=O(K^{-4})+O\left(N^\delta\frac{K^5}{N}\right)\,.
\end{align*}
Thus, combining the last two estimates and recalling~\eqref{le almost equibrium relation} as well as~\eqref{le choice of K} we find 
\begin{align}\label{le rigidity estimate entropy 2}
 \left|\frac{1}{2} x_i+\upsilon_L+\frac{1}{N}\sum_{k \; : \; |k-L|\ge K^5}
 \frac{1}{\wt\gamma_k-x_i}\right|\le \frac{C}{K^4}+C\frac{N^{\xi+\chi}}{K^2}\,.
\end{align}

Plugging~\eqref{le rigidity estimate entropy 2} and~\eqref{le rigidity estimate entropy 1} into~\eqref{le entropy estimate with mu aux} we get 
$S(\mu_{\aux}^{\bz}  |\om_{T_1}) \le CN^{2\xi+2\chi}K^{-2}$,
which finally leads, in combination with~\eqref{rigz}, to
\begin{align}\label{le probability estimate for rigidity in expectation}
    \P^{\om_{T_1}}\Big( |x_i-\al_i|\ge N^\xi N^{-1}\,,\forall i\in I\Big) &\le \P^{\mu_{\aux}^{\bz}} \Big( |x_i-\al_i|\ge N^\xi N^{-1}\,,\forall i\in I\Big)   
   +  \sqrt{2S(\mu_{\aux}^{\bz}  |\om_{T_1})} \nonumber\\ 
   &\le C\frac{N^{\xi}K}{N}+C\frac{N^{\xi+\chi}}{K}\,,
\end{align}
where we used~\eqref{entropyneq}. Together with the \emph{a priori} bound $|x_i-\alpha_i|\le C(K/N)$, this implies
\begin{align}\label{le rigidity for theorem 4.1}
 \left|\E^{\omega_{T_1}}x_i-\alpha_i\right|\le C\frac{N^{\xi}}{N}+C\frac{K}{N}\frac{N^{\xi+\chi}}{K}\le C\frac{N^{\xi+\chi}}{N}\,,\qquad\qquad i\in I\,.
\end{align}
Thus, choosing, \eg $\chi=\xi$, we get the bound~\eqref{Ex} for the measure $\omega_{T_1}$.

Applying Theorem~4.2 of~\cite{EYSG} as mentioned at the beginning of the proof, we see that the measure~$\om_{T_1}$ satisfies rigidity with exponent~$\xi$.
\end{proof}

The level repulsion estimate~\eqref{regi} in statement~$(2)$ of Proposition~\ref{lm:rigreg} is proved using the explicit Vandermonde structure of $\om_{T_1}$. The proof is essentially identical to the proof of Theorem~4.3 in~\cite{EYSG} given Section~7.2 of~\cite{EYSG}. We therefore leave the details aside.

\subsubsection{Proof of statement $(2)$ of Proposition~\ref{lm:rigreg}}
The rigidity for $g_t\,\om_{T_1}$, with fixed $t\ge T_1'$ in the sense of Definition~\ref{definition of regularity and rigidity} immediately follows from the rigidity for $\om_{T_1}$ and the entropy estimate~\eqref{SD}. Using the stochastic continuity of $(\wt\bx(t))$ and the rigidity of $g_t\,\om_{T_1}$ a sufficiently large set of discrete times, we can conclude that $\wt\P^{\ol \bY}$ itself is rigid; see Section~9.3 of~\cite{EYSG} for details.

It remains to prove the level repulsion for $g_t\,\om$ given in~\eqref{regii}.
\begin{proof}[Proof of~\eqref{regii}]

The level repulsion bound~\eqref{regii} follows from~\eqref{regi} and 
the entropy bound~\eqref{SD}. More precisely, we have
to introduce $\om_{T_1}^{\e_*}$,  an  $\e_*$-regularization in the $\om_{T_1}$ measure
in the same way as in Section 9.3 of~\cite{EYSG}. The parameter
$\e_*= \euler{-K^c}$ will be chosen tiny with a small $c>0$.  This regularization
modifies the interaction terms in~\eqref{DBMwtx} and in the
Hamiltonian $\cH_{T_1}$. In the latter the $\log$ becomes
$\log_{\e_*}$ defined
\begin{align}\label{logedef}
  \log_{\e_*}(x)\deq {\bf 1}(x\ge {\e_*})\log(x)+ {\bf 1}(x\le {\e_*})
    \Big\{\log \e_* + \frac{x-{\e_*}}{\e_*} - \frac{1}{2\e_*^2}(x-\e_*)^2\Big\}\,.
\end{align}
This has the property that $\pt_x^2\log_{\e_*}(x)$ is the same, $-x^{-2}$,
as before if $x>{\e_*}$, but it remains bounded by $\e_*^{-2}$ for all $x$. The Hamiltonian
is still convex.
The support of the measure $\om_{T_1}^{\e_*}$ is not $\boldsymbol{J}_\bz$ but the whole $\R$,
but it is still overwhelmingly supported on $\boldsymbol{J}_\bz$. In particular, $\omega_{T_1}$ and $\omega_{T_1}^{\e_*}$ are close in entropy sense, see (9.57) from~\cite{EYSG},
\begin{align}\label{epsent}
    S( \om_{T_1} | \om_{T_1}^{\e_*}) \le CK^C \e_*^2\,.
\end{align}
As a consequence, by the entropy inequality~\eqref{entropyneq} we may transfer  the 
rigidity bounds from the measure~$\om_{T_1}$  to the measure~$\om_{T_1}^{\e_*}$, \ie we have
\begin{align}\label{regrig}
  \P^{\om_{T_1}^{\e_*}}( |x_i-\al_i|\ge N^\xi/N) \le \euler{-K^c}\,.
\end{align}
Similar modifications occur in the SDE~\eqref{DBMwtx}; 
the $(\wt x_i-\wt x_j)^{-1}$ and also the $(\wt x_i -\wt\gamma_k)^{-1}$
terms get regularized to 
$$
 (\wt x_i-\wt x_j)^{-1}_{\e_*}\deq\pt_x \log_{\e_*} (\wt x_i -\wt x_j)\,,
$$
and they will be uniformly bounded by $\e_*^{-1}$. Now we can  prove~\eqref{regii}  with 
the regularization, since we can use the entropy inequality~\eqref{entropyneq} to get
\begin{align}\label{hg}
  \E^{g_t\,\om_{T_1}^{\e_*}} \frac{1}{[N|x_i- x_{i+1}|_{\e_*}]^p} \le 
\E^{\om_{T_1}^{\e_*}} \frac{1}{[N|x_i- x_{i+1}|_{\e_*}]^p}  + \e_*^{-p}\sqrt{2S_{\om_{T_1}^{\e_*}}(g_t)}
 \le C_p K^\xi\,,
\end{align}
Here in estimating the first term we used that the level repulsion bounds hold
for the regularized measure
$$
    \P^{\om_{T_1}^{\e_*}}( x_{i+1}-x_i\le s/N)\le CK^{\xi} s^2, \qquad\qquad s\ge K^\xi {\e_*}\,,
$$   
see~(9.58) of~\cite{EYSG}, \ie we have
\begin{align}\label{reglev}
\E^{\om_{T_1}^{\e_*}} \frac{1}{[N|x_i- x_{i+1}|_{\e_*}]^p} 
 \le C_p K^\xi\,.
\end{align}
 The exponential smallness  of the entropy~$ S_{\om_{T_1}^{\e_*}}(g_t)$ is proven exactly the same way as the
proof of~\eqref{SD}, since the Bakry-\'{E}mery type convexity argument
remains valid for the equilibrium measure $\om_{T_1}^{\e_*}$ as well.
This exponential smallness wins over $\e_*^{-p}$ if the constant $c$ in
the definition of $\e_* = \exp(-K^c)$ is small.

Since the only purpose of this regularization is to prove~\eqref{regii},
we will not carry the $\e_*$ superscript throughout the proof, \ie we 
continue to write $\om_{T_1}$ everywhere, although we really mean  $\om_{T_1}^{\e_*}$.
As we have seen, the key input information on $\om_{T_1}$ for our whole analysis, the rigidity~\eqref{regrig},
holds for the regularized measure. The other input, the level repulsion in the form
\eqref{regi} holds with an additional factor $K^\xi$, see~\eqref{reglev}, that plays no role in the applications of this estimate.
\end{proof}

\subsection{Local statistics of  $\om_{T_1}$}\label{subsection:local statitsics of omega}
In this subsection, we show that the gap statistics of the localized reference measure~$\omega_{T_1}$ are universal, \ie are given by the statistics of the Gaussian invariant ensemble up to negligible errors for large $N$. The precise universality statement for $\omega_{T_1}$ is as follows.
\begin{theorem}\label{thm gap statistics of omega0 is universal}
There is a small universal constants $\frak{e},\chi,\alpha>0$, such that for any $\by\in\cG_{T_1}$ (see~\eqref{le cGs}), for any fixed $j$ and for any smooth compactly supported function $\mathcal{O}$ of $n$ variables, we have
\begin{multline}\label{univ thm}
 \E^{\om_{T_1}}  \mathcal{O}\bigg( \Big((N\varrho_{T_1}(\gamma_L))\,( x_{i_0}-x_{i_0+j})\Big)_{j=1}^n\bigg)  =\E^{\Gauss}  \mathcal{O}\bigg(\Big((N\varrho_\#)\,( x_{i'_0}-x_{{i'_0}+j})\Big)_{j=1}^n\bigg) \\ +O(\|O'\|_\infty N^{-\frak{e}})\,,
\end{multline}
for $N$ sufficiently large, for any $i_0,i_0'\in\N_N$ satisfying $|i_0-L|\le N^\chi$, $|i_0'-L'|\le N^{\chi}$ with any $L'\in [\alpha N,(1-\alpha)N]$, and where $\varrho_\#\deq \varrho_{sc}(\gamma_{L',sc})$ denotes the density of the semicircle law $\varrho_{sc}$ at the location of the $L'$-th $N$-quantile of $\varrho_{sc}$. 
\end{theorem}
 In short, Theorem~\ref{thm gap statistics of omega0 is universal} assures that the gap statistics of the localized measure $\omega_{T_1}$ in the bulk is determined by the Gaussian invariant ensemble in the limit of large $N$. This result follows from Theorem~4.1 of~\cite{EYSG} and the properties of $\omega_{T_1}$ established in this section so far. 
\begin{proof}
 Theorem~4.1 of~\cite{EYSG}, as stated, directly compares two local measures, 
but together with Proposition~5.3 in~\cite{EYSG} it can also be stated as
a direct universality result: if the conditions of Theorem~4.1
of~\cite{EYSG} hold for a local measure, then it
has universal local gap statistics.

Theorem~4.1 (see also remark after Lemma~4.5) in~\cite{EYSG} has two types
of conditions.  

$(1)$ Regularity of the external potential in the sense of Definition 4.4 of~\cite{EYSG}. For the case at hand, the external potential $V^{\wt{\boldsymbol{\gamma}}}$ defined in~\eqref{le external potential for omega} is regular if
\begin{align}
   (V^{\wt{\boldsymbol{\gamma}}})'(x) &= \varrho_{T_1}( \cz) \log\frac{d_+(x)}{d_-(x)} + O\left(\frac{N^{c\xi}}{Nd(x)}\right)\,,\label{VYder}\\
  ( V^{{\wt{\boldsymbol{\gamma}}}})''(x)& \ge \frac{c}{d(x)}\,, \qquad \qquad \qquad\qquad \qquad x\in\boldsymbol{J}_\bz=[z_-,z_+]\,,\label{VY2der}
\end{align}
hold, with some fixed constant $c$, 
where $ \cz=(z_++z_-)/2$, $d(x)= \min\{ |x-z_+|, |x-z_-|\}$ and $d_\pm(x)$ as in~\eqref{le dpm}. Here we used the notation $z_-=z_{L-K-1}$, $z_+=z_{L+K+1}$.

In proving~\eqref{Vder} with external points $\widetilde{\boldsymbol{\gamma}}$, we already established~\eqref{VYder}.  The convexity estimate~\eqref{VY2der} follows from the rigidity of $\wt{\boldsymbol{\gamma}}$: there is a constant $c>0$ such that
$$
   (V^{\wt{\boldsymbol{\gamma}}})''(x) = V''(x) + \frac{2}{N}\sum_{j\not\in I} \frac{1}{(\wt{\gamma}_j-x)^2}
  \ge  \frac{1}{2}+ \frac{c}{d(x)}\,.
$$

$(2)$  The second input for  Theorem~4.1 of~\cite{EYSG} is 
\begin{align}\label{expx1}
   \big| \E^{ \om_{T_1}} x_i - \al_i\big| \le C\frac{N^{c\xi}}{N}\,,\qquad\qquad i\in I\,,
\end{align}
where $(\alpha_i)$ denote the $\cK$ equidistant points in $\boldsymbol{J}_\bz$. We have already established~\eqref{expx1} in~\eqref{le rigidity for theorem 4.1}.

Based upon these two inputs, Theorem~4.1 of~\cite{EYSG} implies~\eqref{univ thm}.
\end{proof}

\section{Universal gap statistics for small times}
In Section~\ref{subsec:localizing the measures}, we showed that the equilibrium measure $\omega_{T_1}$ of the dynamics~\eqref{DBMwtx} has universal gap statistics. In the present section, we compare the gaps of the two dynamics~\eqref{DBMwtx} and~\eqref{DBMx}. We proceed in three steps that are outlined in the Subsections~\ref{subsection:step 1},~\ref{Nashorn} and~\ref{le subsection removing forcing terms}. In Subsection~\ref{Tapir}, we then complete the proof of Theorem~\ref{main theorem}.

As in Section~\ref{subsec:localizing the measures}, we will fix a $\bY\in\cG$, or equivalently $\ol \bY\in \ol\cG$, but do not always indicate this choice in the notation. All estimates obtained will be uniform on $\cG$, so that we can integrate out $\bY$ at the very end of Subsection~\ref{Tapir}.

\subsection{Step 1: Small scale regularization}\label{subsection:step 1}
First we introduce a small regularization in~\eqref{DBMx} 
starting from the time $T_1$. This regularization is only needed for the critical case $\beta=1$, where the level repulsion, \cf Assumptions~$(3)$ of Theorem~\ref{main theorem}, is weakest. Level repulsion and the regularization introduced below will allow use to bound the kernel $B_{ij}$ defined below in~\eqref{def:Bij} as $\E|B_{ij}|\lesssim N$; see~\eqref{le expectation of Bij}. For $\beta>1$, a similar bound may be obtained without any regularization. In the following we carry the regularization along since the case $\beta=1$ is the hardest.

This regularization procedure is the same as in Section~3.1 of~\cite{FE}, but it is different from the regularization in the $\om_{T_1}$ measure
and in the $\wt\bx$ dynamics explained in part~(2) of the proof
of Proposition~\ref{lm:rigreg} (where the regularization parameter was called $\e_*$). Choose 
\begin{align}\label{defep}
 \e_{jk}\deq \begin{cases}\e\cdot {\bf 1}(j,k\in \II)\quad&\textrm{ if }j \ge k\,,\\ -\e\cdot {\bf 1}(j,k\in \II)\quad&\textrm{ if }j<k\,,\end{cases}\qquad\quad\textrm{ with }\quad \e\deq N^{-10C_1}\,,
\end{align}
for a large $C_1>1$. (Note that by the above choice $\e_*\ll \e$.)

Define the regularized version of~\eqref{DBMx} as
\begin{multline}\label{DBMxreg}
  \rd \wh x_i(t) = \sqrt{\frac{2}{\beta N}}\rd B_i(t) -\upsilon_L\,\dd t+\frac{1}{N} \sum_{j\in I} \frac{1}{\ol x_i(t)-\ol x_j(t)+\e_{ij}} \,\rd t\\ +
  \frac{1}{N} \sum_{k\not\in I} \frac{1}{ \ol x_i(t)-\ol y_k(t) +\e_{ik}}-\frac{\wh x_i(t)+\upsilon_L t}{2}\rd t\,, \qquad\qquad i\in I\,,\quad t\in[T_1,t_2]\,,
\end{multline}
with initial condition $\wh \bx(T_1)=\ol\bx(T_1)$, where the Brownian motions~$(B_i)$ are the same as in~\eqref{DBMx} and~\eqref{DBMwtx}. Note that~$\wh \bx$ may not preserve the ordering of the particles, but we will not need this property.

\begin{lemma}\label{lemma:xx} Define the event 
\begin{align}
 \Xi^1\deq \bigcap_{t\in[T_1,t_2]}\left\{\max_{i\in I}|\ol x_i(t)-\wh x_i(t) |\le N^{-5C_1}\right\}\,.
\end{align}
Under the conditions of Theorem~\ref{main theorem}, 
especially the level repulsion assumption \eqref{lever}, there is a set $\ol{\mathcal{G}}^* \subset\ol{\mathcal{G}}$
with $\P ( \ol{\mathcal{G}}^*)\ge 1-N^{-C_1}$ such that  $\P^{\ol\bY}( \Xi^1)\ge 1-CN^{-C_1}$ holds for any $\ol\bY\in \ol{\mathcal{G}}^*$. In particular, the local statistics of $\ol\bx(t)$ and $\wh\bx(t)$  are 
asymptotically the same for any $t\in [T_1,t_2]$. 
\end{lemma}

\begin{proof} Let $\mathcal{R}$ be the rigidity set
\begin{align*}
  \mathcal{R} \deq \{  |\ol x_i(t)-\ol\gamma_i(t)|\le N^{\xi}/N\; : \; t\in [T_1,t_2]\,,\; i\in I\}\,.
\end{align*}
 First we claim that $\P({\mathcal{R}}\cap \Xi^1)\ge 1- N^{-2C_1}$. This estimate can be proved following the argument in Section~3.1 of~\cite{FE} for $\upsilon_L=0$. Mutatis mutandis the same proof applies for $\upsilon_L\not=0$. As an input, we need a level repulsion estimate of the form
\begin{align}\label{levrep}
  \E { \bf{1}(\mathcal{R})} \frac{1}{[N|\ol x_{i+1}(t)-\ol x_i(t) + \e|]^2} \le N^{\delta+\xi} |\log \e|\, , \qquad \qquad\forall t\in [T_1,t_2]\,,\qquad i\in I\,, 
\end{align}
that follows from \eqref{lever}.  Therefore, by conditioning,
 there is an event $\ol{\mathcal{G}}^*$ such that $\P(\ol{\mathcal{G}}^*)
\ge 1- N^{-C_1}$ and 
$$
     \P^{\ol\bY}({\mathcal{R}}\cap \Xi^1)\ge 1-N^{-C_1}, \qquad \quad\forall \;\;\ol\bY\in \ol{\mathcal{G}}^*\,.
$$
Since $\P(\ol{\mathcal{G}})\ge 1- N^{-D}$ for any $D>0$, see \eqref{proba cG},
without loss of generality we can assume that $ \ol{\mathcal{G}}^*\subset\ol{\mathcal{G}}$.
Note that for $\ol\bY\in\ol{\mathcal{G}}$ we have $\mathbb{P}^{\ol \bY}(\mathcal{R})\ge 1-N^{-D}$, for any large $D>0$; \cf~\eqref{le event calG}. Choosing $D$ larger than $C_1$, completes the proof.
\end{proof}

\subsection{Step 2: H\"older regularity}\label{Nashorn}
To compare the gaps of $\wh \bx $ and $\wt \bx$, we introduce
\begin{align}\label{le bv}
 \bv\equiv \bv (t)\deq\euler{(t-T_1')/2}(\wh \bx(t)-\wt\bx(t))  \,,\qquad\qquad t\ge T_1'\,.
\end{align}
Subtracting~\eqref{DBMxreg} from~\eqref{DBMwtx} and dropping the $t$ argument for brevity, we have
\begin{multline}\label{le preheat}
  \frac{ \rd v_i}{\rd t}  = - \frac{1}{N}\sum_{\substack{j\in I \\ j\not=i}} \frac{v_i-v_j}{(\ol x_i- \ol x_j+\e_{ij})(\wt x_i - \wt x_j)}
 - v_i \frac{1}{N}\sum_{k\not\in I} \frac{1}{(\ol x_i-\ol y_k +\e_{ik})(\wt x_i - \wt\gamma_k)}-\frac{1}{2}{\euler{(t-T_1')/2}}\upsilon_L( t-T_1)\,\\
 + \frac{1}{N}\sum_{\substack{j\in I \\ j\not=i}} \frac{(\wh x_i -\ol x_i) - (\wh x_j - \ol x_j)+\e_{ij}}{(\ol x_i-\ol x_j+\e_{ij})(\wt x_i - \wt x_j)}
+\frac{1}{N}\sum_{k\not\in I} \frac{(\ol y_k-\wt\gamma_k) + (\wh x_i - \ol x_i)+\e_{ik}}{(\ol x_i-\ol y_k+\e_{ik})(\wt x_i -\wt \gamma_k)}\,,\qquad\qquad t\ge T_1'\,,
\end{multline}
$i\in I$. We rewrite~\eqref{le preheat} in the form
\begin{align}\label{heat}
  \frac{ \rd v_i}{\rd t}  = - (\cB \bv)_i + F_i^{(1)}+F_i^{(2)}\,, \qquad   (\cB \bv)_i\deq \sum_{j\in I} B_{ij} (v_i-v_j) + W_i v_i\,,
\end{align}
with time-dependent (symmetric) coefficients\footnote{Sometimes we write $B_{i,j}$ instead of $B_{ij}$ to clarify the notation. }, $i,j\in I$,
\begin{align}\label{def:Bij}
    B_{ij}  \deq  \frac{1}{N(\ol x_i- \ol x_j+\e_{ij})(\wt x_i - \wt x_j)}\,, \qquad W_i\deq  
\frac{1}{N}\sum_{k{ \not\in I}} \frac{1}{(\ol x_i- \ol y_k+\e_{ik})(\wt x_i -\wt \gamma_k)}\,, 
  \end{align}
  and with the ``forcing terms''
\begin{align}
F_i^{(1)}&\deq   \frac{1}{N}\sum_{\substack{j\in I \\ j\not=i}} \frac{(\wh x_i - \ol x_i) - (\wh x_j -\ol  x_j)+\e_{ij}}{(\ol x_i-\ol x_j+\e_{ij})(\wt x_i - \wt x_j)}-\frac{1}{2}\euler{(t-T_1')/2}\upsilon_L (t-T_1)\,,\\
F_i^{(2)}&\deq\frac{1}{N}\sum_{k\not\in I} \frac{ \wh x_i -\ol x_i+\e_{ik}}{(\ol x_i-\ol y_k+\e_{ik})(\wt x_i - \wt\gamma_k)}+\frac{1}{N}\sum_{k\not\in I} \frac{\ol y_k-\wt \gamma_k}{(\ol x_i-\ol y_k+\e_{ik})(\wt x_i -\wt \gamma_k)}\,.
\end{align} 
(Since $\ol x_i, \wt x_i,\widehat x_i$ and $\ol y_k$ depend on time, we have $B_{ij}\equiv B_{ij}(t)$, $W_i\equiv W_i(t)$, {\it etc}.)

We first study in Subsection~\ref{le subsubsection hoelder regularity of the free dynamics}  the ``free dynamics'' generated by~$\mathscr{B}$ and then deal with the forcing terms $(F_i^{(1)}),(F_i^{(2)})$ with a perturbative argument in Subsection~\ref{le subsection removing forcing terms}.

\subsubsection{H\"older regularity of the free dynamics}\label{le subsubsection hoelder regularity of the free dynamics}
Let $\wt\bv$ solve~\eqref{heat} without the forcing terms, \ie
\begin{align}\label{heat1}
  \frac{ \rd \wt v_i}{\rd t}  = -  (\cB\wt\bv)_i = - \sum_j B_{ij} (\wt v_i-\wt v_j) - W_i \wt v_i\,,\qquad\qquad t\ge T_1'\,,
\end{align}
with initial condition $\wt\bv(T_1')=\bv(T_1')$.

 We will need a certain upper bound on the coefficients~$B_{ij}$ in a space-time averaged sense. 
Let $\cT\deq[T_1', T_1'']$. Mimicking
Definition~9.7 in~\cite{EYSG}, we say that the Equation~\eqref{heat1} is {\it regular} at a space-time
point $(Z, \theta)\in I\times \cT$ with exponent $\rho>0$ if
\begin{align}\label{def:reg}
\sup_{t\in \cT}\;\sup_{1\le M\le \cK} \frac{1}{N^{-1} + |t-\theta|} \int_t^\theta \frac{1}{M}\sum_{i\in I\, :\, |i-Z|\le M}
 \sum_{j\in I\, :\, |j-Z|\le M} |B_{ij}(s)| \rd s \le N^{1+\rho}\,.
\end{align}

Furthermore, we say that the equation is {\it strongly regular} at a space-time
point $(Z, \theta)\in I\times \cT$ with exponent $\rho>0$ if it is regular at all points
$\{ Z\}\times \{\theta +\Omega\}$, where the set $\Omega$ is defined as
$$
   \Omega\deq \Big\{ -\frac{K}{N}\cdot 2^{-m}(1+2^{-k}) \; : \; 0\le m,k\le C\log K\Big\}\,.
$$

From Theorem~10.1 of~\cite{EYSG} and Lemma~\ref{lemma:xx} we obtain the following result.
\begin{lemma}\label{prop hoelder}
Let $c_1\sim 1/100$ and choose $T_1''=T_1'+K^{c_1}/N$. Then, there is an event $\Xi^2$ and constants $C$ and $\rho\equiv\rho(\xi)>0$ such that on the event~$\Xi^2$ the Equation~\eqref{heat1} is strongly regular at~$(L, T_1'')$, 
 and
  \begin{align}\label{le conclusion of hoelderregularity}
    {\bf 1}(\Xi^1\cap\Xi^2) |\wt v_{i+1}(T_1'')-  \wt v_i(T_1'')|&\le  CN^{-1+\xi} K^{-\fq/4}\,, \qquad\qquad  |i-L|\le C\,,
 \end{align}
 where $\fq>0$ is a universal constant.
 Moreover, we have the estimate $\mathbb{P}^{\ol\bY}(\Xi^1\cap\Xi^2)\ge 1-N^{-\rho/8}$, for~$N$ sufficiently large.  
\end{lemma}

\begin{proof}

We apply the H\"older regularity result, Theorem~10.1 of~\cite{EYSG}, to the evolution equation~\eqref{heat1}. Thanks to the regularization introduced in Step~1, \cf Subsection~\ref{subsection:step 1}, we have, for any $i,j\in I$ and $t\in[T_1',t_2]$, that
\begin{align}\label{le expectation of Bij}
  \E B_{ij} &\le N \left(\max_{i\in I} \E  \frac{1}{[N |\ol x_i-\ol x_{i-1}+\e|]^p}\right)^{1/p}
   \left( \max_{i\in I}\E  \frac{1}{ [N|\wt x_i -\wt x_{i-1}|]^q}\right)^{1/q}\nonumber\\
   &\le N \left(\max_{i\in I}\frac{1}{(N\e)^\phi} \E  \frac{1}{[N |\ol x_i-\ol x_{i-1}+\e|]^2}\right)^{1/p}
   \left( \max_{i\in I}\E  \frac{1}{ [N|\wt x_i -\wt x_{i-1}|]^q}\right)^{1/q}\nonumber\\
&   \le N \big( \e^{-\phi}  N^{\delta+\xi+\phi} |\log \e|\big)^{1/p} C_{q(\phi)} \nonumber\\
&\le C N^{1+\delta+2\xi}\,,
\end{align}
where we first applied H\"older's inequality with conjugate exponents $p,q$, with $p=2+\phi$, $\phi>0$, then used~\eqref{levrep} and~\eqref{regii}, and finally chose $\phi$ sufficiently small depending on $C_1$ in~\eqref{defep}.

Notice that to guarantee regularity in the sense of~\eqref{def:reg} (modulo a constant factor), 
instead of taking suprema over all $s\in \cT$, $M\in \llbracket 1, \cK
\rrbracket$,  it suffices to take suprema over a dyadic sequence of times $s_k= \theta \pm 2^{-k}$ and parameters $M_l=2^l$, $k,l\in \llbracket 1,C\log N\rrbracket$, since space-time averages on comparable scales are comparable.
Using~\eqref{le expectation of Bij}, setting $\rho \deq\delta+3\xi$ and applying Markov inequality,
for any fixed values of $s$ and $M$, the space-time average in~\eqref{def:reg} is bounded by $N^{1+\rho}$
with probability at least $1-N^{-\rho/2}$. Taking the union bound for not more than $C(\log N)^2$ times, we can guarantee regularity at any space-time point with 
probability at least $1-N^{-\rho/3}$. Since the definition of strong regularity requires
regularity at not more than $C(\log N)^2$ space-time points, an additional union bound guarantees strong regularity at any fixed space-time point with probability at least $1-N^{-\rho/4}$.
Defining
$$
  \wt\Xi^2\deq \big\{ \mbox{Equation
~\eqref{heat1} is strongly regular at $(L, T_1'')$ } \big\}\,,
$$
this proves that $\P^{\ol\bY}(\wt\Xi^2)\ge 1-N^{-\rho/4}$ and verifies condition ${\bf (C1)}_\rho$ in Theorem~10.1 of~\cite{EYSG} on~$\wt\Xi^2$. 

The other condition~${\bf (C2)}_\xi$ in Theorem~10.1 of~\cite{EYSG}
concerns large distance estimates of $B_{ij}$. More precisely, condition~${\bf (C2)}_\xi$ requires that
\begin{align}\label{c2 1}
     B_{ij}(t) \ge & \frac{N^{1-\xi}}{|i-j|^2}\,,\qquad\qquad t\in[T_1',T_1'']\,,
\end{align}
for any $i,j$ with $|L-i|\le K/C\,, |L-j|\le K/C$, and that 
     \begin{align}\label{c2 2}
     \frac{N{\bf 1}(\min \{ |L-i|,|L-j|\} \ge K/C)}{ C |i-j|^2} \le  B_{ij} (t) \le & \frac{CN}{ |i-j|^2}\,,\qquad\qquad t\in[T_1',T_1'']\,,
     \end{align}
for any $|i-j|\ge C'N^{\xi}$, with some constants $C, C'>10$. Further, $W_i$ is required to satisfy
\begin{align}\label{c2 3}
 \frac{CN^{1-\xi}}{\Delta_i}\le W_i(s)\le \frac{CN^{1+\xi}}{\Delta_i}\,,\qquad \qquad t\in[T_1',T_1'']\,,
\end{align}
where $\Delta_i\deq\min\{L+K+1-i,L-K-1-i \}$. Using the rigidity estimates for $\bx,\wt\bx$ of Lemma~\ref{lm:rigreg}, it is easy to check that, for any $\xi>0$, there is an event $\widehat \Xi^2$ and constants $c$, with $\P^{\ol\bY}(\widehat \Xi^2)\ge 1-\euler{-cN^{\xi}}$, such that~\eqref{c2 3},~\eqref{c2 2} and~\eqref{c2 1} hold. Set $\Xi^2=\wt \Xi^2\cap\widehat \Xi^2$. Then for all sufficiently small $\xi>0$, we have $\P(\Xi^2)\ge 1-CN^{-\rho/4}$.

Let $\|A\|_\infty\deq \sup_{i\in I} |A_i|$, $A\in\C^N$. Then the conclusion of Theorem~10.1 of~\cite{EYSG} for the equation~\eqref{heat1} is that
\begin{align}\label{le conclusion of hoelderregularity pre}
   {\bf 1}(\Xi^2) |\wt v_{i+1}(T_1'')-  \wt v_i(T_1'')|&\le CK^{-\fq/4}\|\tilde \bv(T_1')\|_\infty\,, \qquad\qquad  |i-L|\le C\,, 
\end{align}
where $\fq>0$ is a universal exponent and where $T_1''=T_1'+ K^{c_1}/N$. More precisely,~\eqref{le conclusion of hoelderregularity pre} follows from~(10.6) of~\cite{EYSG} after rescaling space and time by setting the constant $\alpha$ equal to $1/4$ (this $\alpha$ is different from the $\alpha$ used in the present paper).

Next, recalling from~\eqref{heat1} that $\tilde v_i(T_1')=v_i(T_1')$ and that $v_i(T_1')=\widehat x_i(T_1')-\widetilde x_i(T_1')$, we get
\begin{align}\label{le initial estimate for the hoelder}
 \|\bv(T_1')\|_\infty&\le\|\widehat{\bx}(T_1')-\ol\bx(T_1')\|_{\infty}+\|\ol\bx(T_1')-\wt \bx(T_1') \|_{\infty}\le CN^{-5C_0}+CN^{-1+\xi}\,,
\end{align}
on $\Xi^1\cap\Xi^2$, where we used Lemma~\ref{lemma:xx} and that the processes $(\widetilde\bx(t)), (\ol\bx(t))$,  are both rigid in the sense of Definition~\ref{definition of regularity and rigidity} for $t\in[T_1,t_2]$. Thus, combining~\eqref{le conclusion of hoelderregularity pre} with~\eqref{le initial estimate for the hoelder} we get 
\begin{align}\label{le conclusion of hoelderregularity bis}
    {\bf 1}(\Xi^1\cap\Xi^2) |\wt v_{i+1}(T_1'')-  \wt v_i(T_1'')|&\le  CN^{-1+\xi} K^{-\fq/4}\,, \qquad\qquad  |i-L|\le C\,,
\end{align}
where the event $\Xi^1\cap \Xi^2$ satisfies $\P^{\ol\bY}(\Xi^1\cap \Xi^2)\ge 1-N^{-\rho/8}$, $\rho\equiv\rho(\xi)>0$, for sufficiently small~$\xi>0$. 
\end{proof}
 
 \subsection{Step 3: Removing the forcing terms}\label{le subsection removing forcing terms}
 Having established the H\"older regularity of the free dynamics of~\eqref{heat1}, we now deal with the ``full dynamics''~\eqref{le preheat}. The main result is as follows.
\begin{proposition}\label{proposition vwtv}
Let $c_1\sim 1/100$ and choose $T_1''=T_1'+K^{c_1}/N$. Then there is an event $\Xi$ and constants $C$ and $c_2,c_3>0$ such that, for any $\ol\by\in\ol{\mathcal{G}}$,
\begin{align}\label{vwtv}
 {\bf 1}(\Xi)\,\max_{i\in \frac{1}{4}I}|v_i(T_1'')-\wt v_i(T_1'')|\le \frac{N^{-c_2}}{N}\,,
\end{align}
for $N$ sufficiently large, where $\frac{1}{4}I\deq\llbracket L-K/4,L+K/4\rrbracket$. Moreover the event $\Xi$ is such that $\Xi\subset\Xi^1$ and satisfies $\P(\Xi\cap\cG)\ge 1-N^{-c_3}$,
for $N$ sufficiently large.

\end{proposition}
The proof of Proposition~\ref{proposition vwtv} is given in the following subsections.

\subsubsection{Removing the forcing terms $F_i^{(1)}$}\label{le subsubsection f1}
Subtracting~\eqref{heat1} from~\eqref{heat} we have
$$
    \frac{ \rd (v_i-\wt v_i)}{\rd t}  = - \big[ \cB(\bv-\wt\bv)\big]_i + F_i\,,\qquad\qquad i\in I\,,\qquad t\in[T_1',T_1'']\,.
$$
Eventually, we will choose $i\in \frac{1}{4}I\deq \llbracket L-K/4,L+K/4\rrbracket$, yet here we can take $i\in I$.

Let $\cU_\cB(t,s)$ denote the time-dependent propagator of the equation~\eqref{heat1} from time $s$ to $t$, with $s\le t$. From Duhamel formula we have
$$
    v_i(t)-\wt v_i(t) = \int_{T_1'}^t \left(\cU_{\cB}(t, s) F(s)\right)_i\,\rd s\,,\qquad\qquad i\in I\,,\qquad t\ge T_1'\,.
$$
Note that $\cU_\cB$ is a contraction in the sup norm by maximum principle (recall that~$W_i\ge 0$). Thus
\begin{align}\label{vvv}
   |v_i(t) -\wt v_i(t)|\le \int_{T_1'}^t \max_{i\in I} |F_i^{(1)}(s)| \rd s+\int_{T_1'}^t \left|\left(\cU_{\cB}(t,s)\, F^{(2)}(s)\right)_i\, \right|\rd s \,,\qquad\qquad  t\ge T_1'\,.
\end{align}
Fixing $t=T_1''$, using Lemma~\ref{lemma:xx} and the choice of~$\e$ in~\eqref{defep}, we estimate
\begin{multline}\label{EF}
   {\bf 1}(\Xi^1\cap \Xi^2) \int_{T_1'}^{T_1''}\max_{i\in I} |F_i^{(1)}(s)| \rd s \le{\bf 1}(\Xi^2) C  N^{-5C_1}\int_{T_1'}^{T_1''}\sum_{i\in I}\sum_{j\in I} 
   |B_{ij}(s)|\,\dd s +C\upsilon_L(T_1''-T_1)^2\,,
\end{multline}
where we estimated the maximum  by the sum. Thus recalling~\eqref{def:reg} and using~\eqref{EF},~\eqref{vvv} we get
\begin{align*}
   {\bf 1}(\Xi^1\cap \Xi^2) \int_{T_1'}^{T_1''}\max_{i\in I} |F_i^{(1)}(s)| \rd s &\le CN^{-5C_1} K(T_1''-T_1')N^{1+\rho}+C\upsilon_L(T_1''-T_1)^2\nonumber\\
   &\le C K^{1+c_1}N^{-5C_1+\rho}+C \frac{K^{1+c_1}}{N^2}\,.
\end{align*}
Since $C_1>1$, we conclude that the effect due to $F^{(1)}$ is below the precision we are interested in, \ie there is $c>0$ such that, for any $i\in I$,
\begin{align}\label{Sa 1}
{\bf 1}(\Xi^1\cap \Xi^2)| v_i(T_1'') -\wt v_i(T_1'')|
&\le CN^{-1-c}+{\bf 1}(\Xi^1\cap \Xi^2)\int_{T_1'}^{T_1''} \left|\left(\cU_{\cB}(T_1'',s)\, F^{(2)}(s)\right)_i\right|\, \rd s\,.
\end{align}

\subsubsection{Removing the forcing terms $F_i^{(2,\mathrm{in})}$}\label{le subsubsection f2 one}
To estimate the influences of the forcing terms $(F_i^{(2)})$, we write
\begin{align*}
    F_i^{(2)} = F_i^{(2,\mathrm{in})} + F_i^{(2, \mathrm{out})}\,,
\end{align*}
with
\begin{align*}
  F^{(2,\mathrm{in})}_i \deq F^{(2)}_i{\bf 1}\Big(i\in \frac{1}{2}I\Big)\,,\qquad  F^{(2,\mathrm{out})}_i \deq F^{(2)}_i{\bf 1}\Big(i\in I\,, i\not\in \frac{1}{2}I\Big)\,,
\end{align*}
where we introduced $\frac{1}{2}I\deq\llbracket L-K/2,L+K/2\rrbracket$.

To control the inside part $F_i^{(2,\mathrm{in})}$, we use the following lemma. Recall the definition of the event $\cG$ in~\eqref{le event calG} and the definitions of the intervals of consecutive integers $\IO$ and $\II$ in~\eqref{le definition of Is}.

\begin{lemma}\label{lemma speed}
Let $K$ satisfy~\eqref{le choice of K} and fix $\ol\bY\in\ol\cG$. Then we have the following estimates.
 \begin{itemize}[noitemsep,topsep=0pt,partopsep=0pt,parsep=0pt]
  \item[$(1)$] For all $k\in \II\backslash I$, we have
 \begin{align}\label{lemma speed 2}
  |\ol y_k(t)-\ol y_k(T_1)|\le C\frac{N^\xi}{{N}}+CN^\delta(t-T_1)\frac{|L-k|}{N}+CN^\delta (t-T_1)^2\,,\qquad t\in[T_1,t_2]\,.
 \end{align}
\item[$(2)$] For all $k\in \IO\backslash I$, we have
\begin{align}\label{lemma speed 3}
 |\ol y_k(t)-\widetilde\gamma_k|\le C\frac{N^\xi}{N}+C\frac{N^{\xi}}{K}\frac{|L-k|}{N} \,,\qquad t\in[T_1,t_2]\,.
\end{align}
 \end{itemize}
 \end{lemma}
 We complement Lemma~\ref{lemma speed} with the estimate
\begin{align}\label{le speed estimate 4}
 |\ol y_k(t)-\wt\gamma_{k}|\le C N^{\xi}\sqrt{t}\,,\qquad\qquad t\ge T_1\,,\qquad k\in \II^c\,,
\end{align}
on $\cG$, as follows immediately from the definition of $\cG$ in~\eqref{le event calG}; \cf Assumption~$(4)$ of Theorem~\ref{main theorem}.

\begin{proof}
To prove~\eqref{lemma speed 2}, we estimate
\begin{align}\label{portello}
 |\ol y_k(t)-\ol y_k(T_1)|&\le |\ol y_k(t)-\ol \gamma_{k}(t)|+|\ol \gamma_{k}(t)-\ol \gamma_{k}(T_1)|+|\ol y_k(T_1)-\ol \gamma_{k}(T_1)|\nonumber\\
 &\le C\frac{N^\xi}{N}+|\ol\gamma_{k}(t)-\gamma_{k}(T_1)|\,,
\end{align}
on the event $\ol\cG$, where we used the rigidity bound in~\eqref{le Cgs1} for $k\in \II$. Next, we write
\begin{align*}
 \ol\gamma_{k}(t)-\gamma_{k}(T_1)=\gamma_k(t)-\upsilon_L(t-T_1)-\gamma_k(T_1)=\int_{T_1}^t\dot\gamma_k(s)\,\dd s-\upsilon_L(t-T_1)\,.
\end{align*}
Then by Lemma~\ref{le lemma for new gamma} we have
\begin{align*}
 \int_{T_1}^t\dot\gamma_k(s)\,\dd s&=\int_{T_1}^t\dot\gamma_L(s)\dd s+O(N^{-1+\delta}(t-T_1)|k-L|)\nonumber\\
&=\dot\gamma_L(T_1)(t-T_1)+O(N^{-1+\delta}(t-T_1)^2)+O(N^{-1+\delta}(t-T_1)|k-L|)\,.\nonumber
\end{align*}
Thus recalling that $\upsilon_L=\dot\gamma_L(T_1)$ by definition, we conclude that
\begin{align}\label{le speed estimate 2}
 |\ol\gamma_{k}(t)-\gamma_{k}(T_1)|\le CN^{-1+\delta}(t-T_1)^2+CN^{-1+\delta}(t-T_1)|k-L|\,.
\end{align}
Together with~\eqref{portello} this implies~\eqref{lemma speed 2}.

To bound the left side of~\eqref{lemma speed 3}, we split
\begin{align}\label{le speed estimate 3}
 |\ol y_k(t)-\widetilde\gamma_k|\le |\ol y_k(t)-\ol\gamma_{k}(t)|+\left|\ol\gamma_{k}(t)-\gamma_{k}(T_1)\right|+| \gamma_{k}(T_1)-\wt{\gamma}_k|\,.
\end{align}
Then, using the rigidity from the definition of $\cG^{1}_t$ in~\eqref{le Cgs1}, the first term on the right side can be bounded by $CN^{\xi}/N$. The second term on the right side is controlled by~\eqref{le speed estimate 2}. To bound the third term on the right side we apply Corollary~\ref{corollary almost quantiles} to find
\begin{align*}
 \left|\gamma_{k}(T_1)-\wt{\gamma}_k\right|\le C\frac{N^{\xi}}{N}+C\frac{N^{\xi}|\wt\gamma_k-\wt\gamma_L|}{K}+C N^\delta|\wt\gamma_k-\wt\gamma_L|^2\,.
\end{align*}
Recalling that $|\wt\gamma_k-\wt\gamma_L|\le CK^5/N$, for $k\in \IO\backslash I$, and using that $K$ satisfies~\eqref{le choice of K}, we get~\eqref{lemma speed 3}. 
\end{proof}

We now bound the term $F_i^{(2,\mathrm{in})}$.  Abbreviate 
\begin{align}\label{le wtbij}
 \widetilde B_{ik}\deq \frac{1}{N(\ol x_i-\ol y_k+\e_{ik}) \,(\wt x_i-\wt \gamma_k)}\,,\qquad\qquad i\in I\,,\quad k\in I^c\,.
\end{align}

Recall the bound on ${\bf 1}(\Xi^1)|\ol x_i-\widehat x_i|$ from Lemma~\ref{lemma:xx} and the definition of~$(\epsilon_{ik})$ in~\eqref{defep}. Using Lemma~\ref{lemma speed} and recalling that $t-T_1\le CK^2/N$, we obtain
\begin{align}\label{le what to be summed up}
{\bf 1}( \Xi^1)\big|F_i^{(2,\mathrm{in})}(s)\big|\le &C\sum_{k\in \IO\backslash I}\left(\frac{1}{N^{5C_1}}+\frac{N^{\xi}}{N}+\frac{N^{\xi}}{K}\frac{|k-L|}{N}\right)\,|\wt B_{ik}(s)|\nonumber\\
&\qquad+ C\sum_{k\in \II\backslash \IO}\left(\frac{N^{\xi}}{N}+N^\delta\frac{K^2}{N}\frac{|L-k|}{N}\right)|\wt B_{ik}(s)|\nonumber\\
&\qquad+ C\sum_{k\not\in \II}\left(N^{\xi}\frac{K}{\sqrt{N}}\right)|\wt B_{ik}(s)|\,,\qquad\qquad i\in \frac{1}{2}I\,,\qquad\qquad s\in[T_1',T_1'']\,,
\end{align}
where we also used~\eqref{le speed estimate 4} together with $T_1''-T_1\le CK^2/N $ to get the last term on the right of~\eqref{le what to be summed up}.

To perform the sums over $k$ in the first two terms on the right, we recall~\eqref{le wtbij} and we note that there are two constants $c,c'>0$, such that
\begin{align}\label{le to sum trivially}
 |\ol y_k-\ol x_i|\ge c|L-k|/N\,,\qquad\quad | \wt\gamma_k-\wt x_i|\ge c'|L-k|/N\,,\qquad\qquad k\in \II\backslash I\,,
\end{align}
where we used the rigidity estimate for $\by$ (embodied in $\cG$; see~\eqref{le event calG}), the rigidity estimate for~$\boldsymbol{\wt\gamma}$ obtained in Corollary~\ref{corollary almost quantiles} and the rigidity estimate for~$\bx$, $\wt\bx$ obtained in Proposition~\ref{lm:rigreg}, as well as the choice $i\in \frac{1}{2}I=\llbracket L-K/2,L+K/2\rrbracket$. The summation over $k\not\in \II$ in the third term is estimated using that $|\ol y_k-\ol x_i|\,|  \wt\gamma_k-\wt x_i|\ge c''(\sigma)>0 $, $k\not\in \II$. Hence, after summing up the right side of~\eqref{le what to be summed up}, we get
\begin{align}\label{le apriori estimate on Fi2}
{\bf 1}(\Xi^1)\big|F_i^{(2,\mathrm{in})}(s)\big|&\le \frac{1}{N^{5C_1-1}K}+\frac{N^{\xi}}{K}+\frac{N^{2\xi}}{K}
+ CN^\delta\frac{N^{\xi}}{NK^3}+C\frac{N^{\xi}K}{\sqrt{N}}\nonumber\\
&\le C\frac{N^{2\xi}}{K}\,,\qquad\qquad\quad i\in \frac{1}{2}I\,,\qquad s \in[T_1',t_2]\,,
\end{align}
where we used $5C_1-1>1$ and that $K$ satisfies~\eqref{le choice of K}. Thus by~\eqref{le apriori estimate on Fi2} we have
\begin{align*}
 {\bf 1}( \Xi^1)   \int_{T_1'}^{T_1''} \big| F_i^{(2,\mathrm{in})}(s)\big|\,\rd s \le (T_1''-T_1')  \frac{N^\xi}{K} = \frac{K^{c_1}}{N} \frac{N^{2\xi}}{K}\,,
\end{align*}
for all $i\in \frac{1}{2}I$, so this error is below the precision we are interested in: For some $c>0$ we have that
\begin{align}\label{Sa 2}
 {\bf 1}( \Xi^1) \big| v_i(T_1'') -\wt v_i(T_1'')\big| 
&\le C\frac{N^{-c}}{N}+{\bf 1}( \Xi^1)\int_{T_1'}^{T_1''}\left|\left(\cU_{\cB}(T_1'',s)\, F^{(2,\mathrm{out})}(s)\right)_i\right|\, \rd s\,,\qquad  i\in \frac{1}{2}I\,.
\end{align}

The outside part $F^{(2, \mathrm{out})}$ is treated with a finite speed of propagation estimate.
\subsubsection{Removing the forcing term $F^{(2,\mathrm{out})}$}\label{le subsubsection f2 two}
We first recall the finite of propagation estimate for the propagator $\mathscr{U}_\mathscr{B}$. Abbreviate for simplicity $\mathscr{U}(t,s)\equiv\mathscr{U}_{\mathscr{B}}(t,s)$ and denote its kernel by $\mathscr{U}_{ij}(t,s)$, $i,j\in I$. By Lemma~9.6 of~\cite{EYSG} there is $C$ such that
\begin{align}\label{le propagtation estimate}
 \left|\mathcal{U}_{ij}(t,s)\right|\le C\,\frac{K^{1/2}\sqrt{N(t-s)+1}}{|i-j|}\,,\qquad\qquad i,j\in I\,,\qquad\qquad t\ge s\ge T_1\,.
\end{align}
on $\Xi^2$. We refer to~\eqref{le propagtation estimate} as a finite speed of propagation estimate.

Next recall that we want to control
\begin{align*}
 \max_{i\in  \frac{1}{4}I}&\sum_{j\in I}\mathcal{U}_{ij}(t,s)F_{j}^{(2,\mathrm{out})}=\max_{i\in  \frac{1}{4}I}\frac{1}{N}\sum_{j\in I\backslash \frac{1}{2}I}\sum_{k\in I^c}\mathcal{U}_{ij}(t,s)F^{(2)}_{jk}(s)\,, 
\end{align*}
where $T_1'\le s\le t\le T_1''$, and where we have introduced
\begin{align*}
 F_{jk}^{(2)}(s)\deq\left( \frac{ \wh x_j -\ol x_j+\e_{jk}}{(\ol x_j-\ol y_k+\e_{jk})(\wt x_j - \wt\gamma_k)}+\frac{\ol y_k-\wt \gamma_k}{(\ol x_j-\ol y_k+\e_{jk})(\wt x_j -\wt \gamma_k)}\right)(s)\,.
\end{align*}
With some large $C$, we next split the summations over $k$ and $j$ as
\begin{multline}\label{le splitting of the propagation estimate}
 \sum_{j\in I}\mathcal{U}_{ij}(t,s)F_{j}^{(2,\mathrm{out})}=\frac{1}{N}\sum_{j\in I\backslash \frac{1}{2}I}\sum_{k\in I^c}{\bf 1}(|j-k|\ge CN^{\xi})\, \mathcal{U}_{ij}(t,s)F^{(2)}_{jk}(s)\\
 +\frac{1}{N} \sum_{j\in I\backslash \frac{1}{2}I}\sum_{k\in I^c}{\bf 1}(|j-k|< CN^{\xi})\, \mathcal{U}_{ij}(t,s)F^{(2)}_{jk}(s)\,,\qquad i\in  \frac{1}{4}I\,.
 \end{multline}

We start with bounding the first term on the right side of~\eqref{le splitting of the propagation estimate}. On the event~$\Xi^1$, we can bound
\begin{align}\label{chuchu}
 \frac{1}{N}\sum_{k\in I^c}\left|{\bf 1}(|j-k|\ge CN^{\xi})\,F^{(2)}_{jk}(s)\right|&\le C\sum_{k\in \IO\backslash I}{\bf 1}(|j-k|\ge CN^{\xi})\left(\frac{1}{N^{5C_1}}+\frac{N^\xi}{N}+\frac{N^{\xi}}{K}\frac{|k-L|}{N} \right)|\wt B_{jk}(s)|\nonumber\\
 &+ C\sum_{k\in \II\backslash \IO}{\bf 1}(|j-k|\ge CN^{\xi})\left(\frac{N^{\xi}}{N}+N^\delta\frac{K^2}{N}\frac{|k-L|}{N}\right)\,|\wt B_{jk}(s)|\nonumber\\
 &+C\sum_{k\not\in \II}{\bf 1}(|j-k|\ge CN^{\xi})\frac{N^\xi K}{\sqrt{N}}\,|\wt B_{jk}(s)|\,,
\end{align}
here we used~\eqref{le propagtation estimate} and the Lemmas~\ref{lemma:xx} and~\ref{lemma speed}. We further used that $s\le t$, $t-T_1\le CK^2/N$ by assumption. Thus, summing over $k$, we get
\begin{align*}
{\bf 1}(\Xi^1)\frac{1}{N}\sum_{k\in I^c}\left|{\bf 1}(|j-k|\ge CN^{\xi})\,F^{(2)}_{jk}(s)\right|\le C\frac{N^{2\xi}}{|j-
L + K +N^\xi|}\,,\qquad \qquad j\in I\backslash \frac{1}{2}I\,,
\end{align*}
where we used the estimates in~\eqref{le to sum trivially}. See~\eqref{le what to be summed up} and~\eqref{le apriori estimate on Fi2} for a similar estimate. Returning to~\eqref{le splitting of the propagation estimate} we see that the first term on the right side is bounded as
\begin{align}\label{le final estimate on first term in propagation}
 \frac{1}{N} \sum_{j\in I\backslash \frac{1}{2}I}\sum_{k\not\in I}{\bf 1}(|j-k|\ge CN^{\xi})\, |\mathcal{U}_{ij}(t,s)F^{(2)}_{jk}(s)|&\le C\sum_{j\,:\,K/2\le |j-L|\le K}\frac{N^{2\xi}K^{1/2}\sqrt{N(t-s)+1}}{|j-L+K/4|\,|j-L+K+N^\xi|}\nonumber\\
 &\le C N^{2\xi}K^{-1/2+c_1/2}\,,
\end{align}
on $\Xi^1\cap \Xi^2$, where we used that $t-s\le K^{c_1}/N$.

It remains to control the second term in~\eqref{le splitting of the propagation estimate}. Similar to~\eqref{chuchu}, we have on $\Xi^1$, for $j\in I\setminus \frac{1}{2}I$, that
\begin{align*}
\frac{1}{N}\sum_{k\in I^c}\left|{\bf 1}(|j-k|< CN^{\xi})\,F^{(2)}_{jk}(s)\right|&\le C\sum_{k\in \IO\backslash I}{\bf 1}(|j-k|< CN^{\xi})\left(\frac{1}{N^{5C_1}}+\frac{N^\xi}{N}+\frac{N^{2\xi}}{K}\frac{1}{N} \right)\,|\wt B_{jk}(s)|\nonumber\\
 &\le \frac{CN^\xi}{N}\sum_{k\in \IO\backslash I}{\bf 1}(|j-k|< CN^{\xi})|\wt B_{jk}(s)|\,.
\end{align*}
We thus have, for $i\in \frac{1}{4}I$, 
\begin{align}
 &\frac{1}{N}\sum_{k\in I^c}\left|\int_{T_1'}^{T_1''}\,\dd s\,{\bf 1}(|j-k|< CN^{\xi})\,\mathcal{U}_{ij}(T_1'',s)\,F^{(2)}_{jk}(s)\right|\nonumber\\
 &\qquad\le C \frac{N^{\xi}}{N}\int_{T_1'}^{T_1''}\dd s\sum_{k\in I^c}\,\sum_{j\in I} {\bf 1}(|j-k|< CN^{\xi})\frac{K^{1/2} \sqrt{N(T_1''-s)+1}}{|i-j|}\,| \wt B_{jk}(s)|\nonumber\\
 &\qquad\le C \frac{N^{2\xi}K^{1/2}}{\sqrt{N}K}{\sqrt{T_1''-T_1'}}\int_{T_1'}^{T_1''}\dd s\,\sum_{j=L+K-\lfloor CN^{\xi}\rfloor}^{L+K}| B_{j,j+1}(s)|\nonumber\\
 &\qquad\qquad \qquad+C \frac{N^{2\xi}K^{1/2}}{\sqrt{N}K}{\sqrt{T_1''-T_1'}}\int_{T_1'}^{T_1''}\dd s\,\sum_{j=L-K}^{L-K+\lceil CN^{\xi}\rceil}| B_{j,j-1}(s)|\,,\label{le at the edge of I}
\end{align}
where we used that $| \wt B_{jk}|\le|  B_{j,j+1}|$, $k> j$, respectively  $|\wt B_{jk}|\le  |B_{j,j-1}|$, $k< j$, for all $k\in I^c$, $j\in I$. (Here and below we use the convention that, for $j\in I$, $B_{j,L\pm(K+1)}=\wt B_{j,L\pm(K+1)}$, respectively $B_{j,k}=0$ if $|k|> L+K+1$.) To bound the two terms on the right side of~\eqref{le at the edge of I}, we use that the evolution equation~\eqref{heat1} is ``regular'' at the space-time points $(L+K,T_1'')$ and $(L-K,T_1'')$.
\begin{lemma}\label{le lemma regular at the edge}
 There is an event $\Xi^3$ such that evolution equation~\eqref{heat1} is regular at the space-time points $(L+K,T_1'')$ and $(L-K,T_1'')$ in the sense that
 \begin{align}
  {\bf 1}(\Xi^3)\sup_{s\in\cT}\sup_{1\le M\le \cK}\frac{1}{N^{-1}+|s-T_1''|}\int_s^{T_1''}\frac{1}{M}\sum_{i\in I\,:\, |i-L\pm K|\le M}|B_{i,i\pm 1}(s)| \dd s\le N^{1+\rho}\,.
 \end{align}
Moreover, we have the estimate $\P^{\ol\bY}(\Xi^1\cap\Xi^2\cap\Xi^3)\ge 1-N^{-\rho/10}$.
\end{lemma}
\begin{proof}
 We can follow almost verbatim the first part of the proof of Lemma~\ref{prop hoelder}. Using~\eqref{levrep} and~\eqref{regii}, we can bound $\E |B_{i,i\pm 1}|$ as in~\eqref{le expectation of Bij}. Then dyadic decompositions around the space-time points $(L\pm K,T_1'')$ combined with applications of Markov inequality yield the claim.
\end{proof}
Next, returning to the estimate in~\eqref{le at the edge of I}, we conclude from Lemma~\ref{le lemma regular at the edge} that the first term on the right can be bounded as
\begin{align*}
 {\bf 1}(\Xi^1\cap \Xi^3)\frac{N^{2\xi}K^{1/2}}{\sqrt{N}K}{\sqrt{T_1''-T_1'}}\int_{T_1'}^{T_1''}\dd s\,\sum_{j=L+K-\lfloor CN^{\xi}\rfloor}^{L+K}| B_{j,j+1}(s)|&\le C\frac{N^{3\xi}}{\sqrt{KN}}(T_1''-T_1')^{3/2} N^{1+\rho}\nonumber\\
 &\le C{N^{3\xi}K^{3c_1/2}}N^{-1+\rho}\,.
\end{align*}
Using the same argument to bound the second term on the right side of~\eqref{le at the edge of I}, we conclude that
\begin{align}\label{Sa 3}
 {\bf 1}(\Xi^1\cap \Xi^3)\frac{1}{N}\sum_{k\in I^c}\left|\int_{T_1'}^{T_1''}\,\dd s\,{\bf 1}(|j-k|< CN^{\xi})\,\mathcal{U}_{ij}(T_1'',s)\,F^{(2)}_{jk}(s)\right|\le C\frac{N^{3\xi+\rho}K^{3c_1/2}}{N}\,,\quad i\in \frac{1}{4}I\,. 
\end{align}

Summarizing the estimates above, we can now state the proof of Proposition~\ref{proposition vwtv}.
\begin{proof}[{Proof of Proposition~\ref{proposition vwtv}}]
 Let $\Xi\deq\Xi^1\cap\Xi^2\cap\Xi^3$. Note that $\P(\Xi)\ge 1-N^{-c_2}$,  for any $0<c_2\le \rho $ (with $\rho=\delta+3\xi$). Adding up the estimates~\eqref{Sa 3},~\eqref{Sa 2} and~\eqref{Sa 1}, and recalling that $K$ satisfies~\eqref{le choice of K} and that $c_1\sim 1/100$, we conclude that there is $c_3\equiv c_3(\xi)>0$ such that ~\eqref{vwtv} holds, for $\xi>0$ sufficiently small and $N$ large enough. 
 \end{proof}

\subsection{Conclusion of the Proof of Theorem~\ref{main theorem}}\label{Tapir}
Recall from~\eqref{le bv} the definition of $(v_i(t)$. Combining~\eqref{vwtv} and~\eqref{le conclusion of hoelderregularity} we obtain
\begin{align}\label{le final estimate on the gaps}
    {\bf 1}(\Xi) | v_{i+1} (t_2)-   v_i(t_2)|\le N^{-1-c_4}\,,\qquad\qquad |i-L|\le C\,,
    \end{align}
with some small $c_4>0$. Moreover, we have $\P(\Xi\cap{\cG})\ge 1-N^{-c_5}$, for some $c_5>0$. This exactly proves the following result.

\begin{lemma}\label{lemma:xxxx} The
gap statistics of $\wh\bx(T_1'')$ and $\wt\bx(T_1'')$ for indices near $L$ coincide in the limit of large~$N$.
\end{lemma}

Combining this with Lemma~\ref{lemma:xx}, we need only to understand the local statistics of $\wt\bx$. 
But by Lemma~\ref{lm:SD}, this is the same as the gap statistics of the local equilibrium measure $\om_{t_0}$. The latter one is universal as we showed in Subsection~\ref{subsection:local statitsics of omega}. To conclude the proof of Theorem~\ref{main theorem}, we note that we can integrate over~$\cG$, as follows from Lemma~\ref{le lemma probability estimates} and the assumption that the observable $\mathcal{O}$ is compactly supported.  Finally, choosing $T_1\ge t_1$ such that $T_1''=T$, we obtain~\eqref{equation main result}. This completes the proof of Theorem~\ref{main theorem}.

\begin{appendix}

\section{Semicircular flow}\label{section classical semicircle flow}

In this appendix we study the semicircular flow in more detail. In Subsection~\ref{sec:class} we prove Lemma~\ref{le density lemma} and Lemma~\ref{le lemma for new gamma}. In Subsection~\ref{appendix hoelder} we discuss the Assumption~$(1)$ of Theorem~\ref{main theorem} in more detail by arguing that it is satisfied for a large number of random matrix models.

\subsection{Classical flow of the density}\label{sec:class}
Recall from~\eqref{free convolution first} that $m_t$ satisfies 
\begin{align}\label{free convolution}
m_t(z)=\int_\R\frac{\varrho(y)\rd y}{\euler{-t/2}y-(1-\euler{-t})m_t(z)-z}\,,\qquad \im m_t(z)> 0\,,\qquad\im z>0\,,
\end{align}
for all $t\ge 0$, and that $m_t$ determines a density $\varrho_t$ via the Stieltjes inversion formula, \ie $\varrho_t(x)=\frac{1}{\pi}\lim_{\eta\searrow 0}\im m_t(x+\ii\eta)$, $x\in\R$. We call the map $t\mapsto\varrho_t$ the semicircular flow started from $\varrho$. 

It was shown in~\cite{B} that the density $\varrho_t$ is a real analytic function inside support for fixed $t>0$. Yet, without any further assumptions on $\varrho$, estimates on the derivatives of $\varrho_t$ deteriorate for small~$t$, since the Equation~\eqref{free convolution} may lose its stability properties (\ie the denominator on the right side can become singular). This can, for example, be remedied by imposing the conditions in Assumption~$(1)$ of Theorem~\ref{main theorem}. Consider for $\Sigma>0$ the domain
\begin{align*}
 \mathcal{D}_{\Sigma}\deq\{ z=x+\ii\eta \in \C\,:\,  x\in [E-\Sigma,E+\Sigma]\,, \eta \ge 0\}\,.
\end{align*}
Denote by $m_{0}$ the Stieltjes transform of $\varrho$. In accordance with the Assumption~$(1)$ of Theorem~\ref{main theorem}, we assume here that $m_0$ extends to a continuous function on $\mathcal{D}_\Sigma$ and that there is a small $\delta\ge 0$ and a constant $C$ such that
\begin{align}\label{nochmals assumptions}
 \sup_{z\in\mathcal{D}}|m_{0}(z)|\le C\,,\qquad\quad \sup_{z\in\mathcal{D}}|\partial_z^n m_{0}(z)|\le C(N^{\delta})^n\,,\quad\qquad n=1,2\,.
\end{align}

\begin{lemma}\label{le lemma mfc}
Consider the semicircular flow $\varrho_t$ started from $\varrho$. Let $\varrho$ satisfy Assumption~\eqref{nochmals assumptions} with exponent $\delta>0$ and $\Sigma>0$. Then there is $C'>0$, such that $m_t(z)$ is uniformly bounded on $\mathcal{D}_\Sigma$, for all $0\le t\le CN^{-2\delta}$. Moreover, there are constants $C,C'$, depending only on $\varrho$, such that
\begin{align}\label{le first bound in lemma mfc}
\sup_{z\in\mathcal{D}_{\Sigma/2} }|m_t(z)-m_0(z)|\le C t N^\delta\,,
\end{align}
for all $0\le t\le C'N^{-2\delta}$. Further, there are constants $C,C'$, depending only on~$\varrho$ , such that we have the bounds
\begin{align}\label{le second bound in lemma mfc}
 \sup_{z\in\mathcal{D}_{\Sigma/2}}|\partial_z^n m_t(z)|\le C(N^{\delta})^n\,,\qquad\qquad n=1,2\,,\qquad 0\le t\le C'N^{-2\delta}\,.
\end{align}
\end{lemma}
\begin{proof}
Set $\sigma_t\deq 1-\euler{-t}$. Starting from~\eqref{free convolution} we obtain, for $t>0$,
\begin{align*}
 |m_t(z)|^2&\le \left(\int_\R\frac{\varrho(y)\rd y }{|\euler{-t/2}y-z-\sigma_t m_t(z)|^2} \right)=\frac{\im m_t(z)}{\eta+\sigma_t\im m_t(z)}\,,\qquad\qquad z\in\C^+\,,
\end{align*}
where we first used Schwarz inequality for the probability measure $\varrho$ to get the second line. Then we used once more~\eqref{free convolution}. We thus obtain the rough \emph{a priori} bound
\begin{align}\label{le a priori bound on mt}
 |m_t(z)|\le \sigma_t^{-1/2}\,, \qquad\qquad t>0\,,
\end{align}
for $z\in\C^+\cup\R$. Next, we introduce
\begin{align}\label{le widetilde mt}
 \widetilde m_t(z)\deq\int_\R\frac{\varrho(\euler{t/2}v)\euler{t/2}\rd v}{v-z}\,.
\end{align}
Note that $\widetilde m_t$ is uniformly bounded on, say, $\mathcal{D}_{\Sigma/2}$ for, say, $t\le 1$. This may be seen by writing $\widetilde m_t(z)=\euler{t/2}m_0(\euler{t/2}z)$. Then we can write
\begin{align*}
m_t(z)=\widetilde m_t(z+\sigma_t m_t(z))\,.
\end{align*}
Thus, using the estimates on $m_0$ in~\eqref{nochmals assumptions}, we have
\begin{align}\label{le to be interated}
 | m_t(z)-\widetilde m_t(z)|&= |\widetilde m_t(z+\sigma_t m_t(z))-\widetilde m_t(z))|\le CN^\delta\sigma_t|m_t(z)|\le CN^\delta\sigma_t^{1/2}\,,
\end{align}
$0<t\le 1$, on $\mathcal{D}_{\Sigma/2}$, where we used the \emph{a priori} bound~\eqref{le a priori bound on mt}. It follows that
\begin{align*}
 |m_t(z)|\le |\widetilde m_t(z)|+CN^\delta\sigma_t^{1/2}\le C \,,\qquad\qquad 0<t\le CN^{-2\delta}\,.
\end{align*}
But then reasoning once more as in~\eqref{le to be interated}, we must have
\begin{align}\label{lelel first}
 | m_t(z)-\widetilde m_t(z)|&\le |\widetilde m_t(z+\sigma_t m_t(z))-\widetilde m_t(z))|\le C\sigma_tN^\delta|m_t(z)|\le C\sigma_t N^\delta\,,
\end{align}
 $0<t\le CN^{-2\delta}$, for all $z\in\mathcal{D}_{\Sigma/2}$. We hence obtain that $|m_{t}(z)|\le C$ on $\mathcal{D}_{\Sigma}$ and $0<t\le CN^{-2\delta}$. Next, we observe that
\begin{align}\label{lelel second}
 |\widetilde m_t(z)-m_0(z)|&=|\euler{t/2}m_0(\euler{-t/2}z)-m_0(z)|\le C(\euler{t/2}-1)+CN^{\delta}(\euler{t/2}-1)\,.
\end{align}
Combining~\eqref{lelel first} and~\eqref{lelel second},~\eqref{le first bound in lemma mfc} follows since $t\le C'N^{-2\delta}$ by assumption.

To deal with the derivatives of $m_t(z)$, we first note that we for $z\in\mathcal{D}_{\Sigma/2}$, we have $\euler{t/2}(z+\sigma_tm_t(z))\in\mathcal{D}_{\Sigma}$, for $t\le1$. Thus we can bound, for $z\in \mathcal{D}_{\Sigma/2}$ and $t\le N^{-2\delta}$,
\begin{align}\label{to be used twice}
 \left|\sigma_t\int_\R\frac{\varrho(v)\rd v}{(\euler{-t/2}v-z-\sigma_t m_t(z))^2}\right|&=\sigma_t|(\partial_z\widetilde m)(z+\sigma_t m_t(z))|\le CN^{\delta} \sigma_t\le CN^{-\delta}\,,
\end{align}
where we used the definition of $\widetilde m(z)$ in~\eqref{le widetilde mt} and the assumptions in~\eqref{nochmals assumptions}.

Next, differentiating~\eqref{free convolution} with respect to $z$, we obtain
\begin{align*}
 (\partial_z m_t(z))\left(1-\sigma_t\int_\R\frac{\varrho(v)\rd v}{(\euler{-t/2}v-z-\sigma_t m_t(z))^2}\right)=\int_\R\frac{\varrho(v)\rd v}{(\euler{-t/2}v-z-\sigma_t m_t(z))^2}\,.
\end{align*}
Hence, using twice~\eqref{to be used twice}, we get $ |\partial_z m_t(z)|\le CN^{\delta}$, for $z\in\mathcal{D}_{\Sigma/2}$ and $0\le t\le CN^{-2\delta}$. Repeating the argument, we see that there is another constant $C$ such that $ |\partial_z^2m_t(z)|\le C N^{2\delta}$, $0\le t\le N^{-2\delta}$,
for all $z\in\mathcal{D}_{\Sigma/2}$. This proves~\eqref{le second bound in lemma mfc}.
\end{proof}

\subsubsection{Quantiles}
From~\eqref{le second bound in lemma mfc}, we see that the derivatives of $m_t(z)$ are bounded inside $\mathcal{D}_{\Sigma/2}$ for $t\ll 1$. Without further assumptions on $\varrho$ we have little control on $m_t(z)$ outside $\mathcal{D}_{\Sigma/2}$. Yet, we can circumvent this problem, by introducing a regularization of $m_t(z)$, respectively the measure $\varrho_t$, as follows. Throughout the rest of this appendix, let $\regu>0$ satisfy 
\begin{align}\label{le reguregu}
\regu N^\delta\ll \frac{1}{N}\,.
\end{align}
Recall the definition of the Poisson kernel $P_\cdot$ in~\eqref{le poisson kernel}. We then set
\begin{align}\label{le regularization}
 \varrho^\regu_t(x)\deq(P_{\regu}*\varrho_t)(x)=\frac{1}{\pi}\im m_{t}(x+\ii\regu) \,.
\end{align}
We claim that
\begin{align}\label{le regularization trick}
\int_\R\frac{\varrho_{t}^\regu(y)\rd y}{y-z}=m_{t}(z+\ii\regu)\,,\qquad \qquad z\in\C^+\,.
\end{align}
Indeed, using~\eqref{le 1eleven} and~\eqref{le regularization}, we have
\begin{align*}
 \frac{1}{\pi}\im \int_\R\frac{\varrho_{t}^\regu(y)\rd y}{y-E-\ii\eta}&=(P_\eta * \varrho_t^\regu)(E)=(P_{\eta+\regu} * \varrho_t)(E)=\frac{1}{\pi}\im m_{t}(E+\ii\eta+\ii\regu)\,,
\end{align*}
where we used $P_{\eta}*P_{\regu}=P_{\eta+\regu}$. Since $z=E+\ii\eta$ and since the Stieltjes transform is analytic in the upper half plane, we get~\eqref{le regularization trick}. In the following we write $  m_{t}^\regu(z)\deq m_{t}(z+\ii\regu)$. Note that $\varrho_t^\regu$ is a probability measure. It follows from basic properties of the Poisson kernel that $\varrho_t^\regu$ converges uniformly on compact sets to $\varrho_t$ as $\regu\searrow 0$. Using~\eqref{nochmals assumptions} it is then easy to check that $|\varrho_0^\regu(x)-\varrho_0(x) |\le CN^\delta\regu\ll N^{-1}$, for all $x\in[E_*-\Sigma,E_*+\Sigma]$. Since the semicircular flow preserve regularity (for short times see Lemma~\ref{le lemma mfc}), we also get $|\varrho_t^\regu(x)-\varrho_t(x)|\le CN^\delta\regu\ll N^{-1}$, for all $x\in[E_*-\Sigma/2,E_*+\Sigma/2]$, $0\le t\le N^{-2\delta}$. 

As a consequence of the regularization in~\eqref{le regularization}, $\varrho_t^\regu$ is smooth with bounded derivatives (in terms of inverse powers of $\regu$) that all lie in $\mathrm{L}^p(\R)$, $1\le p\le \infty$.  Consequently, the following basic properties of the Hilbert transform can be justified easily (see, \eg~\cite{Pandey}): For $n\in\N$,
\begin{align}\label{le simple consequence of regu}
 \partial^n \Hi{\varrho_t^\regu}=(-1)^n\Hi{(\partial^n\varrho_t^\regu)}\,,
\end{align}
(here $\partial^n$ denotes the $n$-th spatial derivative). Further, we have $\mathrm{T}\Hi{\varrho_t^\regu}=-\varrho_t^\regu$.

Next, we define the continuous quantile, $\gamma^\regu_{w}(t)$, $w\in[0,N]$, of the measure $\varrho_t^\regu$ by
\begin{align}\label{definition of gammaw}
 \int_{-\infty}^{\gamma^\regu_{w}(t)}\varrho^\regu_t(y)\,\rd y=\frac{w}{N}\,,\qquad \qquad\varrho_0^\regu=\varrho^\regu\,,
\end{align}
\cf~\eqref{the gammas}. Note that $\gamma^\regu_w(t)$ is defined for any $w$ by~\eqref{definition of gammaw}, since the measure $\varrho_t^\regu$ is supported on the whole real axis. The measure $\varrho_t$ (without regularization) may be supported on several disjoint intervals. This leads to some ambiguities in the definition of the quantiles, \cf~\eqref{the gammas} for one way of resolving the ambiguity. Using the regularized density $\varrho_t^\regu$ is another way of avoiding this ambiguity. Nonetheless, we emphasize that the $\regu$-regularization is simply a technical tool: for every practical purpose we have $\regu=0$ and the reader may forget about it in the subsequent arguments.

\begin{corollary}\label{corollary le not so fast}
 Under the assumption of Lemma~\ref{le lemma mfc}, the following holds. For $0\le t\le C' N^{-2\delta}$, with~$C'$ sufficiently small, we have the estimates
 \begin{align}\label{le not so fast}
 \left|\varrho_{t}^\regu(E)-\varrho^\regu(E)\right|\le C N^\delta t\,,\qquad\qquad \left| \Hi{\varrho_t^\regu}(E)- \Hi{\varrho^\regu}(E)\right|\le CN^\delta t\,,
 \end{align}
for all $E\in\R$. Further, for $w$ such that $\gamma_w^\regu(0)\in[E_*-\Sigma,E_*+\Sigma]$, we have the estimate
\begin{align}\label{le not so fast gamma}
|\gamma^\regu_w(t)-\gamma^\regu_w(0)|\le C{N^{-\delta/2}}\,,\qquad\qquad 0\le t\le N^{-2\delta}\,,
\end{align}
for some $C$. In particular, if $\gamma_w^\regu(0)\in[E_*-\Sigma/4,E_*+\Sigma/4]$, then $\gamma_w^\regu(t)\in [E_*-\Sigma/2,E_*+\Sigma/2]$, for all $0\le t\le C'N^{-2\delta}$.
\end{corollary}
\begin{proof}
 The estimates in~\eqref{le not so fast} follow from~\eqref{le first bound in lemma mfc} by noticing that $\varrho_t^\regu(E)=\pi^{-1}\im m_{t}(E+\ii\regu)$, respectively $\Hi{\varrho^\regu_t}(E)=\re m_t(E+\ii\regu)$.  To establish~\eqref{le not so fast gamma}, we note that, by definition,
\begin{align*}
 \int_{-\infty}^{\gamma_w^\regu(t)}\varrho^\regu_t(y)\,\dd y=\int_{-\infty}^{\gamma_w^\regu(0)}\varrho^\regu(y)\,\dd y=\frac{w}{N}\,.
\end{align*}
Thus, using~\eqref{le not so fast}, we get
\begin{align}\label{le equation star}
 \int_{-\infty}^{\gamma_w^\regu(t)}\varrho^\regu(y)\,\dd y=\int_{-\infty}^{\gamma_w^\regu(0)}\varrho^\regu(y)\,\dd y+O(\sqrt{N^\delta t})\,,
\end{align}
where we also used that $\varrho$ has finite second moment. By our assumption on $w$ we must have $\varrho^\regu(\gamma_w^\regu(0))\ge c$, for some $c>0$. We thus get from~\eqref{le equation star} and~\eqref{nochmals assumptions} that
\begin{align*}
\varrho^\regu(\gamma_w^\regu(0))|\gamma_w^\regu(t)-\gamma_w^\regu(0)|&\le C \left| \int_{-\infty}^{\gamma_w^\regu(t)}\varrho^\regu(y)\dd y-\int_{-\infty}^{\gamma_w^\regu(0)}\varrho^\regu(y)\dd y \right|\le C(N^\delta t)^{1/2}\,.
 \end{align*}
Thus, for $0\le t\le CN^{-2\delta}$,  we get  $ |\gamma_w^\regu(t)-\gamma_w^\regu(0)|\le CN^{-\delta/2}$.
\end{proof}

The estimate~\eqref{le not so fast gamma} will serve as {\it a priori} bound below. To get precise estimates, we derive next the equation of motion of~$\gamma^\regu_w(t)$ under the semicircular flow $t\to \varrho_t^\regu$.
\begin{lemma}\label{lem: constant weight}
For $t>0$, we have for all $w\in[0,N]$,
 \begin{align}\label{evolution of gammaw}
 \frac{\rd\gamma_{w}^\regu(t)}{\rd t}&=-\Hi{\varrho_t^\regu}(\gamma^\regu_{w}(t))-\frac{\gamma^\regu_{w}(t)}{2}-\frac{\regu}{2}\frac{\Hi{\varrho_t^\regu}(\gamma^\regu_{w}(t))}{\varrho_t^\regu(\gamma^\regu_w(t))}\,,
 \end{align}
and 
\begin{align}\label{le spatial derivative of gammaw}
 \frac{\dd \gamma_{w}^\regu(t)}{\dd w}&=\frac{1}{N\varrho_t^\regu(\gamma_w^\regu(t))}\,.
\end{align}
In particular, if $\varrho_t^\regu(\gamma_w^\regu(t))\ge c$, for some fixed $c>0$, we have the uniform estimates
\begin{align}\label{le estimate in eom}
 \frac{\rd\gamma_{w}^\regu(t)}{\rd t}=-\Hi{\varrho_t^\regu}(\gamma^\regu_{w}(t))-\frac{\gamma^\regu_{w}(t)}{2}+O(\regu)\,,\qquad\qquad \frac{\dd \gamma_{w}^\regu(t)}{\dd w}=O(N^{-1})\,.
\end{align}

\end{lemma}
\begin{remark} Lemma~\ref{lem: constant weight} directly controls $\gamma_w^\regu(t)$ in the bulk. With some more effort the third term on the right of~\eqref{evolution of gammaw} can by controlled at the edges. For example, assuming that~$\varrho_t$ vanishes as a square root at, say, its lowest endpoint, it can also be shown that~$\varrho_t^\regu(\gamma^\regu_w(t))\gtrsim\sqrt{\regu}$, for small~$w$. Thus the ``error'' term in~\eqref{evolution of gammaw} is of order~$\sqrt{\regu}$, as~$\regu\searrow 0$, at the lowest edge of the density~$\varrho^\regu_t$.
\end{remark}

 \begin{proof}
  We recall that $m_t(z)$, $z\in \C^+$, defined in~\eqref{free convolution}, satisfies the following complex Burgers' equation~\cite{Pastur} (see also~\cite{Voi86})
 \begin{align}\label{Burgers}
  \frac{\rd m_t(z)}{\rd t}=\frac{1}{2}\frac{\rd }{\rd z}\Big(m_t(z)\left(m_t(z)+z\right)\Big)\,,\qquad \qquad z\in\C^+\,,\,t\ge0\,.
 \end{align}
 Indeed differentiating~\eqref{free convolution} with respect to time we obtain~\eqref{Burgers} after a series of elementary manipulations. We use~\eqref{Burgers} with $m_t^\regu(z)=m_t(z+\ii\regu)$ replacing $m_t(z)$ in the following.

To deal with the right side of~\eqref{Burgers}, we note that
\begin{align}\label{small loop equation}
 m_t^\regu(z)^2&=\int_{\R^2}\frac{\varrho^\regu_t(x)\rd x}{x-z}\frac{\varrho^\regu_t(y)\rd y}{y-z}
 =2\int_{\R^2}\frac{\varrho_t^\regu(x)\rd x}{x-z}\frac{\varrho_t^\regu(y)\rd y}{y-x}
 =2\int_\R\frac{\Hi{\varrho_t^\regu}(x)\varrho_t^\regu(x)\rd x}{x-z}\,,
\end{align}
$z\in\C^+$. Plugging~\eqref{small loop equation} into~\eqref{Burgers}, we get
\begin{align*}
 \frac{\rd m_t^\regu(z)}{\rd t}
 &=\frac{1}{2}\frac{\rd }{\rd z}\left(\int_\R\frac{2\Hi{\varrho_t^\regu}(x)\varrho_t^\regu(x)\rd x}{x-z}+\int_\R\frac{x\varrho_t^\regu(x)\rd x}{x-z}+\int_\R\frac{\ii\regu\varrho_t^\regu(x)}{x-z}\right)\,,\qquad z\in\C^+\,,
\end{align*}
where we used that $\int_\R\varrho_t^\regu(x)\dd x=1$. Differentiating the right side with respect to $z$, we get
\begin{align*}
  \frac{\rd m_t^\regu(z)}{\rd t}&=-\frac{1}{2}\int_\R\frac{\left(2\Hi{\varrho_t^\regu}(x)+x+\ii\regu\right)\varrho_t^\regu(x)\rd x}{(x-z)^2}\,,\qquad\qquad z\in\C^+\,.
\end{align*}
Integrating by parts in $x$, taking the imaginary part and the limit $\eta\searrow 0$ (with $\regu>0$), we obtain
\begin{align}\label{continuityregu}
 \dot\varrho^\regu_t(E)=\frac{1}{2}\left[\left(2\Hi{\varrho_t^\regu}(E)+E\right)\varrho_t^\regu(E)\right]'+\frac{\regu}{2}\left[\Hi{\varrho_t^\regu}(E)\right]'\,,
\end{align}
where we use the notation $\dot{a}\equiv \partial_t a$ and $a'\equiv \partial_E a$, for any function $a\equiv a(t,E)$. On the other hand, differentiating the defining equation~\eqref{definition of gammaw} of $\gamma_{w}^\regu(t)$ with respect to $t$, we get
\begin{align}\label{samas}
\dot\gamma_{w}^\regu(t)=-\frac{1}{\varrho_t^\regu(\gamma^\regu_w(t))}\int_{-\infty}^{\gamma^\regu_{w}(t)}\dot\varrho^\regu_t(y)\rd y\,.
\end{align}
Hence, combining~\eqref{samas} and~\eqref{continuityregu} we get~\eqref{evolution of gammaw}. 

Finally, to establish the first estimate in~\eqref{le estimate in eom}, we note that $\Hi{\varrho_t^\regu}$ is uniformly bounded by Lemma~\ref{le lemma mfc}. Together with the assumption $\varrho_t^\regu(\gamma_w^\regu(t))\ge c>0$, ~\eqref{le estimate in eom} follows. To prove~\eqref{le spatial derivative of gammaw} and the second estimate in~\eqref{le estimate in eom} it suffices to differentiate~\eqref{definition of gammaw} with respect to~$w$.
\end{proof}
\begin{remark} For $\varrho\in \mathcal{M}(\R)$, let
\begin{align}
 \mathrm{Ent}[\varrho]\deq\int_\R\frac{1}{2}x^2\varrho(x)\,\dd x-\int_\R \log |x-y|\dd \varrho(x)\,\dd\varrho(y)\,,
\end{align}
Voiculescu's free entropy. Then the limiting equation of~\eqref{continuityregu}, \ie as $\eta_*\searrow 0$, is the gradient flow of $\mathrm{Ent}[\varrho_t]$ on the Wasserstein space $\mathcal{P}_2(\R)$; see~\cite{B2,BS,LLX}.
\end{remark}

To conclude this subsection, we give an estimate on the second derivative $\ddot\gamma_w^\regu({t})$ inside the bulk.
\begin{lemma}\label{le lemma estimate on ddotgamma}
Under the assumptions of Lemma~\ref{le lemma mfc} the following holds. For $w\in[0,N]$, such that $\gamma_w^\regu(0)\in[E_*-\Sigma/4,E_*+\Sigma/4]$ we have
 \begin{align}\label{le estimate on ddotgamma}
 |\ddot\gamma_w^\regu(t)|&\le CN^{\delta}\left(1+|\dot\gamma_w^\regu(t)|\right)\,, \qquad\qquad 0<t\le C'N^{-2\delta}\,.
 \end{align}
\end{lemma}
\begin{proof}
 Let $0<t\le C'N^{-2\delta}$. For notational simplicity, we abbreviate here $\gamma_{w,t}\equiv\gamma_w^\regu(t)$ and $\varrho_t\equiv \varrho_t^\regu $. We first compute
\begin{align*}
 \frac{\dd}{\dd t} \left[\Hi{\varrho_t}(\gamma_{w,t})\right]&=\Hi{\dot\varrho_t}(\gamma_{w,t})+\left[\Hi{\varrho_t}\right]'(\gamma_{w,t})\,\dot\gamma_{w,t}\nonumber\\
&=-\frac{1}{2}\left[\mathrm{T}{\left((2\Hi{\varrho_t}(\cdot)+\cdot)\varrho_t(\cdot)\right)}\right]'(\gamma_{w,t})- \frac{\regu}{2}\left[\mathrm{T}{\Hi{\varrho_t}}\right]'(\gamma_{w,t})+\left[\mathrm{T}{\varrho_t}\right]'(\gamma_{w,t})\,\dot\gamma_{w,t}\,,
\end{align*}
where we used~\eqref{continuityregu} and~\eqref{le simple consequence of regu}. (Here $\mathrm{T}{\left((2\Hi{\varrho_t}(\cdot)+\cdot)\varrho_t(\cdot)\right)}(x)$ denotes the Hilbert transform of the function $y\to (\Hi{\varrho_t}(y)+y)\varrho_t(y) $ evaluated at $x$.)  Next, we note the identities
 \begin{align*}
  \frac{1}{2}\left(\mathrm{T}{\varrho_t}\right)^2-\frac{1}{2}\varrho_t^2=\mathrm{T}\left({\Hi{\varrho_t}\varrho_t}\right)\,,
 \end{align*}
which follows from~\eqref{small loop equation} and 
$\Hi{\left(\varrho_t(\,\cdot\,)\,\cdot\,\right)}(x)=1+x \Hi{\varrho_t}(x)$, $x\in\R$,
which can be checked by hand. We hence obtain
\begin{align*}
  \frac{\dd}{\dd t} \Hi{\varrho_t}(\gamma_{w,t})&=-\frac{1}{2}\left[\Hi{\varrho_t}^2\right]'(\gamma_{w,t})+\frac{1}{2}\left[\varrho_t^2\right]'(\gamma_{w,t})-\frac{1}{2}\left[\,\cdot\,\Hi{\varrho_t}(\,\cdot\,) \right]'(\gamma_{w,t})\nonumber\\ &\qquad\qquad+\left[\mathrm{T}\varrho_t\right]'(\gamma_{w,t})\,\dot\gamma_{w,t}+\frac{\regu}{2}[\varrho_t]
 '(\gamma_{w,t})\,,
\end{align*}
where we also used that $\Hi{\Hi{\varrho_t}}=-\varrho_t$. Simplifying further and using~\eqref{evolution of gammaw}, we eventually obtain
\begin{align*}
  \frac{\dd}{\dd t}\left[ \Hi{\varrho_t}(\gamma_{w,t})\right]&=-\frac{1}{2}\left[\varrho_t^2\right]'(\gamma_{w,t})+\frac{1}{2}[\Hi{\varrho_t}](\gamma_{w,t})+2[\Hi{\varrho_t}]'(\gamma_{w,t})\,\dot\gamma_{w,t}\nonumber\\ &\qquad\qquad+\frac{\regu}{2}[\varrho_t]'(\gamma_{w,t})-\frac{\regu}{2}\left[\frac{\Hi{\varrho_t}'\Hi{\varrho_t}}{\varrho_t}\right](\gamma_{w,t})\,.
\end{align*}
We further compute
\begin{multline}\label{le ddot gamma 1}
 \frac{\dd}{\dd t}\left[\left[\frac{\Hi{\varrho_t}}{\varrho_t}\right](\gamma_{w,t})\right]=\left[\frac{\partial_t \Hi{\varrho_t}}{\varrho_t}\right](\gamma_{w,t})-\left[\frac{\Hi{\varrho_t}\,\dot\varrho_t}{\varrho_t^2}\right](\gamma_{w,t})\\+\left[\frac{\left[\Hi{\varrho_t}\right]'}{\varrho_t}\right](\gamma_{w,t})\,\dot\gamma_{w,t}-\left[\frac{\Hi{\varrho_t}\,\varrho_t'}{\varrho_t^2}\right](\gamma_{w,t})\,\dot\gamma_{w,t}\,.
\end{multline}
We next recall that we have the bounds $|\varrho_t(\gamma_{w,t})|+|\Hi{\varrho_t}(\gamma_{w,t})|\le C$, as follows from the boundedness of $m_t$ (see Lemma~\ref{le lemma mfc}), and
 \begin{align}\label{le ddot gamma 2}
  |\partial_x \im m_{t}(x+\ii\regu)|+|\partial_x \re m_{t}(x+\ii\regu)|\le CN^{\delta}\,,
 \end{align}
for any $x\in[E_*-\Sigma/2,E_*+\Sigma/2]$, see~\eqref{le second bound in lemma mfc}. Thus, recalling that $\Hi\varrho_t^\regu(E)=\re m_t(E+\ii\regu)$, $\pi\varrho^\regu_t(E)=\im m_t(E+\ii\regu)$, we find
\begin{align*}
|\Hi{\varrho_t}'(\gamma_{w,t})|+|\varrho_t'(\gamma_{w,t})|\le CN^\delta\,,\qquad|\dot\varrho_t(\gamma_{w,t})|\le CN^\delta\,. 
\end{align*}
Hence, differentiating~\eqref{evolution of gammaw} with respect to $t$ and using~\eqref{le ddot gamma 1},~\eqref{le ddot gamma 2}, we can bound
\begin{align*}
 |\ddot\gamma_{w,t}|\le \frac{1}{2}|\dot\gamma_{w,t}|+C\left(1+\regu+\frac{\regu}{\varrho_t(\gamma_{w,t})}+\frac{\regu}{(\varrho_t(\gamma_{w,t}))^2} \right)\left(1+CN^\delta+N^\delta|\dot\gamma_{w,t}|\right)\,.
\end{align*}
Finally, using that $\varrho_t(\gamma_{w,0})\ge c/2$, for $0\le t\le CN^{-2\delta}$, as follows from the estimates~\eqref{le not so fast} and~\eqref{le not so fast gamma}, and our assumption $\varrho(\gamma_{w,0})\ge c>0$, we immediately~\eqref{le estimate on ddotgamma}.
\end{proof}
\begin{corollary}\label{corollary le dot gamma estimate}
Under the assumptions of Lemma~\ref{le lemma estimate on ddotgamma} the following holds true. Let $w,w_0\in[0,N]$ such that $\gamma_w(0),\gamma_{w_0}(0)\in[E_*-\Sigma/4,E_*+\Sigma/4]$. Then we have the estimates
\begin{align}
 |\dot\gamma^\regu_{w_0}(t)|\le C\,,\qquad\qquad |\ddot\gamma^\regu_{w_0}(t)|\le CN^\delta\,,
\end{align}
and
\begin{align}
 |\dot\gamma_{w}^\regu(t)-\dot\gamma_{w_0}^\regu(t)|\le CN^\delta\frac{|w-w_0|}{N}+C\regu\,,\qquad \qquad |\dot\gamma_{w_0}^\regu(t)-\dot\gamma_{w_0}^\regu(0)|\le  CN^\delta t\,,\label{le dot gammaw is almost zero}
\end{align}
uniformly in $0\le t\le CN^{-2\delta}$, with constants depending only on $\delta$, $\varrho$ and $E$ and $\Sigma$.
\end{corollary}
\begin{proof}
Since $\gamma_w(0),\gamma_{w_0}(0)\in[E_*-\Sigma/4,E_*+\Sigma/4]$, we have by Corollary~\ref{corollary le not so fast} that $\gamma_w(t),\gamma_{w_0}(t)\in\in[E_*-\Sigma/2,E_*+\Sigma/2]$, for $t\le CN^{-2\delta}$, in particular we have $\varrho_t(\gamma_w(t)),\varrho_t(\gamma_{w_0}(t))\ge c>0$ for such $t$.

Recalling the identity $\Hi{\varrho_{t_0}^\regu}(E)=\re m_{t_0}(E+\ii\regu)$ (as follows from ~\eqref{le regularization trick} and~\eqref{le def of hilbert transform}), and the estimates on $m_{t}$, $\partial_z m_{t}$ derived in Lemma~\ref{le lemma mfc}, we conclude from that from~\eqref{evolution of gammaw} and~\eqref{le spatial derivative of gammaw} that
\begin{align*}
|\dot\gamma_{w_0}(t)|\le C\,,\qquad\qquad| \dot\gamma_w^\regu(t)-\dot\gamma_{w_0}^\regu(t)|\le C N^{-1+\delta}{|w-w_0|}+C\regu\,,
\end{align*}
for all $0\le t\le CN^{-2\delta}$. We further get from~\eqref{le estimate on ddotgamma} that
\begin{align*}
 |\ddot\gamma_{w_0}^\regu(t)|\le CN^{\delta}\,,\qquad\qquad|\dot\gamma_{w_0}^\regu(t)-\dot\gamma_{w_0}^\regu(0)|\le CN^\delta t\,.
\end{align*}
This proves~\eqref{le dot gammaw is almost zero}.
\end{proof}

\subsubsection{Proof of Lemma~\ref{le density lemma} and of Lemma~\ref{le lemma for new gamma}}\label{le subsubsection proof of classical flow result}

\begin{proof}[Proof of Lemma~\ref{le density lemma}]
 Without lost of generality, we can assume that $t_1=0$. Fix $N\in\N$. From Lemma~\ref{le lemma mfc}, we directly get $|\varrho_t(x)-\varrho(t)|\ll CN^{\delta}t$. Thus for $t\le C'N^{-2\delta}$, we get the first estimate in~\eqref{assumption on regularity of varrho}. Next, we note that first and second derivative of $m_t(z)$, $z=E+\ii\eta$, are bounded on~$\mathcal{D}_{\Sigma/2}$ by~\eqref{le second bound in lemma mfc}. Thus the first derivative converges uniformly as $\eta\searrow 0$, $E\in[E_*-\Sigma/2,E_*+\Sigma/2]$, and we have $\pi\partial_E\varrho_t(E)=\lim_{\eta\searrow 0}\im m_t(E+\ii\eta)$. In particular, we obtain from~\eqref{le second bound in lemma mfc} the second estimate~in~\eqref{assumption on regularity of varrho}.
\end{proof}

\begin{proof}[Proof of Lemma~\ref{le lemma for new gamma}]
 Without lost of generality, we can assume that $t_1=0$ here. Let~$\regu>0$ as in~\eqref{le reguregu}. We then recall that we have
 $|\varrho_t^\regu(x)-\varrho_t|\le CN^\delta\regu$, for all $x\in [E_*-\Sigma,E_*+\Sigma]$ and $0\le t\le C'N^{-2\delta}$. This follows directly from the definition of the Poisson kernel in~\eqref{le poisson kernel} and Lemma~\ref{le lemma mfc}. Thus, using the definition of $\gamma_w^\regu(t)$ in~\eqref{definition of gammaw}, we must have $ |\gamma_i^\regu(t)-\gamma_i(t)|\le CN^\delta\eta_*$, $ i\in \II$. By Assumption $(1)$ of Theorem~\ref{main theorem}, $m_0(z)$ extends to a continuous function on $\mathcal{D}_{\Sigma}$. Thus reasoning as in the proof of Lemma~\ref{le lemma mfc}, we conclude that $m_t(z)$ extends to a continuous function on $\mathcal{D}_{\Sigma/2}$ for $0\le t\le C'N^{-2\delta}$. Hence, considering now $\regu$ as a free parameter (not depending on $N$) and taking $\regu\searrow 0$, we conclude from~\eqref{evolution of gammaw} that
\begin{align}\label{le to be taken the limit of}
 \lim_{\eta_*\searrow 0}\frac{\dd \gamma_i^\regu(t)}{\dd t}=-\int_\R\frac{\varrho_t(y)\dd y}{y-\gamma_{i}(t)}-\frac{\gamma_{i}(t)}{2}\,,\qquad\qquad i\in\II\,,
\end{align}
for all $i\in \II$ and $0\le t\le C'N^{-2\delta}$. Here, we also used that $\varrho_t^\regu(\gamma^\regu_i(t))>0$, $\varrho_t(\gamma_i(t))>0$. Further, since $\dot\gamma_i^\regu(t)$ converges uniformly to $\dot\gamma_i(t)$ for all $0\le t\le C'N^{-2\delta}$, we can also exchange derivative and limit on the left side of~\eqref{le to be taken the limit of} and we obtain~\eqref{le eom dotgamma}. In particular, we have $\lim_{\regu\searrow 0}\dot\gamma^\regu_i(t)=\dot\gamma_i(t)$, $i\in \II$, $0\le t\le C'N^{-2\delta}$. Thus~\eqref{le estimate on dotgamma} follows from~\eqref{le dot gammaw is almost zero}.

Now we show ~\eqref{le eom dotgamma 2}. For any fixed $t$ 
we  define the ``continuous'' quantiles
\begin{align}\label{halfquantile}
     \int_{-\infty}^{\gamma_u(t)}\varrho_t (x)\dd x = \frac{u}{N}\,, \qquad\qquad u\in[0,N]\,,
\end{align}
and also the ``half-quantiles'' $ \wh\gamma_i(t)\deq \gamma_{i-1/2}(t)$ for any integer $i$. Since $t$ is fixed throughout the proof, we drop the $t$ argument. From \eqref{halfquantile} we get
the regularity of the continuous quantiles:
\begin{align}\label{gamer}
  \frac{\dd \gamma_u}{\dd u} = \frac{1}{N\varrho(\gamma_u)} = O(N^{-1})\,,\qquad \qquad \frac{\dd^2 \gamma_u}{\dd u^2} = -\frac{\varrho'(\gamma_u)}{N^2\varrho(\gamma_u)^3} = O(N^{-2+\delta})\,,
\end{align}
in the bulk regime, where we used \eqref{assumption on regularity of varrho}.

Setting $j=\ell(i)$ for brevity, we can write
\begin{align}\label{3int}
   \int_\R\frac{\varrho_t(y)\dd y}{y-\gamma_{\rl(i)}} =\Bigg[
   \int_{-\infty}^{\wh\gamma_{j-\sigma N}} + \int_{\wh\gamma_{j-\sigma N}}^{\wh\gamma_{j+\sigma N+1}}
   + \int_{\wh\gamma_{j+\sigma N+1}}^\infty\Bigg] \frac{\varrho(y)\,\dd y}{y-\gamma_j}\,.
\end{align}
The first integral can be written as (with $\wh\gamma_0 =-\infty$ and using \eqref{halfquantile})
\begin{align}\label{mid0}
  \sum_{k=0}^{j-\sigma N-1}\int_{\wh\gamma_k}^{\wh \gamma_{k+1}}  \frac{\varrho(y)\dd y}{y-\gamma_{j}}
  =\frac{1}{N}  \sum_{k=1}^{j-\sigma N-1}\frac{1}{\gamma_k-\gamma_{j}} +
   \sum_{k=1}^{j-\sigma N-1} \int_{\wh\gamma_{k}}^{\wh\gamma_{k+1}} \frac{(\gamma_k-y)\varrho(y)\dd y}{(y-\gamma_{j})(\gamma_k-\gamma_{j})} + \int_{-\infty}^{\wh \gamma_1}  \frac{\varrho(y)\,\dd y}{y-\gamma_{j}}\,.
\end{align}
The last term is $O(N^{-1})$.
The error term in the middle is bounded by
$$
  \Bigg| \sum_{k=1}^{j-\sigma N-1} \int_{\wh\gamma_{k}}^{\wh\gamma_{k+1}} \frac{(\gamma_k-y)\varrho(y)\dd y}{(y-\gamma_{j})(\gamma_k-\gamma_{j})}
   \Bigg|\le \frac{C}{N} \sum_{k=1}^{j-\sigma N-1} \int_{\wh\gamma_{k}}^{\wh\gamma_{k+1}} \frac{\varrho(y)\dd y}{(\gamma_k-\gamma_j)^2} \le C\sum_{k=1}^{j-\sigma N-1} \frac{1}{(k-j)^2} \le CN^{-1}.
 $$
 The third integral in \eqref{3int} is estimated similarly.
 Finally, for the second integral we write 
 \begin{align}\label{middle}
    \int_{\wh\gamma_{j-\sigma N}}^{\wh\gamma_{j+\sigma N+1}}\frac{\varrho(y)\dd y}{y-\gamma_j}
    =\sum_{k:|k-j|=1}^{\sigma N} \left(\frac{1}{N(\gamma_k-\gamma_j)}
    + \int_{\wh\gamma_{k}}^{\wh\gamma_{k+1}} \frac{(\gamma_k-y)\varrho(y)\dd y}{(y-\gamma_{j})(\gamma_k-\gamma_{j})} \right)
    + \int_{\wh\gamma_{j}}^{\wh\gamma_{j+1}} \frac{\varrho(y)\dd y}{y-\gamma_j}\,.
 \end{align}
 In the last integral we Taylor expand $\varrho(y) = \varrho(\gamma_j) + O(N^\delta|\gamma_j -y|)$ by
 \eqref{assumption on regularity of varrho} to get
 $$
   \int_{\wh\gamma_{j}}^{\wh\gamma_{j+1}} \frac{\varrho(y)\dd y}{y-\gamma_j} 
   = \varrho(\gamma_j) \int_{\wh\gamma_{j}}^{\wh\gamma_{j+1}} \frac{\dd y}{y-\gamma_j} + O(N^{-1})\,.
$$
Computing the integral explicitly and using $\wh \gamma_j =\gamma_{j-1/2}$, $\wh\gamma_{j+1} =\gamma_{j+1/2}$ we have
$$
 \int_{\wh\gamma_{j}}^{\wh\gamma_{j+1}} \frac{\dd y}{y-\gamma_j} = \log \Big| 1 + \frac{\gamma_{j-1/2} - 2\gamma_j
  + \gamma_{j+1/2}}{\gamma_{j-1/2} - \gamma_j}\Big| = O(N^{-1+\delta})\,.
$$
Here we used  $\gamma_{j-1/2} - \gamma_j \ge c/N$ and \eqref{gamer}  to estimate the second order discrete derivative.

 Finally, we estimate the integral in the middle term in \eqref{middle}
 and we will show that
 \begin{align}\label{mid4}
  \sum_{k:|k-j|=1}^{\sigma N} 
   \int_{\wh\gamma_{k}}^{\wh\gamma_{k+1}} 
   \frac{(\gamma_k-y)\varrho(y)\dd y}{(y-\gamma_{j})(\gamma_k-\gamma_{j})}  = O(N^{-1+\delta}\log N)\,.
  \end{align}
 Clearly, \eqref{mid4} together with \eqref{3int}, \eqref{mid0}, \eqref{middle} and \eqref{mid2}
 imply~\eqref{le eom dotgamma 2}. 
 
 In the rest of the proof we show \eqref{mid4}. We write
 \begin{align}\label{mid2}
 \sum_{k:|k-j|=1}^{\sigma N} 
   \int_{\wh\gamma_{k}}^{\wh\gamma_{k+1}} & \frac{(\gamma_k-y)\varrho(y)\dd y}{(y-\gamma_{j})(\gamma_k-\gamma_{j})} \\
   & =\sum_{m=1}^{\sigma N}\left( \int_{\wh\gamma_{j-m}}^{\wh\gamma_{j-m+1}} \frac{(\gamma_{j-m}-y)\varrho(y)\dd y}{(y-\gamma_{j})(\gamma_{j-m}-\gamma_{j})}
   + \int_{\wh\gamma_{j+m}}^{\wh\gamma_{j+m+1}} \frac{(\gamma_{j+m}-y)\varrho(y)\dd y}{(y-\gamma_{j})(\gamma_{j+m}-\gamma_{j})}\right)\,. \nonumber
\end{align}

In the first integral, we replace $\varrho(y)$ with $\varrho(\gamma_{j})$. From Taylor expansion, $|\varrho(y)-\varrho(\gamma_{j})|\le CmN^{-1+\delta}$, the error
in this replacement is bounded by 
$$
   \frac{CmN^\delta}{N}\int_{\wh\gamma_{j-m}}^{\wh\gamma_{j-m+1}} \frac{|\gamma_{j-m}-y|\dd y}{|y-\gamma_{j}||\gamma_{j-m}-\gamma_{j}|}\le \frac{CmN^\delta}{N^2}\int_{\wh\gamma_{j-m}}^{\wh\gamma_{j-m+1}} \frac{\dd y}{|y-\gamma_{j}||\gamma_{j-m}-\gamma_{j}|}
\le \frac{CN^\delta}{Nm}\,,
$$
since $|\gamma_{j-m} -\gamma_j|\sim m/N$. We get a similar error when replacing $\varrho(y)$ with $\varrho(\gamma_{j})$
in the second integral in~\eqref{mid2}. These errors, even after summation over $m$, are still of order $O(N^{-1+\delta}\log N)$,
hence negligible. Thus we get that~\eqref{mid2} equals
\begin{multline*}
 \sum_{m=1}^{\sigma N}\left( \frac{\varrho(\gamma_j)}{\gamma_{j-m}-\gamma_{j}}
  \int_{\wh\gamma_{j-m}}^{\wh\gamma_{j-m+1}} \frac{\gamma_{j-m}-y}{y-\gamma_{j}} \dd y
   +  \frac{\varrho(\gamma_j)}{\gamma_{j+m}-\gamma_{j}}\int_{\wh\gamma_{j+m}}^{\wh\gamma_{j+m+1}} \frac{\gamma_{j+m}-y}{y-\gamma_{j}}\dd y\right)+O(N^{-1+\delta}\log N)\,.
 \end{multline*}
 Next, using~\eqref{gamer}, we have 
 $$
   \frac{1}{\gamma_{j-m}-\gamma_{j}} = -\frac{\varrho(\gamma_j)N}{m} + O(N^{\delta})\,, \qquad\qquad 
   \gamma_{j+m}-\gamma_{j}= \frac{m}{\varrho(\gamma_j)N} + O(N^{-2+\delta}{ m^2})\,,
$$
so we get, after a change of variables, that~\eqref{mid2} equals
\begin{align*}
  & \sum_{m=1}^{\sigma N} \frac{N}{m}\left(
-  \int_{\wh\gamma_{j-m}}^{\wh\gamma_{j-m+1}} \frac{\gamma_{j-m}-y}{y-\gamma_{j}} \dd y
   +  \int_{\wh\gamma_{j+m}}^{\wh\gamma_{j+m+1}} \frac{\gamma_{j+m}-y}{y-\gamma_{j}}\dd y\right) +O(N^{-1+\delta}\log N)
   \\
   = & -\sum_{m=1}^{\sigma N} \frac{N}{m}\left(
  \int_{\wh\gamma_{j-m}-\gamma_{j-m}}^{\wh\gamma_{j-m+1}-\gamma_{j-m}} \frac{u\,\dd u}{\gamma_j-\gamma_{j-m}-u} 
   +  \int_{\wh\gamma_{j+m}-\gamma_{j+m}}^{\wh\gamma_{j+m+1}-\gamma_{j+m}} \frac{u\,\dd u}{\gamma_{j+m}-\gamma_{j}+u}\right)+O(N^{-1+\delta}\log N)\,.
 \end{align*}
 The limits of integrations can be approximated as follows:
 $$
   \wh\gamma_{j-m}-\gamma_{j-m} = -\frac{1}{2N\varrho(\gamma_{j})} + O(mN^{-2+\delta})\,,\qquad \qquad
    \wh\gamma_{j+m}-\gamma_{j+m}  =
   -\frac{1}{2N\varrho(\gamma_{j})} + O(mN^{-2+\delta})\,,
   $$
   $$ \wh\gamma_{j-m+1}-\gamma_{j-m}= \frac{1}{2N\varrho(\gamma_{j})} + O(mN^{-2+\delta})\,,\qquad \qquad 
   \wh\gamma_{j+m+1}-\gamma_{j+m}= \frac{1}{2N\varrho(\gamma_{j})} + O(mN^{-2+\delta})\,.
 $$
 Replacing these limits with their common values yields negligible errors, for example:
 $$
   \sum_{m=1}^{\sigma N} \frac{N}{m}\int_{\frac{1}{2N\varrho(\gamma_{j})}}^{\wh\gamma_{j-m+1}-\gamma_{j-m}}
   \frac{u\,\dd u}{|\gamma_j-\gamma_{j-m}-u|} \le C\sum_{m=1}^{\sigma N} \frac{N}{m}\cdot \frac{mN^\delta}{N^2} \cdot \frac{1}{N}
    \frac{1}{m/N} = CN^{-1+\delta	}\log N\,.
 $$
 Thus, with the notation $d\deq 1/2N\varrho(\gamma_j)$, we get
 \begin{align}
  \eqref{mid2} 
   = & -\sum_{m=1}^{\sigma N} \frac{N}{m}\left(
  \int_{-d}^{d} \frac{u\,\dd u}{\gamma_j-\gamma_{j-m}-u} 
   +  \int_{-d}^{d} \frac{u\,\dd u}{\gamma_{j+m}-\gamma_{j}+u}\right) +O(N^{-1+\delta}\log N)\,. \label{mid3}
 \end{align} 
Using that $\gamma_j -\gamma_{j-m} = \gamma_{j+m}-\gamma_j + O(m^2N^{-2+\delta})$ from \eqref{gamer}, we can
write
$$
    \frac{1}{\gamma_{j+m}-\gamma_{j}+u} =  \frac{1}{\gamma_{j}-\gamma_{j-m}+u} + O(N^\delta)\,,
 $$
 for any $u$ in the second integration regime, since here
  $|u|\le \frac{1}{2N\varrho(\gamma_{j})}$ and $\gamma_{j+m}-\gamma_{j}\ge 
  \frac{m}{N\varrho(\gamma_{j}) }+ O(m^2N^{-2+\delta})$. Thus, replacing $\gamma_{j+m}-\gamma_{j}$ 
  with $\gamma_{j}-\gamma_{j-m}$ in the denominator of the second integral in \eqref{mid3}
  yields an error of order $\sum_m (N/m) N^{-2+\delta} = CN^{1+\delta}\log N$.
 After this replacement the two integrals in \eqref{mid3} cancel out:
 $$
   \int_{-d}^{d} \frac{u\,\dd u}{\gamma_j-\gamma_{j-m}-u} 
   +  \int_{-d}^{d} \frac{u\,\dd u}{\gamma_{j}-\gamma_{j-m}+u} = \int_{-d}^{d} \frac{u\,\dd u}{(\gamma_j-\gamma_{j-m})^2-u^2}=0\,. 
 $$
 We have shown that \eqref{mid2} is of order $O(N^{-1+\delta}\log N)$, which finishes the proof of \eqref{mid4}. 
\end{proof}

\subsection{Remarks on Assumption~$(1)$ of Theorem~\ref{main theorem}}\label{appendix hoelder}
We conclude this Appendix with some remarks on Assumption~$(1)$ of Theorem~\ref{main theorem}. We consider the semicircular flow $\varrho_t=\mathcal{F}_t(\varrho)$ started from~$\varrho$, for $t\ge 0$. As remarked earlier, the semicircle law~$\varrho_{sc}$ is invariant under the flow. It is then easy to check that $m_{sc}$, the Stieltjes transform of $\varrho_{sc}$ satisfies the bound in~\eqref{le bounds in assumption 1} for all $t$ with $\delta=1$. 

For many matrix models the distribution $\varrho$, and hence also $\varrho_t$, are not explicit and checking Assumption~$(1)$ directly may be not an easy task. In many situations, one can however use the smoothing effect of the semicircle flow to establish these estimates. The following example may be of some interest.

Denote by $C^{0,\alpha}(\R)$, $C^{0,\alpha}(\C^+)$ the spaces of uniformly $\alpha$-H\"older continuous functions on $\R$, $\C^+$. Assume that $\varrho\in\C^{0,\alpha}(\R)$, for some $\alpha>0$. Then the Stieltjes transform, $m_0$, of $\varrho$ is in $\C^{0,\alpha}(\C^+)$. 
Adapting the proof of Lemma~\ref{le lemma mfc} one can establish the following result. Abbreviate $\sigma_t=1-\euler{-t}$. 

\begin{lemma}\label{le mfc hoelder}
Assume that $\varrho\in\C^{0,\alpha}(\R)$. Then, $m_t(z)$, the Stieltjes transform of $\varrho_t=\mathcal{F}_t(\varrho)$, is uniformly bounded for all $z\in\C^+\cup\R$ and $t\ge 0$. Moreover, there is a constant $C$, depending only on~$\varrho$, such that
\begin{align}
 |m_t(z)-m_0(z)|\le C\sigma_t^\alpha\,,\qquad\qquad 0\le t\le 1\,,
\end{align}
and all $z\in\C^+\cup\R$. Further, for all $n\in\N$, there is $C_n$ such that we have the bounds
\begin{align}
 |\partial_z^n m_t(z)|\le C_n(\sigma_t\im m_t(z))^{\alpha-n}\,,\qquad\qquad t>0\,,
\end{align}
for all $z\in\C^+\cup\R$.
\end{lemma}
Thus, running the semicircular flow from time $t=0$ to time $t_1=N^{-\tau_1}$, $\tau_1>0$, we see that Lemma~\ref{le mfc hoelder} implies the Assumption~$(1)$ of Theorem~\ref{main theorem} for energies inside the ``bulk'' for the choice $\delta\ge (1-\alpha)\tau_1$. For the Wigner-like matrices of~\cite{AEK1,AEK2} typical choices for $\alpha$ are $1/3$ or $1/2$.

\section{Persistent trailing of the DBM}\label{appendix:persist} 
In this section, we prove that the time-dependent quantiles~$\gamma_k(t)$
persistently trail the DBM  up to a time-independent shift in the indices. 
More precisely, we have the following
\begin{proposition}\label{prop:persist}  Consider a time interval $[t_1, t_2]$ of length $t_2-t_1=O(N^{-\epsilon})$
with some small $\epsilon>2\delta$, where $\delta$, given in \eqref{le bounds in assumption 1}, is the regularity exponent
of the initial data of the quantiles.
Let $\bla(t)$ be the solution of~\eqref{le original dbm} and let $\boldsymbol{\gamma}(t)$ be given by~\eqref{the gammas}.
Suppose that
\begin{align}\label{riggi}
   \P \Big\{ |\lambda_i(t) - \lambda_j(t)|\le \frac{N^{\xi}|i-j|}{N}, \quad i,j \in \II \Big\}\le N^{-D}\,,
\end{align}
for any $D$.
Fix an index $L$ in the bulk and let $\ell(L)$ such that 
\begin{align}\label{inL}
     |\lambda_L(t_1) -\gamma_{\ell(L)}(t_1)|\le CN^{-1+\xi}.
\end{align}
Then in the probability space of the Brownian motions   $\{ B_i(t)\; : \; i\in \N_N\,,t\in [t_1,t_2]\}$
we have
\begin{align}\label{eq:per}
   \P\bigg(   \sup_{t\in [t_1, t_2]} |\lambda_L(t) -\gamma_{\ell(L)}(t)|\le CN^{-1+2\xi}  
   \bigg)\ge 1- N^{-\xi}\,.
\end{align}
\end{proposition}

Notice that $\gamma_{\ell(L)}(t)$ is a deterministic trajectory. 
This result therefore also shows that the typical fluctuation of the solution of the DBM
is much smaller than the white noise term in~\eqref{le original dbm} naively indicates. Indeed, the variance of
the integral of this term is
$$ 
   \E\left| \int_{t_1}^{t_2} \sqrt{\frac{2}{\beta N}} \dd B_L\right|^2 \simeq \frac{t_2-t_1}{N}\, , 
$$
which would indicate a behavior $|\lambda_L(t_2)-\lambda_L(t_1)|\gtrsim (t_2-t_1)^{1/2} N^{-1/2}$.
This is much larger than the actual value  $|\lambda_L(t_2)-\lambda_L(t_1)|\le C N^{-1+\xi}$.

\begin{proof}
Let 
$$
  v_i (t) \deq \lambda_i(t) - \gamma_{\ell(i)} (t)\,.
$$
Subtracting the DBM~\eqref{le original dbm} from~\eqref{le eom dotgamma 2}
and localizing it for the indices $i\in I$, we get
$$
    \dd v_i(t) = \sqrt{\frac{2}{\beta N}}\,\dd B_i(t) - \sum_{j\in I} \mathcal{B}_{ij} (v_i- v_j)\dd t - \mathcal{W}_i v_i\, \dd t + \kappa_i(t)
    \,\dd t\,, \qquad i\in I\, ,
 $$
with (time-dependent) coefficients
$$
  \mathcal{B}_{ij} = \frac{1}{N}  \frac{1}{(\lambda_j-\lambda_i)(\gamma_j -\gamma_i)}, \quad i,j \in I\, ,
   \qquad \qquad\mathcal{W}_i =\frac{1}{2}+ \sum_{k\not\in I}
    \frac{1}{N}  \frac{1}{(\lambda_k-\lambda_i)(\gamma_k -\gamma_i)}, \quad i\in I\, ,
$$
and a deterministic error term $|\kappa_i (t)|\le N^{-1+\delta}$.

By \eqref{riggi} and the spacing of the quantiles in the bulk, we know that
\begin{align}\label{cond}
   \mathcal{B}_{ij}(s) \ge \frac{b}{|i-j|}\,,\qquad \qquad \mathcal{W}_i(s) \ge \frac{b}{ \big| |j-L| - K\big| +1}\,,
\end{align}
with $b\deq N^{1-\xi}$ uniformly in time $s\in [t_1,t_2]$ with very high probability.

Let $\cU(s,t )$ be the propagator of the parabolic equation
\begin{align}\label{pareq}
   \frac{\dd u_i(t)}{\dd t} =- \sum_{j\in I} \mathcal{B}_{ij} (u_i- u_j) -\mathcal{ W}_i u_i\,, \qquad\qquad i\in I\,,
\end{align}
 then
\begin{align}\label{veq}
   v_i (t) = v_i(t_1) +  \int_{t_1}^t \sum_j [\cU(s,t)]_{ij} \left[\sqrt{\frac{2}{\beta N}}\,\dd B_j(s) + \kappa_j(s)\dd s\right]  \,.
\end{align}
Since the propagator is a contraction in the supremum norm, the $\kappa(s)$ error term, after integration,
gives a negligible error at most $C(t_2-t_1) N^{-1+\delta} \le CN^{-1}$.
To estimate the main term, notice that the propagator depends on the sigma algebra $\Sigma_s:=\{ \{ B_i(u)\}_{i\in I}\; : \; 
u\in (s,t]\}$ and is independent of the white noise at time $s$. Therefore
\begin{align}\label{veq2}
 \E | v_i(t)- v_i(t_1)|^2 = \frac{2}{\beta N} \E \int_{t_1}^t \sum_j |[\cU(s,t)]_{ij}|^2 \dd s +  O(N^{-2})\,.
\end{align}
Fix $i\in I$, $s$ and $t$  and define
$w_j \deq [\cU(s,t)]_{ij}$, 
which is the same as $[\cU(s,t)]_{ji}$ by symmetry.
Then for any $\nu>0$, we have
$$
\Big| \sum_j |[\cU(s,t)]_{ij}|^2 \Big| \le \|  \cU(s,t) \bw \|_\infty \le \frac{C N^\xi}{[N(t-s)]^{\frac{1}{1+\nu}}} 
\| \bw\|_{1+\nu}\,,
$$
by the heat kernel estimate on the Equation~\eqref{pareq};
see Proposition~9.4  in~\cite{EYSG} (the conditions of this proposition
are guaranteed by~\eqref{cond}). By the~$L^p$-contraction of the semigroup for any~$p\ge 1$, we have
$$
   \| \bw\|_{1+\nu} = \| \cU(s,t) \delta_i\|_{1+\nu} \le \|\delta_i\|_{1+\nu}=1\,.
$$
Thus we get
\begin{align}\label{veq4}
\E | v_i(t)- v_i(t_1)|^2 \le \frac{CN^\xi}{N^{1+\frac{1}{1+\nu}}}  \int_{t_1}^t \frac{1}{(t-s)^{\frac{1}{1+\nu}}} \dd s  
+ O(N^{-2})\le C_\nu\frac{N^{\xi+2\nu}}{N^2} \le C\frac{N^{2\xi}}{N^2}\,,
\end{align}
after choosing $\nu=\xi/2$.
Using Doob martingale inequality, one also has
$$
   \E \sup_{t\le t_2} | v_i(t)- v_i(t_1)|^2 \le C\E | v_i(t_2)- v_i(t_1)|^2\le C\frac{N^{2\xi}}{N^2}\,.
$$
Setting $i=L$ and using Markov inequality and combining it with assumption~\eqref{inL},
 we get~\eqref{eq:per}.
\end{proof}

\end{appendix}

\small{

}

\end{document}